\documentclass[reqno, 11pt]{amsart}% {{{
\usepackage[margin = 1in]{geometry}
\usepackage{amssymb}
\usepackage{amsmath}
\usepackage{amsthm}
\usepackage{amsfonts}
\usepackage{mathrsfs}
\usepackage{graphicx}
\usepackage{graphics}
\usepackage[table]{xcolor}
\usepackage{latexsym}
\usepackage{comment}
\usepackage{titlesec}
\usepackage{mathtools}
\usepackage{tikz}
\usepackage{pgfplots}
\usepackage{pifont}
\usepackage{xfrac}
\usepackage{multirow}
\usepackage{caption}
\usepackage{subcaption}
\usetikzlibrary{positioning}
\usetikzlibrary{hobby}
\usetikzlibrary{intersections}

\excludecomment{codes}
\usepackage{titletoc}
\dottedcontents{section}[1.5em]{}{1.2em}{1pc}
\dottedcontents{subsection}[3.5em]{}{2em}{1pc}

\usepackage{url}
\usepackage[colorlinks = true,
            linkcolor = blue,
            filecolor = magenta,
            urlcolor = blue,
            citecolor = magenta,
            pdftitle = {SPDEs on a domain},
           ]{hyperref}

\usepackage[shortlabels]{enumitem}
\setlist[itemize]{align = parleft,left = 4pt..1.5em}
\setlist[enumerate]{align = parleft,left = 4pt..2.1em}

\usepackage[strict=true,style=english]{csquotes}
\usepackage[english]{babel}
\usepackage[backend = biber,
            style = numeric,
            natbib = true,
            sorting = nyt,
            abbreviate = true,
           ]{biblatex}

\AtEveryBibitem{
 \clearfield{url}
 \clearfield{issn}
 \clearlist{location}
 \clearfield{month}
 \clearfield{series}

  \ifentrytype{book}{}{
    \clearlist{publisher}
    \clearname{editor}
 }
}

\titleformat{\section} {\bf}{\thesection.}{1em}{} \titleformat{\subsection}
{\bf}{\thesubsection}{1em}{}

\numberwithin{equation}{section}
\numberwithin{table}{section}

\newtheorem{theorem}{Theorem}[section]
\newtheorem{lemma}[theorem]{Lemma}
\newtheorem{proposition}[theorem]{Proposition}
\newtheorem{corollary}[theorem]{Corollary}

\newtheorem{assumption}[theorem]{Assumption}

\theoremstyle{definition}
\newtheorem{definition}[theorem]{Definition}
\newtheorem{example}[theorem]{Example}

\newtheorem{remark}[theorem]{Remark}
\newtheorem{conjecture}[theorem]{Conjecture}
\pgfplotsset{compat=1.18}

\newcommand{\cmark}{\ding{51}}%
\newcommand{\xmark}{\ding{55}}%

\newcommand{\op}{\operatorname}
\newcommand{\dom}{U}
\newcommand{\Norm}[1]{\left|\left| #1 \right|\right|}
\newcommand{\Lip}{L}
\newcommand{\lip}{l}
\newcommand{\ud}{\ensuremath{\mathrm{d} }}

\newcommand\R{\mathbb{R}}

\newcommand\eps{\varepsilon}

\newcommand\E{\mathbb{E}}
\renewcommand\P{\mathbb{P}}

\begin{document}
% {{{ Title, abstract
\title[Parabolic SPDEs on bounded domains]
{Parabolic stochastic PDEs on bounded domains with rough initial conditions: moment and correlation bounds}
\author{David Candil}
\address{Institut de math\'ematiques, \'Ecole Polytechnique F\'ed\'erale de Lausanne, Station 8, CH-1015 Lausanne, Switzerland}
\email{david.candil@epfl.ch}
\author{Le Chen}
\address{Department of Mathematics and Statistics, Auburn University, 203 Parker Hall, Auburn, Alabama 36849, United States}
\email{le.chen@auburn.edu}
\author{Cheuk Yin Lee}
\address{Department of Mathematics, National Tsing Hua University, No.~101, Section 2, Kuang-Fu Road, Hsinchu 300044, Taiwan}
\email{cylee@math.nthu.edu.tw}

\begin{abstract}
  We consider nonlinear parabolic stochastic PDEs on a bounded Lipschitz domain
  driven by a Gaussian noise that is white in time and colored in space, with
  Dirichlet or Neumann boundary condition. We establish existence, uniqueness
  and moment bounds of the random field solution under measure-valued initial
  data $\nu$. We also study the two-point correlation function of the solution
  and obtain explicit upper and lower bounds. For $C^{1, \alpha}$-domains with
  Dirichlet condition, the initial data $\nu$ is not required to be a finite
  measure and the moment bounds can be improved under the weaker condition that
  the leading eigenfunction of the differential operator is integrable with
  respect to $|\nu|$. As an application, we show that the solution is fully
  intermittent for sufficiently high level $\lambda$ of noise under the
  Dirichlet condition, and for all $\lambda > 0$ under the Neumann condition.
\end{abstract}
\keywords{Parabolic Anderson model; stochastic heat equation; Dirichlet/Neumann
boundary conditions; Lipschitz domain; intermittency; two-point correlation;
rough initial conditions.}

\subjclass{60H15, 35R60}

\maketitle
\tableofcontents

% }}}
\section{Introduction and main results}

In this paper, we study nonlinear parabolic \textit{stochastic partial differential
equations} (SPDEs) on a bounded Lipschitz domain $\dom$ in $\R^d$. By a domain
we refer to a connected open subset of $\R^d$.
Consider a second-order differential operator
\begin{align}\label{E:op-L}
  \mathscr{L} = -\sum_{i,j = 1}^d \frac{\partial}{\partial x_i}\Big(a_{ij}(x) \frac{\partial}{\partial x_j}\Big),
\end{align}
where $[a_{ij}(x)]_{i,j}$ is a
real-valued symmetric matrix that is H\"older continuous on $\dom$ with some
exponent $0 < \gamma \le 1$ and uniformly elliptic, i.e., there exists a
positive finite constant $C$ such that
\begin{align}\label{E:u-ellip}
      C^{-1} |\xi|^2
  \le \sum_{i = 1}^d \sum_{j = 1}^d a_{ij}(x) \xi_i \xi_j
  \le C |\xi|^2 \qquad \text{for all $x \in \dom$ and $\xi \in \R^d$,}
\end{align}
where $|\xi|\coloneqq \sqrt{\xi_1^2+\dots+\xi_d^2}$. We consider
operators of the form \eqref{E:op-L} because our approach in this paper is based
on heat kernel estimates for operators in divergence form (see Section
\ref{SS:Kernel} below). We consider the following SPDE with (vanishing)
Dirichlet boundary condition:
\begin{equation}\label{E:she-d}
  \begin{dcases}
    \frac{\partial}{\partial t} u(t, x) + \mathscr{L} u(t, x) = \lambda\: \sigma\left(t, x, u(t, x)\right)\: \dot W(t, x),
                              & t > 0,\: x \in \dom, \\
    u(0, \cdot) = \nu(\cdot), & x \in \dom,          \\
    u(t, x) = 0,              & t > 0,\: x \in \partial \dom,
  \end{dcases}
\end{equation}
as well as the same equation with (vanishing) Neumann boundary condition:
\begin{equation}\label{E:she-n}
  \begin{dcases}
    \frac{\partial}{\partial t} u(t, x) + \mathscr{L} u(t, x) = \lambda\: \sigma\left(t, x, u(t, x)\right)\: \dot W(t, x),
                                                     & t > 0,\: x \in \dom, \\
    u(0, \cdot) = \nu(\cdot),                        & x \in \dom,          \\
    \frac{\partial}{\partial \textup{n}}u(t, x) = 0, & t > 0,\: x \in \partial \dom,
  \end{dcases}
\end{equation}
where $\textup{n}$ is the outward normal to the boundary $\partial \dom$ of
$\dom$.

\medskip

We make the following assumption on the noise and correlation function:

\begin{assumption}\label{A:f}
  The noise $\dot W$ is a centered and spatially homogeneous Gaussian noise that
  is white in time with the covariance given by
  \begin{align}\label{E:CovNoise}
    \E\left[\dot W(t, x) \dot W(s, y)\right] = \delta(t-s) f(x-y),
  \end{align}
  where $\delta$ is the delta function and $f$ is a nonnegative and nonnegative
  definite function on $\R^d$. We assume that there exist constants $0 < C_f <
  \infty$ and $0 < \beta < 2 \wedge d$ such that
  \begin{equation}\label{E:f-bd}
    C_f^{-1} |x-y|^{-\beta} \le f(x-y) \le C_f |x-y|^{-\beta}
    \qquad \text{for all $x,y\in \dom$}.
  \end{equation}
\end{assumption}

For example, $f$ may be taken as the Riesz kernel $f(x-y) = |x-y|^{-\beta}$. In
both equations~\eqref{E:she-d} and~\eqref{E:she-n}, $\lambda>0$ is a constant
parameter representing the level or intensity of the noise.

\medskip

We need some regularity and cone conditions on the diffusion coefficient
$\sigma$,  which is given by the following assumption:
\begin{assumption}\label{A:sigma}
  We assume that $\sigma: (0, \infty) \times \dom \times \R \to \R$ in both
  equations~\eqref{E:she-d} and~\eqref{E:she-n} is a non-random function such
  that $\sigma(t, x, 0) = 0$ for all $(t,x)\in (t,\infty) \times \dom$ and there
  exists a constant $\Lip_\sigma>0$ such that
  \begin{equation}\label{E:Lipsigma}
    |\sigma(t, x, u) - \sigma(t, x, v)| \le \Lip_\sigma |u - v| \qquad \text{for all $t > 0$, $x \in \dom$ and $u, v \in \R$}.
  \end{equation}
\end{assumption}

In particular, Assumption~\ref{A:sigma} implies that
\begin{equation}\label{E:Lsigma}
  |\sigma(t, x, u)| \le \Lip_\sigma |u| \qquad \text{for all $t > 0$, $x \in \dom$ and $u, v \in \R$}.
\end{equation}
Besides, we will need the other side of condition~\eqref{E:Lsigma} in order to
derive some lower bounds later: there exists a constant $\lip_\sigma>0$ such
that
\begin{equation}\label{E:lsigma}
  \sigma(t, x, u) \ge \lip_\sigma |u| \qquad \text{for all $t > 0$, $x \in \dom$ and $u \in \R$}.
\end{equation}
Our results will also cover the important case --- the \textit{parabolic Anderson
model} (PAM)~\cite{carmona.molchanov:94:parabolic}:
\begin{equation}\label{E:Anderson}
  \sigma(t, x, u) = u \qquad \text{for all $t > 0$, $x \in \dom$ and $u \in \R$}.
\end{equation}
In this case, our results hold with $\Lip_\sigma = \lip_\sigma = 1$.

\medskip

We assume that the initial condition $\nu$ is a non-random, locally finite,
signed Borel measure on $\dom$. Denote $|\nu| \coloneqq \nu_+ + \nu_-$, where
$\mu = \mu_+ - \mu_-$ is the corresponding \textit{Jordan decomposition} of
$\mu$ with $\mu_\pm$ being two nonnegative Borel measures with disjoint support.
The exact blow-up rate of the locally finite measure near the boundary will be
controlled via an integrability condition by the leading eigenfunction;
see~\eqref{E:RoughData}. Initial conditions of this type will be called
\textit{rough initial conditions}. An important example is the \textit{Dirac
delta measure}, which plays an important role in the studying the long-time
asymptotics of the solution; see, e.g.,~\cite{amir.corwin.ea:11:probability}
and~\cite{corwin:12:kardar-parisi-zhang}. \medskip

For stochastic heat equations on $\mathbb{R}^d$, the probabilistic moment bounds
and the two-point correlation function, both under rough initial conditions,
have been studied in~\cite{chen.dalang:15:moments*1, chen.dalang:15:moments,
chen.hu.ea:17:two-point, chen.hu.ea:19:nonlinear, chen.kim:19:nonlinear}. As for
bounded domains, Foondun and Nualart~\cite{foondun.nualart:15:on} considered the
stochastic heat equation on an interval $(0, L)$ with space-time white noise and
either Dirichlet or Neumann boundary condition, and studied the moments and
intermittency properties of the solutions. Nualart~\cite{nualart:18:moment} and
Guerngar and Nane~\cite{guerngar.nane:20:moment} extended the results
in~\cite{foondun.nualart:15:on} to fractional stochastic heat equations with
colored noise, but only to the case when the domain is the unit ball in $\R^d$
plus a Dirichlet boundary condition. In all these
works~\cite{foondun.nualart:15:on, nualart:18:moment, guerngar.nane:20:moment},
the initial conditions are assumed to be a bounded function. Important initial
data, such as the Dirac delta measure, have not been properly studied. \medskip

\begin{figure}[htpb]
  \centering

  \renewcommand{\thesubfigure}{\thefigure.\arabic{subfigure}}
  \makeatletter
    \renewcommand{\p@subfigure}{}
  \makeatother
  \captionsetup[subfigure]{style = default, margin = 0pt, parskip = 0pt, hangindent = 0pt, labelformat = simple}
  % \captionsetup[subfigure]{style = default, labelformat = simple, labelsep = period}

  \subfloat[][\cmark]{\label{SF:Circ}
    \begin{tikzpicture}[ x = 2em, y = 2em, scale = 0.8]
      \draw[fill = gray!50!white, thick] (0,0) circle (1.5);
    \end{tikzpicture}
  }
  \quad
  \subfloat[][\cmark]{\label{SF:Square}
    \begin{tikzpicture}[ x = 1em, y = 1em, scale = 1.1]
      \draw[fill = gray!50!white, thick] (-2,-2) rectangle (2,2);
    \end{tikzpicture}
  }
  \quad
  \subfloat[][\cmark]{\label{SF:Star}
    \begin{tikzpicture}[ x = 1em, y = 1em, scale = 1.2]
      \draw[fill = gray!50!white, thick] (-2,+1) --++ (1,0) --++ (0,1) --++ (2,0) --++ (0,-1) --++ (1,0) --++ (0,-2) --++ (-1,0) --++ (0,-1)  --++ (-2,0)  --++ (0,1) --++ (-1,0) --++ (0,2);
    \end{tikzpicture}
  }
  \quad
  \subfloat[][\cmark]{\label{SF:NonConvex}
    \begin{tikzpicture}[ x = 2em, y = 2em, scale = 0.9]
      \draw[use Hobby shortcut, fill = gray!50!white, thick] ([closed]-1,0) .. (0,1.5)  ..  (2,0.1)  ..  (0,-0.5)  ..  (-0.6,-1);
    \end{tikzpicture}
  }
  \quad
  \subfloat[][\xmark]{\label{SF:Cups}
    \begin{tikzpicture}[ x = 2em, y = 2em, scale=1.4]
      % \draw[xshift=12cm, use Hobby shortcut, fill = gray!50!white] ([out angle = 0, in angle = 0]-2,0) .. (0,2.5)  ..  (2,0)  ..  (0,-1.5)  .. (-2,0);
      \filldraw[
          draw=black,
          fill=gray!50!white,
          thick,
          samples=100,
        ]
          plot[domain=0:1] (\x, {sqrt(\x)})
          -- (0,2) -- (-1,1) --
          plot[domain=-1:0] (\x, {sqrt(-\x)})
          --
          cycle;
    \end{tikzpicture}
  }

  \caption{Various bounded domains on $\R^2$: Fig.~\ref{SF:Circ} --
  Fig.~\ref{SF:NonConvex} are Lipschitz domains (either convex or not);
Fig.~\ref{SF:Cups} is a typical example of the non-Lipschitz domain where there
is a cusp.}

  \label{F:Domains}
\end{figure}

One of the main objectives/contributions of this paper is to study the moments
and correlation function of the solution of parabolic SPDEs~\eqref{E:she-d}
and~\eqref{E:she-n} under rough initial conditions with a uniformly elliptic
operator $\mathscr{L}$ on a bounded domain $\dom \subset \R^d$. For the
parabolic Anderson model (i.e., $\sigma(t, x, u) = u$) on $\R^d$, Chen and
Kim~\cite{chen.kim:19:nonlinear} have shown that the two-point correlation
function can be expressed as
\begin{equation}\label{E:corr}
  \E[u(t, x) u(t, x)] = \lambda^{-2} \iint_{\R^d \times \R^d}
  \nu(\ud y) \, \nu(\ud y') \, \mathcal{K}(t, x-y, x'-y', y'-y)
\end{equation}
for some kernel function $\mathcal K$; see also~\cite{chen.hu.ea:17:two-point}
for the space-time white noise case with $d=1$. Under the
conditions~\eqref{E:Lsigma} and~\eqref{E:lsigma} for $\sigma$, the correlation
function also admits upper and lower bounds of the same form as the right-hand
side of~\eqref{E:corr}. By establishing sharp upper and lower bounds for the
kernel $\mathcal K$, one can then obtain sharp bounds for the correlation
function. The formula~\eqref{E:corr} is established
in~\cite{chen.kim:19:nonlinear} by using a convolution-type operator
$\triangleright$. It is natural to ask if one can obtain similar formula and
bounds for the correlation function in the case of bounded domains. The
convolution-type operator $\triangleright$ on a bounded domain $\dom \subset
\R^d$ has been considered by Candil in~\cite{candil:22:localization}, where this
operator is used to study the localization error between the solution of the
stochastic heat equation on $\dom$ and the solution of the same equation on
$\R^d$.

In~\cite{nualart:18:moment}, it is mentioned that the extension of the moment
estimates from the unit ball---a smooth, convex and bounded domain---to other
bounded domains is not straightforward. One may expect that some geometric and
regularity conditions on the domain would be required. The majority of our
results below work for a general Lipschitz domain (see examples in
Fig.~\ref{SF:Circ} -- Fig.~\ref{SF:NonConvex}). Convexity of the domain will be
only required for the lower bounds of the moments in case of the Neumann
boundary conditions. Domains with more regularity than the Lipschitz condition
on the domain, such as the $C^{1,\alpha}$-domain ($\alpha>0$), will allow us to
obtain sharper upper moment bounds under Dirichlet boundary condition. See
Table~\ref{Tb:Summary} below for a summary of our results.

For the Neumann boundary condition, there have not been many results except
those in~\cite{foondun.joseph:14:remarks, foondun.nualart:15:on,
khoshnevisan.kim:15:non-linear}, which are concerned with the stochastic heat
equation driven by the space-time white noise on an interval $(0, L)$ in one
spatial dimension. In the Neumann case, they prove that full intermittency
occurs for all levels $\lambda > 0$ of noise. This suggests the formation of
tall peaks for the solutions even when the noise level is small. However, the
precise intermittency behavior has not been well studied for general domains
under the Neumann boundary condition. \medskip

Another main contribution of the paper is about the weak conditions, namely, the
rough initial conditions, that we impose on the initial data $\nu$. When $\dom =
\R^d$, i.e., the boundary is at $|x|\to\infty$, and in case of $\mathcal{L} =
-\frac{1}{2}\Delta$, the rough initial condition refers to locally finite
(signed) measure on $\R^d$ that satisfies the following integrability condition:
\begin{align}\label{E:J0Finite}
  \int_{\R^d} e^{- a |x|^2} |\nu|(\ud x) <\infty, \quad \text{for all $a>0$,}
\end{align}
which is equivalent to the solution to the homogeneous heat equation exists for
all time:
\begin{align}\label{E:J_0R^d}
  (p_t*|\nu|)(x) \coloneqq \int_{\R^d} (2\pi t)^{-d/2} e^{- \frac{|x-y|^2}{2t}}
   | \nu | (\ud y) <\infty, \quad \text{for all $t>0$ and $x\in\R^d$}.
\end{align}
Initial conditions of this type were studied
in~\cite{chen.dalang:15:moments,chen.huang:19:comparison,chen.kim:17:on}; see
also~\cite{conus.joseph.ea:14:initial}. When the domain $\dom$ is bounded, we
will show that the integrability condition~\eqref{E:J0Finite} should be replaced
by
\begin{align}\label{E:RoughData}
  \int_\dom \Phi_1(y)\, |\nu|(\ud y) < \infty,
\end{align}
where $\Phi_1(\cdot)$ is the eigenfunction corresponding to the leading
eigenvalue of the operator $\mathcal{L}$. In particular, for the Neumann
boundary condition case (see Theorems~\ref{T:ExUnque} and~\ref{T:corr-bd}),
since $\Phi_1(x)$ is a constant function that does not vanish at the boundary,
condition~\eqref{E:RoughData} is equivalent to $|\mu|(\dom)<\infty$, i.e.,
$|\mu|$ is a finite measure on the domain $\dom$. In case of the Dirichlet
boundary condition (see Theorem~\ref{T:C1alpha} and Corollary~\ref{C:prod-dom}),
condition~\eqref{E:RoughData} allows locally finite measure with certain growth
rate near the boundary. For specific domains, condition~\eqref{E:RoughData} can
be made more explicit; see examples in Table~\ref{Tb:Examples}.

\begin{table}[htpb]
  \centering
  \caption{Simplification of Condition~\eqref{E:RoughData} for specific domains
  in case of Dirichlet boundary conditions and $\mathcal{L} = \Delta$.}
  \label{Tb:Examples}
  \renewcommand{\arraystretch}{2.2}
  \begin{tabular}{|c|c|c|} \hline
    % \rowcolor{black!50!white}
    \rowcolor{lightgray}
    Domain $\dom$                    & $\displaystyle \int_\dom \Phi_1(y)\, |\nu|(\ud y) < \infty$                                         & Ref.                      \\ \hline
    Interval: $(0,1)$                & $\displaystyle \int_0^1 x(1-x)\, |\nu|(\ud x) <\infty$                                              & Example~\ref{Eg:interval} \\
    Ball in $\R^d$: $|x|<1$          & $\displaystyle \int_{|x|<1} \left(1-|x|\right)\, |\nu|(\ud x) <\infty$                              & Example~\ref{Eg:ball}     \\
    Annulus in $\R^2$: $R_1<|x|<R_2$ & $\displaystyle \int_{R_1<|x|<R_2} \left(R_2-|x|\right) \left(|x|-R_1\right)\, |\nu|(\ud x) <\infty$ & Example~\ref{Eg:Ann}      \\
    Box: $(0,1)^d$                   & $\displaystyle \int_{(0,1)^d} \bigg(\prod_{i = 1}^{d} x_i(1-x_i)\bigg)\, |\nu|(\ud x) <\infty$      & Example~\ref{Eg:Rect}     \\ \hline
  \end{tabular}
\end{table}

\medskip In order to obtain precise moment results and allow rough initial
conditions, we start by considering in Lemmas~\ref{Lem:iint-d}
and~\ref{Lem:iint-n} the following heat kernel integral:
\begin{align*}
  \iint_{\dom^2} G(t, x, y) G(t, x', y') f(y-y')\, \ud y\, \ud y'.
\end{align*}
By using the heat kernel estimates in Proposition~\ref{P:G}, we find that the
sharp bound for this integral is $e^{-2\mu_1 t}(1\wedge t)^{-\beta/2}$. In
particular, the factor of $(1\wedge t)^{-\beta/2}$ improves the estimates
in~\cite{nualart:18:moment}. Furthermore, we consider the convolution-type
integral of the heat kernel
\begin{align*}
  \iint_{\dom^2} G(t-s, x, z) G(t-s, x', z') f(z-z') G(s, z, y) G(s, z', y') \,\ud z\, \ud z'
\end{align*}
and obtain optimal bounds with a similar factor of $\big(1\wedge
\frac{(t-s)s}{t}\big)^{-\beta/2}$. Also, based on these optimal bounds and the
convolution-type operator $\triangleright$ as considered
in~\cite{candil:22:localization}, we extend the two-point correlation
formula~\eqref{E:corr} and related bounds in Proposition~\ref{P:2p-corr}, and
establish explicit upper and lower bounds for the kernel function in
Propositions~\ref{P:K-d} and~\ref{Lem:K-n}. In Theorem~\ref{T:corr-bd}, we use
these kernel bounds to obtain sharp bounds for the two-point correlation
function. In case of $C^{1, \alpha}$-domains with Dirichlet boundary condition,
we improve the above bounds, in Lemma~\ref{Lem:G^2} and
Proposition~\ref{P:K-C1alpha}, by including a factor which contains
$\Phi_1(x)\Phi_1(x')\Phi_1(y)\Phi_1(y')$.

\medskip As some applications of our moment bounds, in Theorems~\ref{Thm:pm-d}
and~\ref{Thm:pm-n}, we establish the \textit{full intermittency} property for
sufficiently large $\lambda$ under the Dirichlet boundary condition, and for all
$\lambda > 0$ under the Neumann boundary condition. This extends significantly
the results in~\cite{foondun.nualart:15:on}. We also apply our moment bounds to
study the \textit{$L^2$-energy} of the solution as a function of the parameter
$\lambda$. This property has been studied in~\cite{foondun.joseph:14:remarks}
and~\cite{khoshnevisan.kim:15:non-linear} on $(0, L)$ in the large $\lambda$
regime, i.e., as $\lambda \to \infty$, under the Neumann boundary condition. In
this paper, we study this property in both large and small $\lambda$ regimes and
for a general bounded domain with Neumann boundary condition. We find that, at a
fixed time, when $\lambda > 0$ is small, the $L^2$-energy of the solution on
$\dom$ has the exponential rate $\exp(C\lambda^2)$, which is different from the
rate $\exp(C\lambda^{4/(2-\beta)})$ when $\lambda$ is large (see
Theorem~\ref{Thm:pm-n} and Corollary~\ref{Cor:n}).

\begin{remark}
  Since the domain $\dom$ is bounded, it is natural to study the SPDE
  in~\eqref{E:she-d} or~\eqref{E:she-n} under the framework of
  infinite-dimensional stochastic differential equations as in Da Prato and
  Zabczyk~\cite{da-prato.zabczyk:14:stochastic}; see also
  Cerrai~\cite{cerrai:01:second} and Pr\'ev\^ot and
  R\"ockner~\cite{prevot.rockner:07:concise}. However, in order to obtain
  sharper pointwise estimates of the probabilistic moments with both $t>0$ and
  $x\in\dom$ fixed, and in order to demonstrate how the geometric and analytic
  properties of the boundary $\partial \dom$ affect the solution especially
  through the initial conditions, we adopt the random field approach in this
  paper. The random field approach was pioneered by
  Walsh~\cite{walsh:86:introduction} and extended by
  Dalang~\cite{dalang:99:extending};
  see~\cite{dalang.quer-sardanyons:11:stochastic} for a comparison of the two
  approaches.
  % Note that since the correlation function $f(\cdot)$ given in
  % Assumption~\ref{A:f} blows up at zero, the stochastic integral---the Walsh
  % integral (see~\eqref{E:mild-sol} below)---is not obtained by a $Q$-Wiener
  % process, but via a cylindrical Wiener process.
\end{remark}

Before we state our main results, let us first introduce some notations.
Throughout the paper, $G_D$ and $G_N$ denote the Dirichlet and Neumann heat
kernel, respectively. We use $G$ to denote either $G_D$ or $G_N$ when we do not
need to distinguish the two cases. We use $\|\cdot\|_p$ to denote the $L^p
(\Omega)$-norm. Moreover, $a \wedge b = \min\{a, b\}$ for any $a, b \in \R$.

\subsection{Main results}

Our first theorem concerns the existence and uniqueness of random field solution
(see Definition~\ref{D:sol} below) and the $p$-th moment bounds of the solution.
For any $c > 0$, set
\begin{align}\label{E:Jc}
  J_c(t, x) \coloneqq \int_\dom \frac{1}{1 \wedge t^{d/2}} e^{-c\frac{|x-y|^2}{t}} |\nu|(\ud y).
\end{align}
It is clear that $J_c(t,x)<\infty$ for all $t > 0$ and $x\in\dom$ if and only if
$|\nu| (\dom)<\infty$, i.e., $|\nu|$ is a finite Borel measure on $\dom$. Let
$\mu_1$ be the smallest positive eigenvalue of the operator $\mathscr{L}$ with
Dirichlet boundary condition on $\dom$, and $\Phi_1$ be the corresponding
eigenfunction such that
\begin{align}
  \begin{cases}
    \mathscr{L} \Phi_1(x) = \mu_1 \Phi_1(x), & x \in \dom, \\
    \Phi_1(x) = 0,                           & x \in \partial \dom,
  \end{cases}
\end{align}
with $\Phi_1$ chosen to be positive and usually normalized $\|\Phi_1\|_{L^2
(\dom)} = 1$; see Section~\ref{SS:Kernel}. \medskip

We use the following convention for the constants $\mu$, $c$, and $c'$ in
Theorems~\ref{T:ExUnque},~\ref{T:corr-bd}, and~\ref{T:C1alpha}. In case of
Dirichlet (resp.~Neumann) boundary condition, we set $\mu = \mu_1$, $c = c_1$,
$c' = c_2$ (resp.~$\mu = 0$, $c = c_3$, $c' = c_4$) as the constants given
by~\eqref{E:G-d} (resp.~\eqref{E:G-n} and~\eqref{E:G-n-lower}) in
Proposition~\ref{P:G} below.

\begin{theorem}\label{T:ExUnque}
  If $\dom$ is a bounded Lipschitz domain, the noise $\dot W$ satisfies
  Assumption~\ref{A:f} and $\sigma$ satisfies Assumption~\ref{A:sigma}, then
  there exists a random field solution to~\eqref{E:she-d} with Dirichlet
  boundary condition (and~\eqref{E:she-n} with Neumann boundary condition,
  respectively). Moreover:
  \begin{enumerate}
    \item[(i)] If $\nu$ has a bounded density, then for all $T > 0$ and all $p
      \ge 2$,
      \begin{equation}\label{E:sol-Lp}
        \sup_{0 < t \le T} \sup_{x \in \dom} \|u(t, x)\|_p < \infty.
      \end{equation}
    \item[(ii)] If $\nu$ is a signed Borel measure with $|\nu|(\dom)<\infty$,
      then there exists a positive finite constant $C$ such that for all $t >
      0$, $x \in \dom$ and all $p \ge 2$,
      \begin{equation}\label{E:sol-Lp-bd}
        \|u(t, x)\|_p \le C e^{t\Big(Cp\lambda^2 L_\sigma^2 + C p^{\frac{2}{2-\beta}}\lambda^{\frac{4}{2-\beta}}L_\sigma^{\frac{4}{2-\beta}}-\mu\Big)}J_c(t, x).
      \end{equation}
  \end{enumerate}
  In both cases, the solution is unique among all random field solutions such
  that for each $T>0$, there exists $C_T<\infty$ such that
  \begin{equation}\label{E:sol-unique}
    \|u(t, x)\|_2 \le C_T J_c(t, x) \quad \text{for all $(t,x)\in (0, T]\times \dom$}.
  \end{equation}
\end{theorem}

Parts (i) and (ii) of Theorem~\ref{T:ExUnque} are proved in
Sections~\ref{S:Bounded} and~\ref{S:Key}, respectively. \medskip

The next result is about the upper and lower bounds for the two-point
correlation of the solution. We need a few more notations: for $x\in \dom$,
denote
\begin{align}\label{E:Ue}
  \op{dist}(x, \partial \dom) \coloneqq \inf\left\{|x-y|: y \in \partial \dom \right\}, \quad
  \dom_\eps \coloneqq \left\{x \in \dom: \textrm{dist}(x, \partial \dom) > \eps \right\},
\end{align}
and accordingly,
\begin{align}\label{E:Jc,eps}
  J_{c,\eps}(t,x) \coloneqq \int_{\dom_\eps} \frac{1}{1 \wedge t^{d/2}} e^{-c\frac{|x-y|^2}{t}} |\nu|(\ud y).
\end{align}

\begin{theorem}[Two-point correlation]\label{T:corr-bd}
  Suppose that $\dom$ is a bounded Lipschitz domain and the initial data $\nu$
  is a finite nonnegative measure on $\dom$. Assume that the noise $\dot W$
  satisfies Assumption~\ref{A:f} and $\sigma$ satisfies
  Assumption~\ref{A:sigma}. Let $u$ be the solution to~\eqref{E:she-d} with
  Dirichlet boundary condition or~\eqref{E:she-n} with Neumann boundary
  condition.
  \begin{enumerate}
    \item[(i)] Assume~\eqref{E:Anderson} or the nonnegativity of the solution,
      namely, $u(t,x) \ge 0$ a.s. for all $(t,x)\in (0,\infty)\times\dom$. Then
      there exists a positive finite constant $C$ such that for all $t > 0$ and
      $x, x' \in \dom$,
      \begin{align}\label{E:corr-ub}
        \E\left(u(t, x) u(t, x')\right)
        \le C e^{2t \Big(C \lambda^2 \Lip_\sigma^2 + C \lambda^{\frac{4}{2-\beta}}
        \Lip_\sigma^{\frac{4}{2-\beta}} - \mu \Big)} J_c(t,x)J_c(t,x').
      \end{align}
    \item[(ii)] Assume~\eqref{E:lsigma} or~\eqref{E:Anderson}. Then, in case of
      Dirichlet boundary condition, there exists $0 < \eps_0 < 1$ such that for
      all $0 < \eps \le \eps_0$, there exists $\overline{C} = \overline{C}(\eps)
      > 0$ with $\lim_{\eps \to 0} \overline{C}(\eps) = 0$ such that for all $t
      > 0$ and $x, x' \in \dom_\eps$,
      \begin{align}\label{E:corr-lb}
        \qquad \E(u(t, x) u(t, x'))
        \ge \overline{C} e^{2t \Big(\overline{C} \lambda^2 \lip_\sigma^2 + \overline{C} \lambda^{\frac{4}{2-\beta}}
        \lip_\sigma^{\frac{4}{2-\beta}} - \mu \Big)} e^{-16c \frac{|x-x'|^2}{t}}J_{12c',\eps}(t,x)J_{12c',\eps}(t,x').
      \end{align}
      In case of Neumann boundary condition, if the heat kernel lower
      bound~\eqref{E:G-n-lower} below holds (which is the case, for example,
      when $\dom$ is a smooth, convex domain and $\mathscr{L} = -\Delta$; see
      Proposition~\ref{P:G} below), then there exists a constant $\overline{C} >
      0$ such that~\eqref{E:corr-lb} holds with $\eps = 0$ for all $t > 0$ and
      $x, x' \in \dom$.
  \end{enumerate}
\end{theorem}
Theorem~\ref{T:corr-bd} is proved at the end of Section~\ref{S:Two-point}.

\begin{remark}[Nonnegativity and comparison principle]\label{R:PathComparison}
  The condition $u\ge 0$ a.s.\ in Theorem~\ref{T:corr-bd} should be interpreted
  as $u(t,x)\ge 0$ a.s.\ for all $t>0$ and $x\in\dom$. It is generally believed
  that under condition~\eqref{E:Lsigma}, if the initial condition is
  nonnegative, then the solution to the stochastic heat equation (SHE) is
  nonnegative or even strictly positive, which is indeed a consequence of the
  well-known \textit{sample-path comparison principle} for SHE. In particular,
  Mueller~\cite{mueller:91:on} established the sample-path comparison principle
  for the case of SHE on $[0,1]$ with Neumann boundary conditions and space-time
  white noise. Later, Shiga~\cite{shiga:94:two} proved the case of SHE on $\R$
  with space-time white noise. The case of SHE on $[0,1]$ with Dirichlet
  boundary conditions was proved by Mueller and
  Nualart~\cite{mueller.nualart:08:regularity}. The case of SHE on $\R$ with a
  fractional Laplace, space-time white noise, and rough initial conditions was
  established in~\cite{chen.kim:17:on} and the case of SHE on $\R^d$ with rough
  initial data and with a noise that is white in time and homogeneously colored
  in space was proved in~\cite{chen.huang:19:comparison}. The sample-path
  comparison principle under the settings of the current paper is left as a
  future project.
\end{remark}

The last set of results focus on the case of $C^{1, \alpha}$-domains with the
Dirichlet boundary condition and some variations. Note that
$C^{1,\alpha}$-domains are a special case of Lipschitz domains. We need to
introduce some notations:
\begin{align}\label{E:Psi}
  \Psi(t, x) \coloneqq 1 \wedge \frac{\Phi_1(x)}{1 \wedge t^{1/2}}
  \quad \text{and} \quad
  J_c^* (t, x) \coloneqq \int_{\dom} \Psi(t, y)
  \frac{e^{-c\frac{|x-y|^2}{t}}}{1\wedge t^{d/2}} |\nu|(\ud y),
\end{align}
where $\Phi_1$ is the leading eigenfunction. In this case, we are able to
improve the previous results by giving a new condition \eqref{E:C1alpha-nu}
below, namely $\Phi_1 \in L^1(\dom, |\nu|)$, which is weaker than the above
condition $|\nu|(\dom) < \infty$ in Theorem~\ref{T:ExUnque} for existence and
all moments of solutions with measure-valued initial data. This new
integrability condition indicates the rate of blow-up for the initial data which
is allowed near the boundary $\partial \dom$ (see
Examples~\ref{Eg:interval}--\ref{Eg:Rect} and Remark~\ref{R:JcFinite} below),
and hence $\nu$ is not necessarily a finite measure. Moreover,
because
\begin{align*}
  \Psi(t, x)\Big|_{x \in \partial\dom} = 0 \quad \text{and} \quad
  J_c^* (t,x) \le J_c(t,x),
\end{align*}
the bounds in~\eqref{E:C1alpha-Lp}, \eqref{E:C1alpha-corr}
and~\eqref{E:C1alpha-corr-lb} below strengthen the previous
bounds~\eqref{E:sol-Lp-bd}, \eqref{E:corr-ub} and~\eqref{E:corr-lb},
respectively, especially near the boundary of the domain. Indeed,
\begin{enumerate}
  \item for any $t>0$ fixed, when $x$ is close to the boundary of $\dom$, the
    term $\Psi(t,x)$ in~\eqref{E:sol-Lp-bd} and~\eqref{E:corr-ub} plays the
    dominant role in pushing the moments to zero;
  \item for any $x\in\dom$ fixed, since $\Psi(\cdot,\cdot)\le 1$, when $t\to 0$, the
    term $J_c^* (t,x) \asymp (p_t*|\nu|)(x)$ defines the behavior of the
    moments. Here, $p_t(x)$ is the heat kernel on $\R^d$ and ``$*$'' refers to
    the spatial convolution; see~\eqref{E:J_0R^d}.
\end{enumerate}

\begin{theorem}\label{T:C1alpha}
  Let $\dom$ be a bounded $C^{1, \alpha}$-domain for some $\alpha > 0$ with the
  Dirichlet boundary condition at $\partial \dom$. Assume that the noise $\dot
  W$ satisfies Assumption~\ref{A:f} and $\sigma$ satisfies
  Assumption~\ref{A:sigma}. If the initial condition $\nu$ is any locally finite
  and signed measure that satisfies the following integrability condition
  \begin{align}\label{E:C1alpha-nu}
    \|\Phi_1\|_{L^1(\dom,\, |\nu|)} =  \int_\dom \Phi_1(y)\, |\nu|(\ud y) < \infty,
  \end{align}
  where $\Phi_1(\cdot)$ is the leading eigenfunction of the differential
  operator $\mathscr{L}$ on the domain $\dom$, then we have the following:
  \begin{enumerate}
    \item[(i)] There exists a random field solution to~\eqref{E:she-d}. The
      solution has the property that for some $C< \infty$, for all $t > 0$ and
      $x \in \dom$,
      \begin{align}\label{E:C1alpha-Lp}
        \|u(t, x)\|_p \le C e^{t\Big(Cp\lambda^2 L_\sigma^2 + Cp^{\frac{2}{2-\beta}}\lambda^{\frac{4}{2-\beta}}L_\sigma^{\frac{4}{2-\beta}} - \mu_1\Big)}
        \Psi(t, x) J_{2c_1/3}^* (t, x),
      \end{align}
      where $\mu_1$ and $c_1$ are the constants in Proposition~\ref{P:G} below.
      Moreover, the solution is unique among all random field solutions such
      that for each $T > 0$, there exists $C_T < \infty$ such that
      \begin{align}\label{E:C1alpha-unique}
        \|u(t, x)\|_2 \le C_T \Psi(t, x) J_c^* (t, x) \quad \text{for all $(t,x) \in (0, T]\times \dom$}.
      \end{align}
    \item[(ii)] Assume~\eqref{E:Anderson} or the nonnegativity of the solution,
      namely, $u(t,x) \ge 0$ a.s. for all $(t,x)\in (0,\infty)\times\dom$. Then
      for all $t > 0$ and $x, x' \in \dom$,
      \begin{align}\label{E:C1alpha-corr}
      \begin{split}
        \MoveEqLeft \E(u(t, x) u(t, x'))\\
        &\le C e^{2t \Big(C \lambda^2 \Lip_\sigma^2 + C \lambda^{\frac{4}{2-\beta}} \Lip_\sigma^{\frac{4}{2-\beta}} - \mu_1 \Big)}
        \Psi(t, x)\Psi(t, x')
        J_{2c_1/3}^* (t, x) J_{2c_1/3}^* (t, x').
      \end{split}
      \end{align}
    \item[(iii)] Assume~\eqref{E:lsigma} or~\eqref{E:Anderson}. Then, there
      exists $\bar C > 0$ such that for all $t > 0$ and $x, x' \in \dom$,
      \begin{align}\label{E:C1alpha-corr-lb}
        \begin{split}
        \MoveEqLeft \E(u(t, x)u(t, x'))\\
        &\ge \bar C e^{2t \left( \bar C \lambda^2 l_\sigma^2 + \bar C \lambda^{\frac{4}{2-\beta}} l_\sigma^{\frac{4}{2-\beta}} - \mu_1\right)}
        e^{-16c_2\frac{|x-x'|^2}{t}}\Psi(t, x) \Psi(t, x') J^*_{12c_2}(t, x) J^*_{12c_2}(t, x').
        \end{split}
      \end{align}
  \end{enumerate}
\end{theorem}

% \begin{remark}\label{R:RoughData}
%   When $\dom = \R^d$ and hence the boundary is at $|x|\to\infty$, the optimal
%   condition for the initial measure is given by the following integral
%   condition:
%   \begin{align}\label{E:J0Finite}
%     \int_{\R^d} e^{- a |x|^2} |\nu|(\ud x) <\infty, \quad \text{for all $a>0$.}
%   \end{align}
%   Initial conditions of this type, which are called \textit{rough initial
%   conditions}, were studied
%   in~\cite{chen.dalang:15:moments,chen.kim:17:on,chen.huang:19:comparison}.
%   Condition~\eqref{E:C1alpha-nu} is the right extension for the rough initial
%   condition~\eqref{E:J0Finite} from $U = \R^d$ to a bounded domain; see
%   Examples~\ref{Eg:interval}--\ref{Eg:Rect} for more precise conditions on the
%   initial data.
% \end{remark}

\begin{remark}\label{R:JcFinite}
  Condition~\eqref{E:C1alpha-nu} holds if and only if
  \begin{align*}
    J_c^* (t,x) <\infty \quad \text{for all $c>0$, $t>0$ and $x\in \dom$.}
  \end{align*}
  which is the consequence of the following bounds:
  \begin{align}\label{E:JcMnu}
    ((c_0D)^{-1} \wedge 1)e^{-\frac{c D^2}{t}} \|\Phi_1\|_{L^1(\dom,\, |\nu|)}
    \le J_c^*\left(t,x\right)
    \le \frac{\|\Phi_1\|_{L^1(\dom,\, |\nu|)}}{1\wedge t^{(d+1)/2}} \quad
    \text{for all $t>0$ and $x\in\dom$,}
  \end{align}
  where $D\coloneqq\sup\{|x-y|,\: x,y\in\dom\}$ and we have used the fact that
  $\Phi_1(\cdot)$ is bounded (see Remark~\ref{R:Phi1Bdd}). The proof
  of~\eqref{E:JcMnu} is straightforward and is left as an exercise for
  interested readers.
\end{remark}

In fact, Theorem~\ref{T:C1alpha} can easily be extended for Cartesian products
of bounded $C^{1, \alpha}$-domains to allow some Lipschitz domains; see
Example~\ref{Eg:interval} below.

\begin{corollary}\label{C:prod-dom}
  Let $U$ be a bounded domain in the following Cartesian product form:
  \begin{align*}
    \dom = \dom_1 \times \dom_2 \times \cdots \times \dom_m \subseteq\R^d, \quad
    \text{with $m\ge 1$, $U_i\subseteq \R^{d_i}$, $d_i\ge 1$, and
    $\sum_{i = 1}^{m} d_i = d$.}
  \end{align*}
  Assume that each $\dom_i$ is a bounded $C^{1, \alpha_i}$-domain for some
  $\alpha_i > 0$. Let $\mathscr{L}_i$ be a uniformly elliptic differential
  operator on $\dom_i$ of the form~\eqref{E:op-L} satisfying the
  condition~\eqref{E:u-ellip}. Consider the SPDE~\eqref{E:she-d} with
  $\mathscr{L} = \mathscr{L}_1 + \dots + \mathscr{L}_m$ on $\dom$ with the
  Dirichlet boundary condition. Suppose $\sigma$ satisfies Assumption \ref{A:sigma} and
  the initial measure $\nu$ on $\dom$ satisfies the following integrability
  condition
  \begin{align}\label{E:int-prod-nu}
    \int_\dom \prod_{i = 1}^m \Phi_1^{\dom_i}(x_i) |\nu|(\ud x) < \infty,
  \end{align}
  where $\Phi_1^{\dom_i}(x_i)$ is the eigenfunction corresponding to the leading
  eigenvalue $\mu_i^{\dom_i}$ of the Dirichlet operator $\mathscr{L}_i$ on
  $\dom_i$. Let $J^*_c(t, x)$ be defined as in~\eqref{E:Psi} but with $\Psi$
  replaced by
  \begin{align}\label{E:Psi-prod}
    \Psi^* (t, x) \coloneqq \prod_{i = 1}^m \left(1 \wedge \frac{\Phi_1^{\dom_i}(x_i)}{1 \wedge t^{1/2}}\right).
  \end{align}
  Then, the statements (i) and (ii) in Theorem~\ref{T:C1alpha} above hold with
  $\mu_1 = \sum_{i = 1}^m \mu_i^{\dom_i}$, which is the leading eigenvalue of
  the Dirichlet operator $\mathscr{L}$, and with $\Psi$ defined in~\eqref{E:Psi}
  replaced by $\Psi^*$ defined above in~\eqref{E:Psi-prod}.
\end{corollary}

Theorem~\ref{T:C1alpha} and Corollary~\ref{C:prod-dom} are proved at the end of
Section~\ref{S:Holder-Domain}. Finally, Table~\ref{Tb:Summary} below summarizes
the main results of this paper.
\begin{table}[htpb]
  \centering
  \renewcommand{\arraystretch}{1.2}
  \newcommand\gc{\cellcolor{lightgray}}
  \caption{Summary of the main results and the qualitative differences across domains of increasing regularity/conditions (progressing from the first to the third row).}
  \label{Tb:Summary}

  \begin{tabular}{|c|c|c|c|c|}
    \hline \rowcolor{lightgray}
                                                                     & \multicolumn{2}{c|}{Dirichlet} \gc               & \multicolumn{2}{c|}{Neumann} \gc \\ \hline
    Domain \gc                                                       & Upper bound \gc                                  & Lower bound \gc                   & Upper bound \gc                            & Lower bound \gc    \\ \hline
    \multirow{2}{*}{Lipschitz}                                       & \eqref{E:sol-Lp-bd} and~\eqref{E:corr-ub}:       & \eqref{E:corr-lb}:                & \eqref{E:sol-Lp-bd} and~\eqref{E:corr-ub}: &                    \\
                                                                     & $\dagger$                                          & $\ddagger$                          & $\dagger$                                    &                    \\ \hline
    \multirow{2}{2.5cm}{\centering $C^{1,\alpha}$ or their products} & \eqref{E:C1alpha-Lp} and~\eqref{E:C1alpha-corr}: & \eqref{E:C1alpha-corr-lb}:        &                                            &                    \\
                                                                     & $\dagger$, $\ast$                                    & $\dagger$, $\ast$                     &                                            &                    \\ \hline
    \multirow{2}{*}{Smooth and convex}                               &                                                  &                                   &                                            & \eqref{E:corr-lb}: \\
                                                                     &                                                  &                                   &                                            & $\dagger$            \\ \hline
  \end{tabular}
  \bigskip

  $\dagger$\:: Hold(s) on $U$; \quad
  $\ddagger$\:: Hold(s) on $U_\epsilon$ for $\epsilon > 0$; \quad
  $\ast$\:: Better estimates near $\partial U$.
\end{table}

\subsection{Outline of the paper}

The rest of the paper is organized as follows. In Section~\ref{S:Applications},
we first give some concrete examples and apply our moment bounds to establish
the full intermittency of the solution and discuss its $L^2$-energy. Then in
Section~\ref{S:Prelim}, we give some preliminaries which include the definition
of the mild solution in Section~\ref{SS:Mild}, the cone condition for the domain
in Section~\ref{SS:domain}, and the heat kernel estimates for the
equations~\eqref{E:she-d} and~\eqref{E:she-n} in Section~\ref{SS:Kernel}. Then
in Sections~\ref{S:Bounded}, resp.~\ref{S:Key}, we derive the moment bounds in
case of bounded, resp.~rough, initial conditions, and prove the two cases in
Theorem~\ref{T:ExUnque}. The two-point correlation function is studied in
Section~\ref{S:Two-point}, where Theorem~\ref{T:corr-bd} is proved. The case of
bounded $C^{1,\alpha}$-domains with Dirichlet condition is studied in
Section~\ref{S:Holder-Domain}, at the end of which we prove
Theorem~\ref{T:C1alpha} and Corollary~\ref{C:prod-dom}.

\section{Examples and applications}\label{S:Applications}

In this Section, we give some examples of our main results and apply our moment
bounds to study the intermittency property and the $L^2$-energy of the
solutions. \medskip

\subsection{Rough initial conditions under Dirichlet boundary condition}\label{SS:Examples}

In this part, we give a few examples to illustrate Theorem~\ref{T:C1alpha} and
Corollary~\ref{C:prod-dom}.

\begin{example}[Interval for $d = 1$]\label{Eg:interval}
  Consider the stochastic heat equation~\eqref{E:she-d} with $\mathscr{L} =
  -\partial^2 /\partial x^2$ on an interval $\dom = (0, L)$ with Dirichlet
  boundary condition. This is a smooth, and hence $C^{1,\alpha}$ domain. The
  first eigenvalue is $\mu_1 = (\pi/L)^2$ and the corresponding eigenfunction is
  $\Phi_1(x) = (2/L)^{1/2} \sin(\pi x / L)$. In this case, $\Psi(t,x)$ defined
  in~\eqref{E:Psi} reduces to (see Figure~\ref{SF:Psi_1d})
  \begin{align}\label{E:Psi_1d}
    \Psi\left(t,x\right) = 1\wedge\frac{(2 /L)^{1/2} \sin\left(\pi x/L\right)}{1\wedge \sqrt{t}},
  \end{align}
  and condition~\eqref{E:C1alpha-nu} becomes
  \begin{align}\label{E:sin-nu}
    \int_0^L \sin(\pi x / L) \,|\nu|(\ud x) < \infty
    \quad \Longleftrightarrow \quad
    \int_0^L x(L-x) \,|\nu|(\ud x) < \infty.
  \end{align}
  Due to the dissipative or cooling-down effect of the Dirichlet boundary
  condition, one can inject, at time zero, more heat flow into the domain from
  the boundary. For example, the following nonnegative measures with compact
  support satisfies~\eqref{E:sin-nu} or equivalently~\eqref{E:C1alpha-nu}:
  \begin{align}\label{E:sin-nu2}
    \nu(\ud x) = \frac{\textbf{1}_{(0,L)}(x)}{\left[\sin(\pi x/L)\right]^\beta} \ud x
    \quad \text{or} \quad
    \nu(\ud x) = \frac{\textbf{1}_{(0,L)}(x)}{\left[x(L-x)\right]^\beta} \ud x,
    \quad \text{with $\beta<2$.}
  \end{align}
  But since examples in~\eqref{E:sin-nu2} neither have bounded densities nor are
  finite measures on the domain being considered, both parts of
  Theorem~\ref{T:ExUnque} fail to apply for such initial conditions.
\end{example}

\begin{figure}
  \centering

  \renewcommand{\thesubfigure}{\thefigure.\arabic{subfigure}}
  \makeatletter
    \renewcommand{\p@subfigure}{}
  \makeatother
  \captionsetup[subfigure]{style = default, margin = 0pt, parskip = 0pt, hangindent = 0pt, labelformat = simple}

  \subfloat[][$d = 1$.]{\label{SF:Psi_1d}
    \centering
    \includegraphics[width = 0.35\textwidth]{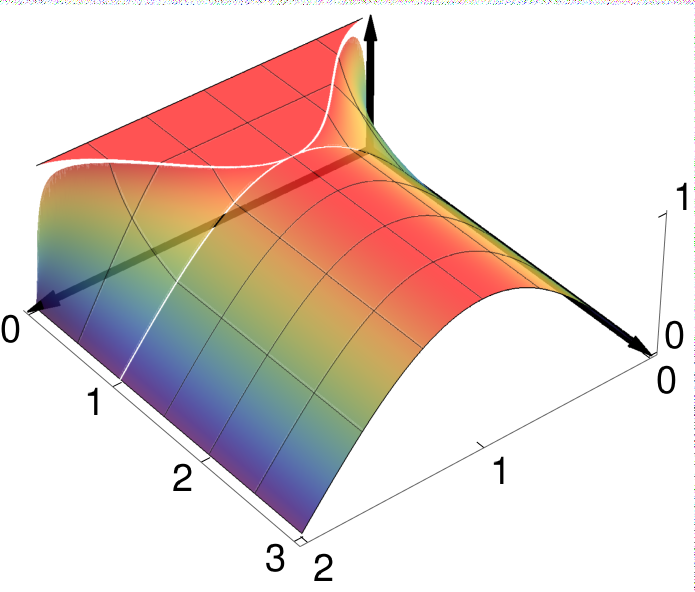}
  }
  \qquad\qquad
  \subfloat[][$d \ge 2$.]{\label{SF:Psi_4d}
    \centering
    \includegraphics[width = 0.35\textwidth]{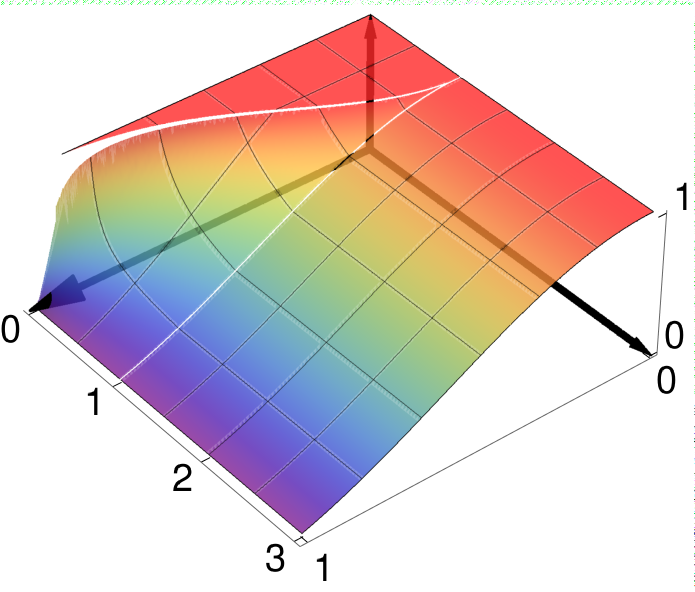}
  }

  \caption{Some plots of the function $\Psi(t,x)$ in case of $d = 1$ in
    Figure~\ref{SF:Psi_1d} with $L = 2$, $x\in [0,2]$ and $t\in[0,3]$ and
    $\Psi(t,r)$ in case of $d\ge 2$ in Figure~\ref{SF:Psi_4d} with $d = 4$,
    $r\in[0,1]$ and $t\in[0,3]$. The normalization constant $C_d$ for the plot
  in Figure~\ref{SF:Psi_4d} is chosen to be the one in~\eqref{E:Cd} so that
$\max_{(t,r)\in (0,\infty)\times(0,1)}\Psi(t,r)=1$.}

  \label{F:Psi}
\end{figure}

\begin{example}[Unit ball in $\R^d$]\label{Eg:ball}
  Consider the stochastic heat equation~\eqref{E:she-d} with $\mathscr{L} =
  -\Delta$ on the unit disk $\dom = B(0, 1)$ in $\R^d$, $d\ge 2$, with Dirichlet
  boundary condition. The first eigenvalue is $\mu_1 = z_0^2$, where $z_0$ is
  the first positive zero of the Bessel function $J_{(d-2)/2}(\cdot)$, and the
  corresponding eigenfunction is $\Phi_1(x) = \frac{1}{C_d}|x|^{(2-d)/2}
  J_{(d-2)/2}(z_0|x|)$; see Remark~\ref{R:EgienBall} below for more details. In
  this case, $\Psi(t,x)$ defined in~\eqref{E:Psi} reduces to (see
  Figure~\ref{SF:Psi_4d} with $r = |x|$)
  \begin{align}\label{E:Psi_4d}
    \Psi(t,x) = 1\wedge \frac{|x|^{(2-d)/2} J_{(d-2)/2}(z_0 |x|)}{C_d\left(1\wedge \sqrt{t}\right)}.
  \end{align}
  Similarly to the previous example, we claim that the following locally finite
  nonnegative measure on $\R^d$ with compact support
  \begin{align}\label{E:BesselJFinite2}
    \nu(\ud x) = \frac{\textbf{1}_{B(0,1)}(x)}{|x|^{\beta_0} \left(1-|x|\right)^{\beta_1}} \ud x,
    \quad \text{with $\beta_0 < 2$ and $\beta_1 < 2$,}
  \end{align}
  satisfies condition~\eqref{E:C1alpha-nu}. Indeed,
  \begin{align*}
    0 \le \int_{B(0,1)}\frac{|x|^{(2-d)/2} J_{(d-2)/2}(z_0|x|)}{|x|^{\beta_0}\left(1-|x|\right)^{\beta_1}}\ud x
    =   C_d \int_0^1 \frac{r^{(2-d)/2} J_{(d-2)/2}(z_0 r)}{r^{\beta_0-1} \left(1-r\right)^{\beta_1}} \ud r
      \le C_d \times C \times I, \shortintertext{where}
    C \coloneqq \max_{r\in (0,1)}\frac{r^{(2-d)/2}J_{(d-2)/2}(z_0 r)}{1-r} <\infty \quad \text{and} \quad
    I \coloneqq \int_0^1 \frac{\ud r}{r^{\beta_0-1} (1-r)^{\beta_1-1}} <\infty.
  \end{align*}
  Note that the above maximum is finite thanks to Lemma~\ref{L:BesselJMax}
  below. By the same reason, condition~\eqref{E:C1alpha-nu} in this case reduces
  to
  \begin{align}\label{E:BesselJFinite}
    \int_{B(0, 1)} \left(1-|x|\right) \; |\nu|(\ud x) < \infty.
  \end{align}
\end{example}

\begin{example}[Annular domain in $\R^2$]\label{Eg:Ann}
  Consider the stochastic heat equation~\eqref{E:she-d} with $\mathscr{L} =
  -\Delta$ on the following annulus with Dirichlet boundary
  condition\footnote{The explicit form of the fundamental solution can be found,
  e.g., in~\cite[Section 4.1.2 on p. 418]{polyanin.nazaikinskii:16:handbook}.}:
  \begin{align*}
    \dom = \left\{x \in \R^2: R_1 < |x| < R_2 \right\}, \quad 0<R_1<R_2<+\infty.
  \end{align*}
  Note that $\dom$ is a nonconvex, but smooth, bounded domain. The leading
  eigenvalue is $\mu_1 = z_0^2$, where $z_0$ is the first positive zero of the
  cross-product Bessel functions
  \begin{align}\label{E:ZeroJY}
    J_0(R_1 z) Y_0(R_2 z) -
    Y_0(R_1 z) J_0(R_2 z) = 0,
  \end{align}
  where $J_0(\cdot)$ and $Y_0(\cdot)$ are the Bessel functions of the first and
  second kind of order zero, respectively. The corresponding eigenfunction is
  \begin{align}\label{E:Z(r)}
    \Phi_1(x) = C Z(|x|) \quad \text{with} \quad
    Z(r) \coloneqq
    J_0(R_1 z_0) Y_0(r z_0) -
    Y_0(R_1 z_0) J_0(r z_0) ;
  \end{align}
  see Figure~\ref{SF:Phi1_Ann} for a plot of $\Phi_1(x)$. Similarly to the
  previous example~\eqref{Eg:ball}, we claim that
  \begin{align}\label{E:nu_Ann}
    \nu(\ud x) = \frac{|x|^{\beta_0}\textbf{1}_{\dom}(x)}{ \left(R_2-|x|\right)^{\beta_2} \left(|x|-R_1\right)^{\beta_1}}  \ud x,
    \quad \text{with $\beta_i<2$, $i=1,2$, and $\beta_0\in\R$,}
  \end{align}
  satisfies condition~\eqref{E:C1alpha-nu}. Indeed,
  \begin{align*}
    0 & < \int_\dom \frac{\Phi_1(x) |x|^{\beta_0}\ud x}{\left(R_2-|x|\right)^{\beta_2} \left(|x|-R_1\right)^{\beta_1}}
    = 2 \pi \int_{R_1}^{R_2} \frac{Z(r) r^{\beta_0+1} \ud r}{\left(R_2 - r\right)^{\beta_2} \left(r-R_1\right)^{\beta_1}}
        \le 2 \pi \times C \times I,
  \end{align*}
  where
  \begin{align*}
    \displaystyle C \coloneqq \max_{r\in \left(R_1,R_2\right)}
    \frac{Z(r)}{\left(R_2-r\right)\left(r-R_1\right)}<\infty
    \quad \text{and} \quad
    I  \coloneqq \int_{R_1}^{R_2} \frac{r^{\beta_0+1} \ud r}{\left(R_2-r\right)^{\beta_2-1} \left(r-R_1\right)^{\beta_1-1}} <\infty.
  \end{align*}
  Note that the finiteness of the above constant $C$ is due to the fact that
  $R_1$ and $R_2$ are both simple zeros of $Z(r)$; see Lemma \ref{L:simple0}
  below. By the same reason, in this case, condition~\eqref{E:C1alpha-nu} can be
  equivalently written as
  \begin{align*}
    \int_{R_1<|x|<R_2} \left(R_2-|x|\right) \left(|x|-R_1\right)\; |\nu|(\ud x) < \infty.
  \end{align*}
\end{example}

\begin{figure}
  \centering

  \renewcommand{\thesubfigure}{\thefigure.\arabic{subfigure}}
  \makeatletter
    \renewcommand{\p@subfigure}{}
  \makeatother
  \captionsetup[subfigure]{style = default, margin = 0pt, parskip = 0pt, hangindent = 0pt, labelformat = simple}

  \subfloat[][Annular domain.]{\label{SF:Phi1_Ann}
    \centering
    \includegraphics[width = 0.35\textwidth]{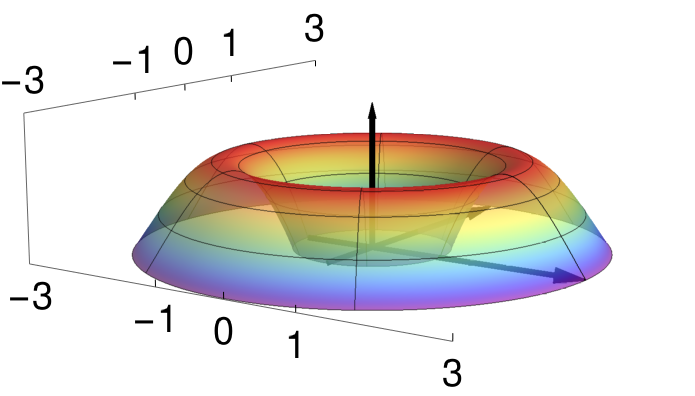}
  }
  \qquad\qquad
  \subfloat[][Rectangular domain.]{\label{SF:Phi1_Box}
    \centering
    \includegraphics[width = 0.33\textwidth]{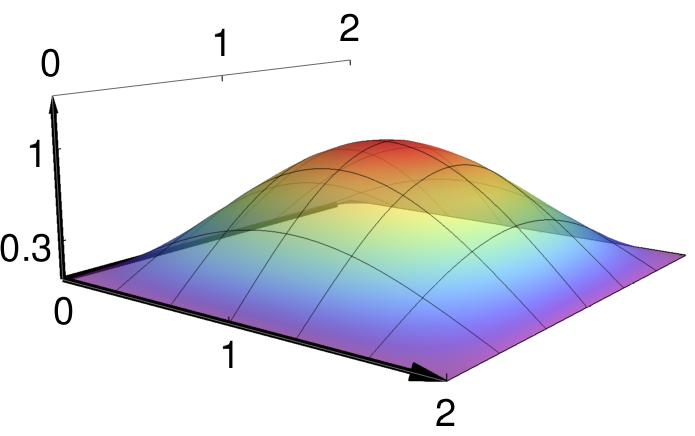}
  }

  \caption{Some plots of the leading eigenfunction $\Phi_1(x)$ with $x\in\R^2$
  both in case of the annular domain in Figure~\ref{SF:Phi1_Ann} where $R_1 = 1$
and $R_2 = 3$ and of the rectangular domain in Figure~\ref{SF:Phi1_Box} where $L
= 2$.}

  \label{F:Phi1}
\end{figure}

\begin{example}[Rectangular domain in $\R^d$ with $d\ge 2$]\label{Eg:Rect}
  Consider the stochastic heat equation~\eqref{E:she-d} with $\mathscr{L} =
  -\Delta$ on the rectangular domain $\dom = (0, L)^d\subseteq \R^d$, $d\ge 2$
  and $L>0$, with Dirichlet boundary condition. Note that for $d \ge 2$, $\dom$
  has corners and hence, is not a $C^{1,\alpha}$--domain, but only a Lipschitz
  domain. The first eigenvalue is $\mu_1 = d \times \left(\pi/L\right)^{2}$ and
  the corresponding (normalized) eigenfunction is given by (see
  Figure~\ref{SF:Phi1_Box} for a plot)
  \begin{align*}
    \Phi_1(x) = (2 / L)^{d/2}\prod_{i = 1}^d \sin\left(\frac{\pi x_i}{L}\right),
    \qquad \text{for $x = \left(x_1,\dots, x_d\right)\in (0,L)^d$.}
  \end{align*}
  Similar to Example~\ref{Eg:interval}, by Corollary~\ref{C:prod-dom},
  $\Psi^*\left(t,x\right)$ defined in~\eqref{E:Psi-prod} reduces to
  \begin{align}\label{E:Psi_rec}
    \Psi^*\left(t,x\right) = \prod_{i = 1}^{d} \left(1\wedge\frac{(2 /L)^{1/2} \sin\left(\pi x_i/L\right)}{1\wedge \sqrt{t}}\right),
  \end{align}
  and condition~\eqref{E:int-prod-nu} becomes
  \begin{align}\label{E:BoxCond}
    \int_{(0,L)^d}|\nu|(\ud x) \prod_{i = 1}^d \sin\left(\frac{\pi x_i}{L}\right) <\infty
    \quad \Longleftrightarrow \quad
    \int_{(0,L)^d}|\nu|(\ud x) \prod_{i = 1}^d \left(x_i (L-x_i)\right) <\infty.
  \end{align}
  Locally finite measures similar to~\eqref{E:sin-nu2} can be constructed
  component-wise.
\end{example}

Other Lipschitz domains can be considered as an application of
Corollary~\ref{C:prod-dom} as well. For example, for the cylinder domain
\begin{align*}
  \dom = \left\{(x_1,x_2,x_3)\in\R^3: x_1^2 + x_2^2 < 1~\text{and}~0< x_3
  <1\right\},
\end{align*}
as an easy exercise (which is left for the interested readers),
condition~\eqref{E:int-prod-nu} becomes
\begin{align*}
  \int_{\R^3} \left(1-x_1^2 -x_2^2\right) x_3 (1-x_3) \: |\nu|(\ud x) <\infty.
\end{align*}

\subsection{Intermittency}

Following~\cite{foondun.khoshnevisan:09:intermittence} and Definition III.1.1
of~\cite{carmona.molchanov:94:parabolic}, we say that $u$ is \emph{weakly
intermittent} if, for all $x \in \dom$,
\begin{gather}\label{E:low-lya}
  \limsup_{t \to \infty} \frac 1t \log \E(|u(t, x)|^2) > 0 \quad \text{and} \\
  \limsup_{t \to \infty} \frac 1t \log \E(|u(t, x)|^p) < \infty
  \quad \text{for all $p\ge 2$}, \label{E:up-lya}
\end{gather}
and $u$ is \emph{fully intermittent} (or simply intermittent)
if~\eqref{E:low-lya} can be strengthened to
\begin{align}\label{E:F-Intermit}
  \liminf_{t \to \infty} \frac 1t \log \E(|u(t, x)|^2) > 0.
\end{align}

The following two theorems extend the corresponding results
in~\cite{foondun.nualart:15:on} and~\cite{nualart:18:moment}. Recall that
$\mu_1$ is the first eigenvalue of the operator $\mathscr{L}$ with Dirichlet
boundary condition. Also, recall the definitions of $J_c(t, x)$ and $J_{c,
\varepsilon}(t, x)$ in~\eqref{E:Jc} and~\eqref{E:Jc,eps}, respectively. The
following theorem provides moment bounds for the solution and shows that full
intermittency occurs when $\lambda$ is sufficiently large, but not when
$\lambda$ is small.

\begin{theorem}\label{Thm:pm-d}
  Let $\dom \subset \R^d$ be a bounded Lipschitz domain. Let $u$ be the solution
  to~\eqref{E:she-d} with Dirichlet boundary condition. Suppose~\eqref{E:Lsigma}
  holds. Suppose $\nu \ge 0$ and $\nu(\dom) < \infty$. Then, there exist
  positive finite constants $C$, $c$ and $c'$ such that for all $p \ge 2$,
  $\lambda > 0$, $t > 0$, $x \in \dom$,
  \begin{equation}\label{pm-d-ub}
    \E(|u(t, x)|^p)
    \le C^p \big(J_{c_1} (t, x)\big)^p
    e^{pt{\Big(c p \lambda^2 \Lip_\sigma^2 +
    c' p^{\frac{2}{2-\beta}} \lambda^{\frac{4}{2-\beta}}\Lip_\sigma^{\frac{4}{2-\beta}}-\mu_1\Big)}}.
  \end{equation}
  Moreover, if~\eqref{E:Anderson} or~\eqref{E:lsigma} holds, then there exists
  $0 < \eps_0 < 1$ such that for all $0 < \eps \le \eps_0$, there exist positive
  finite constants $\overline{C}$, $\overline{c}$ and $\widetilde{c}$ depending
  on $\eps$ such that for all $p \ge 2$, $\lambda > 0$, $t > 0$, $x \in
  \dom_\eps$,
  \begin{equation}\label{pm-d-lb}
  \E(|u(t, x)|^p)
  \ge \overline{C}^p \big(J_{12c_2, \varepsilon}(t, x)\big)^p
  e^{pt{\Big(\overline{c} \lambda^2 \lip_\sigma^2 +
  \widetilde{c} \lambda^{\frac{4}{2-\beta}}\lip_\sigma^{\frac{4}{2-\beta}}-\mu_1\Big)}}.
  \end{equation}
  Consequently, if $0 < \nu(\dom_\eps) \le \nu(\dom) < \infty$, then there exist
  $0 < \lambda_0 < \lambda_1 < \infty$ such that $u$ is fully intermittent on
  $\dom_\eps$ when $\lambda > \lambda_1$, but not when $\lambda < \lambda_0$ as
  \begin{align*}
    \limsup_{t\to \infty} \frac 1t \log\E(|u(t, x)|^2)<0,
  \end{align*}
  where
  \begin{gather*}
    \lambda_0 \coloneqq \sup \Big\{\lambda > 0: 2 c \lambda^2 \Lip_\sigma^2 + 2^{\frac{2}{2-\beta}}c' \lambda^{\frac{4}{2-\beta}} \Lip_\sigma^{\frac{4}{2-\beta}} \le \mu_1 \Big\}, \\
    \lambda_1 = \lambda_1(\eps) \coloneqq \inf \Big\{\lambda > 0: \overline{c} \lambda^2 \lip_\sigma^2 + \widetilde{c} \lambda^{\frac{4}{2-\beta}} \lip_\sigma^{\frac{4}{2-\beta}} \ge \mu_1 \Big\}.
  \end{gather*}
\end{theorem}

\begin{proof}
  The upper bound~\eqref{pm-d-ub} follows from Theorem~\ref{T:ExUnque}. The
  lower bound follows from Theorem~\ref{T:corr-bd} and Jensen's inequality
  $\E(|u(t, x)|^p) \ge \E(|u(t, x)|^2)^{p/2}$.

  It remains to prove the last statement of full intermittency. First, since
  $\nu(\dom) < \infty$,
  \begin{align*}
    \limsup_{t \to \infty} \frac 1t \log \big(J_{c_1}(t, x)\big)^p
    \le \limsup_{t\to \infty} \frac pt \left(-\log(1\wedge t^{d/2}) + \log \nu(\dom)\right) = 0.
  \end{align*}
  Then,~\eqref{pm-d-ub} implies that for all $x \in \dom$ and $\lambda > 0$,
  \begin{align}\label{limsup-ub}
    \limsup_{t \to \infty} \frac 1t \log\E(|u(t, x)|^p) \le p\Big({c p \lambda^2 \Lip_\sigma^2 + c' \, p^{\frac{2}{2-\beta}} \, \lambda^{\frac{4}{2-\beta}}\Lip_\sigma^{\frac{4}{2-\beta}}} - \mu_1\Big),
  \end{align}
  which proves~\eqref{E:up-lya}. Moreover, $\nu(\dom_\eps) > 0$ implies $\log
  \nu(\dom_\eps) > -\infty$, hence
  \begin{align*}
    \liminf_{t\to\infty}\frac1t \log \big(J_{12c_2, \varepsilon}(t, x)\big)^2
    \ge \liminf_{t\to\infty} \frac 2t \bigg(- \log(1\wedge t^{d/2}) - 4c_2\sup_{x, y \in \dom}|x-y|^2 + \log\nu(\dom_\eps)\bigg) = 0.
  \end{align*}
  If $\lambda > \lambda_1$, then it follows from~\eqref{pm-d-lb} that for all $x
  \in \dom_\eps$,
  \begin{align*}
    \liminf_{t\to \infty} \frac 1t \log\E(|u(t, x)|^2)
    \ge 2\Big(\overline{c} \lambda^2 \lip_\sigma^2 + \widetilde{c} \lambda^{\frac{4}{2-\beta}} \lip_\sigma^{\frac{4}{2-\beta}} - \mu_1\Big) > 0,
  \end{align*}
  which proves~\eqref{E:F-Intermit}. Hence, $u$ is fully intermittent on
  $\dom_\eps$. On the other hand, by~\eqref{limsup-ub},
  \begin{align*}
        \limsup_{t\to \infty} \frac 1t \log\E(|u(t, x)|^2)
    \le 2\Big(2 c \lambda^2 \Lip_\sigma^2 + 2^{\frac{2}{2-\beta}}c' \lambda^{\frac{4}{2-\beta}} \Lip_\sigma^{\frac{4}{2-\beta}} - \mu_1\Big)
    <   0, \quad \text{when $\lambda<\lambda_0$,}
  \end{align*}
  which completes the proof of Theorem~\ref{Thm:pm-d}.
\end{proof}

Similarly, we get the following result from Theorem~\ref{T:C1alpha} for $C^{1,
\alpha}$-domains.

\begin{theorem}\label{Thm:pm-d2}
  Let $\dom \subset \R^d$ be a bounded $C^{1, \alpha}$-domain, where $\alpha >
  0$. Let $u$ be the solution to~\eqref{E:she-d} with Dirichlet boundary
  condition. Suppose~\eqref{E:Lsigma} holds. Suppose $\nu \ge 0$ and $\Phi_1 \in
  L^1(\dom, \nu)$. Then, there exist positive finite constants $C$, $c$ and $c'$
  such that for all $p \ge 2$, $\lambda > 0$, $t > 0$, $x \in \dom$,
  \begin{equation}\label{pm-d-ub2}
    \E(|u(t, x)|^p)
    \le C^p \Psi^p(t, x) \big(J_{2c_1/3}^* (t, x)\big)^p
    e^{pt{\Big(c p \lambda^2 \Lip_\sigma^2 +
    c' p^{\frac{2}{2-\beta}} \lambda^{\frac{4}{2-\beta}}\Lip_\sigma^{\frac{4}{2-\beta}}-\mu_1\Big)}}.
  \end{equation}
  Moreover, if~\eqref{E:Anderson} or~\eqref{E:lsigma} holds, then there exist
  positive finite constants $\overline{C}$, $\overline{c}$ and $\widetilde{c}$
  such that for all $p \ge 2$, $\lambda > 0$, $t > 0$, $x \in \dom$,
  \begin{equation}\label{pm-d-lb2}
    \E(|u(t, x)|^p)
    \ge \overline{C}^p \Psi^p(t, x) \big(J_{12c_2}^*(t, x)\big)^p
    e^{pt{\Big(\overline{c} \lambda^2 \lip_\sigma^2 +
    \widetilde{c} \lambda^{\frac{4}{2-\beta}}\lip_\sigma^{\frac{4}{2-\beta}}-\mu_1\Big)}}.
  \end{equation}
  Consequently, if $0 < \|\Phi_1\|_{L^1(\dom,\, \nu)} < \infty$, then there
  exist $0 < \lambda_0 < \lambda_1 < \infty$ such that $u$ is fully intermittent
  on $\dom$ when $\lambda > \lambda_1$, but not when $\lambda < \lambda_0$,
  where
  \begin{align*}
    \lambda_0 & \coloneqq \sup \Big\{\lambda > 0: 2 c \lambda^2 \Lip_\sigma^2 + 2^{\frac{2}{2-\beta}}c' \lambda^{\frac{4}{2-\beta}} \Lip_\sigma^{\frac{4}{2-\beta}} \le \mu_1 \Big\}, \\
    \lambda_1 & \coloneqq \inf \Big\{\lambda > 0: \overline{c} \lambda^2 \lip_\sigma^2 + \widetilde{c} \lambda^{\frac{4}{2-\beta}} \lip_\sigma^{\frac{4}{2-\beta}} \ge \mu_1 \Big\}.
  \end{align*}
\end{theorem}

As in \cite{foondun.nualart:15:on, nualart:18:moment, guerngar.nane:20:moment},
it is not clear whether the solution is intermittent if $\lambda \in [\lambda_0,
\lambda_1]$ in Theorems \ref{Thm:pm-d} and \ref{Thm:pm-d2} above. Instead, we
propose the following conjecture for future investigation.
\begin{conjecture}
  Under the settings of either Theorem~\ref{Thm:pm-d} or
  Theorem~\ref{Thm:pm-d2}, there exists $\lambda^*\in [\lambda_0, \lambda_1]$
  such that when
  $\lambda > \lambda^*$, the solution $u(t,x)$ to~\eqref{E:she-d} is fully intermittent; when $\lambda < \lambda^*$, the solution has all $p$-th moments ($p\ge 2$) bounded in time and is not fully intermittent.
\end{conjecture}

\begin{theorem}\label{Thm:pm-n}
  Let $\dom \subset \R^d$ be a convex bounded Lipschitz domain. Let $u$ be the
  solution to~\eqref{E:she-n} with Neumann boundary condition.
  Suppose~\eqref{E:Lsigma} holds. Suppose $\nu \ge 0$ and $\nu(\dom) < \infty$.
  Then, there exist positive finite constants $C$, $c$ and $c'$ such that for
  all $p \ge 2$, $\lambda > 0$, $t > 0$, $x \in \dom$,
  \begin{equation}\label{pm-n-ub}
    \E(|u(t, x)|^p)
    \le C^p \big(J_{c_3}(t, x)\big)^p e^{pt\Big({c p \lambda^2 \Lip_\sigma^2 + c' \, p^{\frac{2}{2-\beta}} \, \lambda^{\frac{4}{2-\beta}}\Lip_\sigma^{\frac{4}{2-\beta}}}\Big)}.
  \end{equation}
  Moreover, if~\eqref{E:Anderson} or~\eqref{E:lsigma} holds, then there exist
  positive finite constants $\overline{C}$, $\overline{c}$ and $\widetilde{c}$ such
  that for all $p \ge 2$, $\lambda > 0$, $t > 0$, $x \in \dom$,
  \begin{equation}\label{pm-n-lb}
    \E(|u(t, x)|^p)
    \ge \overline{C}^p \big(J_{12c_4}(t, x)\big)^p e^{pt\Big(\overline{c} \lambda^2 \lip_\sigma^2 +
    \widetilde{c} \lambda^{\frac{4}{2-\beta}}\lip_\sigma^{\frac{4}{2-\beta}}\Big)}.
  \end{equation}
  Consequently, if $0 < \nu(\dom) < \infty$, then $u$ is fully intermittent on
  $\dom$ for all $\lambda > 0$.
\end{theorem}

\begin{proof}
  The proof of~\eqref{pm-n-ub} and~\eqref{pm-n-lb} is similar to that of
  Theorem~\ref{Thm:pm-d}. Finally,~\eqref{pm-n-ub} and~\eqref{pm-n-lb} imply
  that
  \begin{align*}
    \liminf_{t \to \infty} \frac 1t \log\E(|u(t, x)|^2) & \ge 2\Big(\overline{c} \lambda^2 \lip_\sigma^2 + \widetilde{c} \lambda^{\frac{4}{2-\beta}} \lip_\sigma^{\frac{4}{2-\beta}}\Big) > 0 \quad \shortintertext{and}
    \limsup_{t \to \infty} \frac 1t \log\E(|u(t, x)|^p) & \le p\Big({c p \lambda^2 \Lip_\sigma^2 + c' \, p^{\frac{2}{2-\beta}} \, \lambda^{\frac{4}{2-\beta}}\Lip_\sigma^{\frac{4}{2-\beta}}}\Big) < \infty
  \end{align*}
  for all $p \ge 2$ and all $\lambda > 0$. Hence, $u$ is fully intermittent for
  all $\lambda > 0$.
\end{proof}

\subsection{\texorpdfstring{$L^2 (\dom)$}{L two U}-energy of solution}

Following~\cite{foondun.nualart:15:on, khoshnevisan.kim:15:non-linear,
liu.tian.ea:17:on}, the $L^2$-\emph{energy} of the solution $u$ at time $t > 0$
is defined as
\begin{align*}
   \mathscr{E}_t(\lambda) = \left(\E\int_\dom |u(t, x)|^2 \ud x\right)^{1/2},
\end{align*}
and the \emph{excitation index} (at infinity) is defined as
\begin{align*}
   \lim_{\lambda \to \infty} \frac{\log\log\mathscr E_t(\lambda)}{\log \lambda}
\end{align*}
provided the limit exists. Note that, for $\eps \ge 0$,
\begin{align*}
\mathrm{Vol}(\dom_\eps) \times \inf_{x \in \dom_\eps} \E(|u(t, x)|^2)
\le \mathscr{E}^2_t(\lambda) \le \mathrm{Vol}(\dom) \times \sup_{x \in \dom} \E(|u(t, x)|^2).
\end{align*}

As a result of Theorems~\ref{Thm:pm-d} and~\ref{Thm:pm-n}, we see that the
solution is intermittent for $\lambda$ large under Dirichlet or Neumann boundary
condition, and the energy of the solution behaves like
\begin{align*}
  \mathscr E_t(\lambda) \sim C e^{Ct\lambda^{\frac{4}{2-\beta}}}
  \quad \text{for $\lambda$ large}.
\end{align*}
Under Neumann boundary condition, the solution remains intermittent even
when $\lambda > 0$ is small. However, in this case, the energy of the solution
has a different exponential rate in $\lambda$ than the one above, namely,
\begin{align*}
   \mathscr E_t(\lambda) \sim C e^{Ct\lambda^2}
   \quad \text{for $\lambda>0$ small.}
\end{align*}
In other words, the excitation index ``at zero'' is different. We obtain the
following corollaries:

\begin{corollary}
  Let $u$ be the solution of~\eqref{E:she-d} with Dirichlet boundary condition.
  \begin{enumerate}
    \item If the conditions of Theorem~\ref{Thm:pm-d} hold, then for all $t>0$,
      \begin{align*}
         \lim_{\lambda \to \infty} \frac{\log\log \mathscr{E}_t(\lambda)}{\log\lambda} = \frac{4}{2-\beta}.
      \end{align*}
    \item Moreover, if the conditions of Theorem~\ref{Thm:pm-d2} hold, then for
      all $t>0$,
      \begin{gather*}
        \mathscr{E}_t^*(\lambda) \coloneqq
        \left( \E \int_\dom |u(t, x)|^2 \frac{\ud x}{|\Phi_1(x)|^2} \right)^{1/2} < \infty \shortintertext{and}
        \lim_{\lambda \to \infty} \frac{\log\log\mathscr{E}_t^*(\lambda)}{\log \lambda} = \frac{4}{2-\beta}.
      \end{gather*}
  \end{enumerate}
\end{corollary}

Note that $\mathscr{E}_t(\lambda) \le C \mathscr{E}_t^*(\lambda)$ for some constant $C$.

\begin{corollary}\label{Cor:n}
  Let $u$ be the solution of~\eqref{E:she-n} with Neumann boundary condition and
  the conditions in Theorem~\ref{Thm:pm-n} hold. Then for any $t > 0$,
  \begin{align*}
   \lim_{\lambda \to \infty} \frac{\log\log \mathscr{E}_t(\lambda)}{\log\lambda} = \frac{4}{2-\beta} \quad\text{and}\quad
   \lim_{\lambda \to 0^+} \frac{\log\log \mathscr{E}_t(\lambda)}{\log\lambda} = 2.
  \end{align*}
\end{corollary}

\section{Preliminaries}\label{S:Prelim}

\subsection{Mild solutions}\label{SS:Mild}

Let $\dot{W}$ be a centered and spatially homogeneous Gaussian noise that is
white in time defined on a complete probability space $(\Omega, \mathscr{F},
\P)$ with covariance given in~\eqref{E:CovNoise}.
\begin{remark}
  Let $f$ be any spatially homogeneous correlation correlation function on
  $\R^d$; see~\eqref{E:CovNoise}. When restricted on the domain
  $\left\{x-y:\:x,y\in\dom\right\}$, $f(\cdot)$ is still a nonnegative definite
  function. Indeed, for any test function $\phi$ defined on $\dom$,
  \begin{align*}
    \iint_{\dom^2}  \phi(x) f(x-y) \phi(y) \ud x\ud y
    = \iint_{\R^{2d}} \left(\phi(x) \mathbf{1}_{\dom}(x)\right) f(x-y) \left(\phi(y)\mathbf{1}_{\dom}(y)\right) \ud x\ud y
    \ge 0.
  \end{align*}
\end{remark}

Let $\mathscr{B}(\dom)$ denote the Borel $\sigma$-algebra on
$\dom\subseteq\R^d$. Let $\left\{W_t(A); t \ge 0, A \in
\mathscr{B}(\dom)\right\}$ be the martingale measure associated to the noise
$\dot W$ in the sense of Walsh~\cite{walsh:86:introduction}. Let
$\{\mathcal{F}_t, t\ge 0\} $ be the underlying filtration generated by $W$ and
augmented by the $\sigma$-field $\mathcal{N}$ generated by all $\mathbb{P}$-null
sets in $\mathcal{F}$, namely,
\begin{align*}
  \mathscr{F}_t = \sigma\left\{W_s(A): 0 \le s \le t, A \in \mathscr{B}(\dom) \right\}\vee \mathcal{N}.
\end{align*}

\begin{definition}\label{D:sol}
  A process $u = \{u(t, x); t > 0, x \in \dom\}$ is called \textit{a random field
  solution} to~\eqref{E:she-d} (or~\eqref{E:she-n}, respectively) if:
  \begin{enumerate}
    \item[(i)] $u$ is adapted, i.e., for each $t > 0$ and $x \in \dom$, $u(t,
      x)$ is $\mathscr{F}_t$-measurable;
    \item[(ii)] $u$ is jointly measurable with respect to $\mathscr{B}((0,
      \infty) \times \dom) \times \mathscr{F}$;
    \item[(iii)] for each $t > 0$ and $x \in \dom$,
      \begin{align}\label{D:sol-L2}
        \E \left(\int_0^t \ud s \iint_{\dom^2}\ud y \, \ud y'\,
        G(t-s,x,y) \sigma(s, y, u(s, y))f(y-y') G(t-s, x, y') \sigma(s ,y', u(s, y'))\right) < \infty;
      \end{align}
    \item[(iv)]
    %$(t, x) \mapsto I(t, x)$ is a continuous map from $(0, \infty)\times \dom$
      %to $L^2(\Omega)$;
    $u$ satisfies
    \begin{equation}\label{E:mild-sol}
      u(t, x) = J(t,x) + \lambda \int_0^t \int_\dom G(t-s, x, y) \sigma(s, y, u(s, y)) W(\ud s, \ud y) \quad \text{a.s.}
    \end{equation}
    for each $t > 0$ and $x \in \dom$, where $G = G_D$ (or $G = G_N$, respectively),
    and $J(t,x)$ is the solution to the homogeneous equation, namely,
    \begin{align}\label{E:J0}
       J(t,x)\coloneqq \int_\dom G(t, x, y) \nu(\ud y).
    \end{align}
  \end{enumerate}
\end{definition}

Note that in (iii) above, the condition~\eqref{D:sol-L2} ensures that the
\textit{Walsh stochastic integral}
\begin{align*}
  \int_0^t \int_\dom G(t-s, x, y) \sigma(s, y, u(s, y)) W(\ud s, \ud y)
\end{align*}
is well-defined and the square of its $\|\cdot\|_2$-norm is equal to the
expression in~\eqref{D:sol-L2}.

\subsection{Regularities and geometric properties of the domain}\label{SS:domain}

The definition of the Lipschitz domain is standard; see
e.g. Section 1.2.1 of~\cite{grisvard:85:elliptic}.

\begin{definition}\label{D:LipDomain}
  A bounded domain $\dom \subset \R^d$ is called a \emph{Lipschitz domain}
  if there exist positive constants $K_\dom$ and $r_0$ such that for every $q
  \in \partial \dom$, there exist a Lipschitz function $F_q: \R^{d-1} \to \R$
  satisfying $|F_q(x)-F_q(\overline{x})| \le K_\dom |x- \overline{x}|$ for all
  $x, \overline{x} \in \R^{d-1}$ and an orthonormal coordinate system with
  origin $q$ such that if $z = (x, y)$, $x \in \R^{d-1}, y \in \R$, in this
  coordinate system, then
  \begin{align*}
    \dom \cap B(q, r_0)          & = B(q, r_0) \cap \left\{z = (x, y): y > F_q(x) \right\} \shortintertext{and}
    \partial \dom \cap B(q, r_0) & = B(q, r_0) \cap \left\{z = (x, y): y = F_q(x) \right\},
  \end{align*}
  where $B(q,r_0)$ is the open ball centered at $q$ of radius $r_0$. We call
  $K_\dom$ the \textit{Lipschitz constant} of $\dom$ and $r_0$ the
  \textit{localization radius} of $\dom$.
\end{definition}

\begin{definition}[Definition 2.4.1 of~\cite{henrot.pierre:05:variation}]\label{D:delta-cone}
  For $\delta > 0$, we say that $\dom$ has the \emph{$\delta$-cone property} if,
  for every $x \in \partial \dom$, there exists a unit vector $\xi_x \in \R^d$
  such that for all $y \in \overline \dom \cap B(x, \delta)$, we have
  $\mathscr{C}(y, \xi_x, \delta) \subset \dom$, where $\mathscr{C}(y, \xi,
  \delta)$ is the \emph{$\delta$-cone} defined by
  \begin{align}\label{E:Cone}
    \mathscr{C}(y, \xi, \delta) \coloneqq \left\{
      z \in \R^d:
        0 < |z-y| < \delta \text{~and~}
        (z-y)\cdot \xi \ge |z-y| \cos(\delta)
    \right\}.
  \end{align}
  See Fig.~\ref{F:delta-cone} for an illustration.
\end{definition}

\begin{figure}[htpb]
  \centering
  \begin{center}
    \begin{tikzpicture}[scale=1, samples = 200]
      \tikzset{>=latex}
      \node[name path = circle] (dom) at (2,1) {$\dom$};
      \draw[domain = 0:2, thick] plot (\x, {sqrt(\x)});
      \draw[domain = 0:2, thick] plot (\x, {-sqrt(\x)});
      % \filldraw[gray!30!white, name intersections={of = upper and circle}]
      \draw[fill = gray!30!white, very thick]
        plot [domain = 0:0.59] (\x, {sqrt(\x)}) --
        plot [domain = 50:-50] ({cos(\x)}, {sin(\x)}) --
        plot [domain = 0.6:0] (\x, {-sqrt(\x)}) --
        cycle;
      \filldraw[thick]  (0,0) circle (0.05) node[anchor = north east] {$x$};
      \draw[dashed, thick]  (0,0) circle (1);

      \begin{scope}[rotate = -20]
        \draw[<->, dashed, shorten >= 2pt, shorten <= 2pt] (0,0) --++ (-1,0) node [midway, above, sloped] {$\delta$};
      \end{scope}

      \begin{scope}[rotate = 5]
        \draw[->, thin] (0,0) -- (2,0) node [right] {$\xi_x$};
      \end{scope}

      \begin{scope}[xshift = 1.8em, yshift= -1.4em, rotate = 5]
        \draw[fill = gray!80!white, thick]
          (0,0) --
          plot [domain = 15:-15] ({cos(\x)}, {sin(\x)}) node [right] {$\mathscr{C}(y, \xi_x, \delta)$} --
          (0,0);
        \draw[->, thin] (0,0) -- (2,0);
        \filldraw[thick] (0,0) circle (0.05) node[above] {$y$};
      \end{scope}

    \end{tikzpicture}

  \end{center}
  \caption{Illustration for the $\delta$-cone property in Definition~\ref{D:delta-cone}.}
  \label{F:delta-cone}
\end{figure}
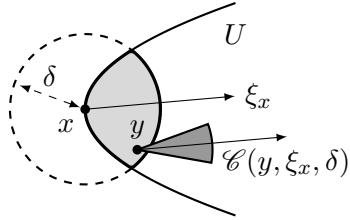

It is clear that if $\dom$ has the $\delta$-cone property, then it also has the
$\delta'$-cone property for $0 < \delta' \le \delta$. It is known that for a
bounded domain $\dom\subset\R^d$, it is a Lipschitz domain if and only if it has
the $\delta$-cone property for some $\delta>0$; see, e.g.,~\cite[Theorem
2.4.7]{henrot.pierre:05:variation} or~\cite[Theorem
1.2.2.2]{grisvard:85:elliptic}. In particular, according to the above
definitions, one can easily see that if $\dom$ is a Lipschitz domain, then it
satisfies the $\delta$-cone condition with
\begin{align}\label{E:delta-Lip}
  \delta = \arctan(1/K_\dom)\wedge r_0.
\end{align}
The $\delta$-cone property gives us a convenient way to handle the Lipschitz
domain.

\begin{example}
  Any $C^{1,\alpha}$-domain with $\alpha>0$ is a Lipschitz domain; see,
  e.g.,~\cite{ouhabaz.wang:07:sharp}. The unit ball in $\R^d$ is a smooth domain
  and also a $C^{1,\alpha}$-domain. $\dom = (-1,1)^d$ is a Lipschitz domain but
  not a $C^{1,\alpha}$-domain for any $\alpha>0$. Domains with cusps are not
  Lipschitz domain; see Fig.~\ref{SF:Cups}.
\end{example}

\subsection{Heat kernel estimates}\label{SS:Kernel}

Recall that $\mathscr{L}$ is the operator defined in divergence form \eqref{E:op-L} satisfying the uniformly elliptic condition \eqref{E:u-ellip}.
It is known that the operator $\mathscr{L}$ with Dirichlet boundary condition on
$\dom$ has a discrete spectrum with a sequence of positive eigenvalues $0 <
\mu_1 \le \mu_2 \le \dots$ such that the first eigenvalue $\mu_1$ is simple and
its eigenfunction $\Phi_1$ can be chosen to be positive and $\|\Phi_1\|_{L^2
(\dom)} = 1$; see, e.g.,~\cite{davies:90:heat}.
Moreover, we have the following heat kernel estimates under Dirichlet and Neumann boundary conditions.

\begin{proposition}\label{P:G}
  Let $\dom\subset \R^d$ be a bounded Lipschitz domain. Then the following estimates hold.
  \begin{enumerate}
  \item[(i)] Dirichlet heat kernel estimates: There exist positive finite
    constants $c_1, c_2, C_1, C_2$ and $0 < a_2 \le 1 \le a_1$ such
    that for all $t>0$ and $x,y\in \dom$,
    \begin{align}\label{E:G-d}
    \begin{gathered}
      C_2 \left(1 \wedge \frac{\Phi_1(x)}{1 \wedge t^{a_2/2}}\right) \left(1 \wedge \frac{\Phi_1(y)}{1 \wedge t^{a_2/2}}\right) \frac{e^{-\mu_1 t}}{1 \wedge t^{d/2}}e^{-c_2 \frac{|x-y|^2}{t}} \\
      \le G_D(t, x, y) \le                                                                                                                                                                      \\
      C_1 \left(1 \wedge \frac{\Phi_1(x)}{1 \wedge t^{a_1/2}}\right) \left(1 \wedge \frac{\Phi_1(y)}{1 \wedge t^{a_1/2}}\right) \frac{e^{-\mu_1 t}}{1 \wedge t^{d/2}}e^{-c_1 \frac{|x-y|^2}{t}}.
    \end{gathered}
    \end{align}
    Moreover, if $\dom$ is a bounded $C^{1, \alpha}$-domain with $\alpha > 0$,
    then~\eqref{E:G-d} holds with $a_1 = a_2 = 1$.
  \item[(ii)] Neumann heat kernel estimates: there exist positive finite
    constants $c_3$ and $C_3$ such that for all $t > 0$ and $x, y \in \dom$,
    \begin{align}\label{E:G-n}
      G_N(t, x, y) \le C_3 \frac{1}{1 \wedge t^{d/2}} e^{-c_3 \frac{|x-y|^2}{t}}.
    \end{align}
    In addition, if $\dom$ is a smooth convex domain and $\mathscr{L} =
    -\Delta$, then there exist positive finite constants $c_4$ and $C_4$ such
    that for all $t > 0$ and $x, y \in \dom$,
    \begin{align}\label{E:G-n-lower}
      G_N(t, x, y) \ge C_4 \frac{1}{1 \wedge t^{d/2}} e^{-c_4\frac{|x-y|^2}{t}}.
    \end{align}
  \end{enumerate}
\end{proposition}

\begin{proof}
  The Dirichlet case is proved in~\cite{riahi:13:estimates} (see Theorem 2.1 and
  Remark 1 on p.123). For the Neumann case, the upper bound in~\eqref{E:G-n} can
  be found in Theorem 3.2.9 of~\cite{davies:90:heat}, where we note that the
  \textit{extension property} referred in that theorem (\textit{ibid.}) is satisfied
  by the Lipschitz domain (see either Proposition 1.7.9 of~\cite{davies:90:heat}
  or Theorem 1.4.3.1 of~\cite{grisvard:85:elliptic}). The lower bound in~\eqref{E:G-n-lower} follows from~\cite[Theorem 3.1 and Examples
  3.3]{saloff-coste:10:heat}; see also~\cite{saloff-coste:92:note}.
\end{proof}

\begin{remark}\label{R:Phi1Bdd}
  In the Dirichlet boundary condition case, by~\cite[(1.2)]{riahi:13:estimates},
  there exists a finite constant $c_0 > 1$ such that for all $z \in \dom$,
  \begin{align}\label{E:Phi-bd}
    c_0^{-1} \left[\op{dist}(z, \partial \dom)\right]^{a_1} \le \Phi_1(z) \le
    c_0      \left[\op{dist}(z, \partial \dom)\right]^{a_2},
  \end{align}
  where $a_1$ and $a_2$ are constants from part (i) of Proposition~\ref{P:G}.
\end{remark}

\begin{remark}\label{R:HeatKernel}
  In general, Theorem 3.1 of~\cite{saloff-coste:10:heat} states that the
  following conditions are equivalent:
  \begin{itemize}
    \item The two-sided bound~\eqref{E:G-n} and~\eqref{E:G-n-lower} holds for
      the Neumann heat kernel on $\dom$, that is, for all $t > 0$ and $x, y \in
      \dom$,
      \begin{align}\label{E:G-n-bd}
        C_4 \frac{1}{1 \wedge t^{d/2}} e^{-c_4\frac{|x-y|^2}{t}}
        \le G_N(t, x, y)
        \le C_3 \frac{1}{1 \wedge t^{d/2}} e^{-c_3 \frac{|x-y|^2}{t}};
      \end{align}
    \item The parabolic Harnack inequality holds;
    \item The domain $\dom$ has the volume doubling property and the Poincar\'e
      inequality holds.
  \end{itemize}
  The results of the present paper under the Neumann boundary condition,
  especially the lower bound results, remain valid for $\dom$ satisfying any one
  of the above equivalent conditions.
\end{remark}

In the rest of the paper, the lower case constants $c_1, c_2, c_3, c_4$ are
reserved for the constants given by Proposition~\ref{P:G} above.

Before the end of this subsection, we prove that the Lipschitz domain $\dom$
satisfies the lower bound in~\eqref{E:RegSet} below. The following lemma may
well be buried in the literature. Since its proof is short, it will be given below. Let $\op{Vol}(A)$ denote the volume (or $d$-dimensional Lebesgue measure)
of any measurable set $A$ in $\R^d$.

\begin{lemma}\label{Lem:vol}
  Suppose that $\dom$ is a bounded Lipschitz domain in $\R^d$. Then, there
  exists positive finite constants $C$ and $C'$ such that for all $y \in
  \overline{\dom}$ and $r > 0$,
  \begin{align}\label{E:RegSet}
    C(1 \wedge r)^d \le \op{Vol}(\dom \cap B(y, r)) \le C' (1\wedge r)^d.
  \end{align}
\end{lemma}
\begin{proof}
The upper bound is trivial since $\mathrm{Vol}(\dom \cap B(y, r)) \le
\mathrm{Vol}(\dom) \wedge \mathrm{Vol}(B(y, r))$. We only need to prove the
lower bound. It is known that the Lipschitz domain $\dom$ satisfies the
$\delta$-cone property with $\delta$ given in~\eqref{E:delta-Lip}.

We first consider the case of $0 < r < \delta$. If $\mathrm{dist}(y, \partial
\dom) > r$, then $B(y, r) \subset \dom$ and
\begin{align*}
  \mathrm{Vol}(\dom \cap B(y, r))
  = \mathrm{Vol}(B(y, r))
  = V_d\: r^d
  \ge V_d (1 \wedge r)^d,
\end{align*}
where $V_d = \pi^{d/2}/\Gamma(d/2+1)$. If $\mathrm{dist}(y, \partial \dom) \le
r$, then $y\in B(x, \delta)$ for some $x \in \partial \dom$, and by the
$\delta$-cone property, we can find a unit vector $\xi = \xi_x \in \R^d$ such
that $\mathscr{C} (y, \xi, \delta) \subset \dom$. It follows that
\begin{align*}
  \mathrm{Vol}(\dom \cap B(y, r))
  & \ge \mathrm{Vol}(\mathscr{C} (y, \xi, \delta) \cap B(y, r))                                    \\
  & = \op{Vol}\{ z \in \R^d : 0 < |z-y| < r \text{ and } (z-y) \cdot \xi \ge  |z-y|\cos\delta \} \\
  & = C_{d,\delta} V_d \: r^d,
\end{align*}
where $C_{\delta,d}\in (0,1]$. Therefore, when $r\in (0,\delta)$,
$\mathrm{Vol}(\dom \cap B(y, r)) \ge C_{d,\delta} V_d\: \left(1\wedge
r\right)^d$. Finally, the case of $r \ge \delta$ follows from the previous case
because
\begin{align*}
  \mathrm{Vol}(\dom \cap B(y, r))
  \ge \mathrm{Vol}(\dom \cap B(y, \delta))
  \ge C_{d,\delta} V_d\: \delta^d
  \ge C_{d,\delta} V_d\:\delta^d (1\wedge r)^d.
\end{align*}
This implies the desired lower bound with $C = C_{d,\delta} V_d\:\left(1\wedge
\delta\right)^d$.
\end{proof}

Next, we need to replace $\dom$ in~\eqref{E:RegSet} above by a subset of $\dom$
with some specific properties. Take $0 < \eps_1 < 1$ such that $\dom_{\eps_1}\ne
\varnothing$ (see~\eqref{E:Ue}) and let
\begin{equation}\label{Def:eps_0}
  \eps_0 = \eps_1 \wedge (\delta/2).
\end{equation}
By the compactness of $\dom$, find and fix a finite collection of open balls $\{
B(y_i, \eps_0) \}_{i = 1}^m$ with centers $y_1, \dots, y_m \in \partial \dom$ such
that $\bigcup_{i = 1}^m B(y_i, \eps_0) \supset \overline\dom \setminus
\dom_{\eps_0/2}$. Then, for any $x \in \dom$, we have one of the following
cases:
\begin{enumerate}
  \item[(1)] If $x \in \bigcup_{i = 1}^m \overline{B(y_i, \eps_0)}$, choose
    the smallest $i$ such that $\overline{B(y_i, \eps_0)} \ni x$. Then, by the
    $\delta$-cone property of $\dom$, we have $\overline{\mathscr{C}(x,
    \xi_{y_i}, \eps_0/2)} \subset \dom$.
  \item[(2)] If $x \not\in \bigcup_{i = 1}^m \overline{B(y_i, \eps_0)}$, then
    $\op{dist}(x, \partial \dom) > \eps_0/2$ and thus $\overline{B(x, \eps_0/2)}
    \subset \dom$.
\end{enumerate}
Accordingly, we define (see Fig.~\ref{F:V(x)} for an illustration)
\begin{align}\label{E:V(x)}
  V(x) = \begin{dcases}
    \overline{\mathscr{C}(x, \xi_{y_i}, \eps_0/2)} & \text{in case (1)}, \\
    \overline{B(x, \eps_0/2)}                      & \text{in case (2)}.
  \end{dcases}
\end{align}

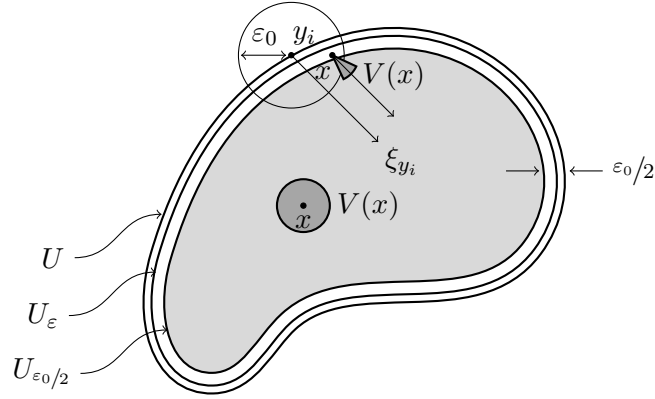
\begin{figure}[htpb]
  \centering
  \begin{tikzpicture}[ x = 3em, y = 3em, samples = 150]

    \draw [use Hobby shortcut,
      % fill = gray!50!white,
      thick,
      scale = 1.4,
      double,
      double distance = 10pt,
      ] ([closed]-1,0) .. (0,1.5)  ..  (2,0.1)  ..  (0,-0.5)  ..  (-0.6,-1);
    % \draw[use Hobby shortcut, fill = gray!50!white, thick, scale = 1.2] ([closed]-1,0) .. (0,1.5)  ..  (2,0.1)  ..  (0,-0.5)  ..  (-0.6,-1);
    \draw [use Hobby shortcut,
      fill = gray!30!white,
      thick,
      scale = 1.4,
      double,
      double distance = 4pt,
      ] ([closed]-1,0) .. (0,1.5)  ..  (2,0.1)  ..  (0,-0.5)  ..  (-0.6,-1);

    \begin{scope}[yshift = 7em, xshift = 2pt]
      \draw (0,0) circle (20pt);
      \filldraw (0,0) circle (1pt) node [above, xshift = 5pt] {$y_i$};
      \draw[<->, shorten >= 1pt, shorten <= 2pt] (0,0) -- (-20pt,0) node [midway, above] {$\eps_0$};
      \draw[->] (0,0) -- (1,-1) node[below, xshift = 9pt, yshift = 2pt] {$\xi_{y_i}$};
    \end{scope}

    \begin{scope}[yshift = 3em, xshift = 7.6em]
      \draw[->] (0,0) --++ (0.38,0);
      \draw[<-,xshift = -3.2em] (1.8,0) --++ (0.38,0) node[right] {$\sfrac{\eps_0}{2}$};
    \end{scope}

    \begin{scope}[xshift = -7.3em]
      \draw[->] (0,0) node[left] {$\dom$} .. controls (0.5,0) and (0.5, 0.5) .. (1,0.5);
    \end{scope}

    \begin{scope}[xshift = -7.5em,yshift = -2em]
      \draw[->] (0,0) node[left] {$\dom_{\eps}$} .. controls (0.5,0) and (0.5, 0.5) .. (1,0.5);
    \end{scope}

    \begin{scope}[xshift = -7.1em,yshift = -4em]
      \draw[->] (0,0) node[left] {$\dom_{\sfrac{\eps_0}{2}}$} .. controls (0.5,0) and (0.5, 0.5) .. (1,0.5);
    \end{scope}

    \begin{scope}[xshift = 0.6em, yshift = 1.8em]
      \draw[->, fill = gray!70!white, thick] (0,0) circle (10pt) node [right, xshift = 0.8em] {$V(x)$};
      \filldraw (0,0) circle (1pt) node [below] {$x$};
    \end{scope}

    \begin{scope}[xshift = 1.6em, yshift = 7em, rotate = -45]
      \draw[fill = gray!70!white, thick]
        plot [domain = 20:-20] ({10pt*cos(\x)}, {10pt*sin(\x)}) --
        (0,0) --
        cycle;
      \draw[->] (0,0) -- (1,0);
      \node at (2em,1em) {$V(x)$};
      \filldraw (0,0) circle (1pt) node [left, yshift = -6.5pt, xshift = 3.5pt] {$x$};
    \end{scope}

  \end{tikzpicture}

  \caption{Illustration for the two cases of $V(x)$ in~\eqref{E:V(x)}.}
  \label{F:V(x)}
\end{figure}

The next lemma will be used later together with the heat kernel
estimates to derive the lower bounds in Lemmas~\ref{Lem:iintGGf-d}
and~\ref{Lem:iintGGf-n} below.

\begin{lemma}\label{Lem:vol2}
  Let $\dom \subset \R^d$ be a bounded Lipschitz domain with the $\delta$-cone
  property. Let $\eps_0\in (0,1)$ and $V(x) \subset \dom$, for $x \in \dom$, be
  defined by~\eqref{Def:eps_0} and~\eqref{E:V(x)} above. Then, for each $\eps \in
  (0, \eps_0]$ there exists $c_\eps > 0$ with $\lim_{\eps \to 0} c_\eps = 0$
  such that for all $x \in \dom_\eps$, we have
  \begin{gather}\label{dist-V(x)}
    d(z)\coloneqq \op{dist}(z, \partial \dom) \ge c_\eps \quad \text{for all $z \in V(x)$} \shortintertext{and}
    \op{Vol}(V(x) \cap B(x, r)) \ge C (1 \wedge r)^d \quad \text{for all $r > 0$},
    \label{vol-V(x)}
  \end{gather}
  where $C>0$ is a constant depending on $d$ and $\eps_0$.
\end{lemma}

\begin{proof}
  Let $\eps \in (0, \eps_0]$. We first prove~\eqref{dist-V(x)}. On the one hand,
  for each $i$ and $x \in \overline{\dom_\eps} \cap \overline{B(y_i,
  \eps_0/2)}$, by the $\delta$-cone property of $\dom$, we have
  $\overline{\mathscr{C}(x, \xi_{y_i}, \eps_0/2)} \subset \dom$, which implies
  that
  \begin{align*}
    \op{dist}\left(\overline{\mathscr{C}(x, \xi_{y_i}, \eps_0/2)}, \partial \dom\right)
    \coloneqq \inf\left\{|z - y|: z \in \overline{\mathscr{C}(x, \xi_{y_i}, \eps_0/2)}, \, y \in \partial \dom\right\} > 0.
  \end{align*}
  Since the function $x \mapsto \op{dist}(\overline{\mathscr{C}(x, \xi_{y_i},
    \eps_0/2)}, \partial \dom)$ is continuous on the compact set
    $\overline{\dom_\eps} \cap \overline{B(y_i, \eps_0/2)}$, we can find
    $c_{\eps, i} > 0$ such that $\op{dist}(\overline{\mathscr{C}(x, \xi_{y_i},
    \eps_0/2)}, \partial \dom) \ge c_{\eps, i}$ for all $x \in
    \overline{\dom_\eps} \cap \overline{B(y_i, \eps_0/2)}$.
    \medskip

  On the other hand, for $x \in \overline{\dom_\eps} \setminus \bigcup_{i = 1}^m
  {B(y_i, \eps_0/2)}$, we have $\overline{B(x, \eps_0/2)} \subset \dom$, and
  hence
  \begin{align*}
    \op{dist}(\overline{B(x, \eps_0/2)}, \partial \dom) > 0.
  \end{align*}
  Then, by the continuity of $x \mapsto \op{dist}(\overline{B(x, \eps_0/2)},
  \partial \dom)$ and the compactness of $\overline{\dom_\eps} \setminus
  \bigcup_{i = 1}^m {B(y_i, \eps_0/2)}$, we can find $c_{\eps, 0} > 0$ such that
  $\op{dist}(\overline{B(x, \eps_0/2)}, \partial \dom) \ge c_{\eps, 0}$ for all
  $x \in \overline{\dom_\eps} \setminus \bigcup_{i = 1}^m {B(y_i, \eps_0/2)}$.
  \medskip

  Therefore, by taking $c_\eps = \min\{c_{\eps, i}: 0 \le i \le m\}$, we get
  that $\op{dist}(V(x), \partial \dom) \ge c_\eps$. Also, note that for $1 \le i
  \le m$, we have $0 < c_{\eps, i} \le \eps$, so $c_\eps \to 0$ as $\eps \to 0$.
  This proves~\eqref{dist-V(x)}. \medskip

  As for~\eqref{vol-V(x)}, if $V(x) = \overline{\mathscr{C}(x, \xi_{y_i},
  \eps_0/2)}$ in case (1) of~\eqref{E:V(x)}, then
  \begin{align*}
    \MoveEqLeft[2] \op{Vol}(V(x) \cap B(x, r))
      = \op{Vol}(\overline{\mathscr{C}(x, \xi_{y_i}, \eps_0/2)} \cap B(x, r)) \\
    & = \op{Vol}\{z \in \R^d: 0 \le |z - x| < (\eps_0/2) \wedge r \text{ and } (z-x) \cdot \xi_{y_i} \ge |z-x|\cos(\eps_0/2) \} \\
    & = C_{d, \eps_0}((\eps_0/2) \wedge r)^d
      \ge C_{d, \eps_0} (\eps_0/2)^d (1 \wedge r)^d.
  \end{align*}
  If $V(x) = \overline{B(x, \eps_0/2)}$ in case (2) of~\eqref{E:V(x)}, then
  \begin{align*}
    \op{Vol}(V(x) \cap B(x, r))
    = \op{Vol}(B(x, (\eps_0/2) \wedge r))
    = C_d ((\eps_0/2) \wedge r)^d
    \ge C_d (\eps_0/2)^d (1 \wedge r)^d.
  \end{align*}
  This shows~\eqref{vol-V(x)} and completes the proof of Lemma~\ref{Lem:vol2}.
\end{proof}

\section{Bounded initial condition case}\label{S:Bounded}

In this section, we give some computations to show how the noise interacts
with the differential operator. Let us first give a general definition,
which is not only restricted to the heat equation.

\begin{definition}\label{D:h&K}
  Let $\dom$ be a general domain in $\R^d$ and let $G(t,x,y)$ be the fundamental
  solution to the corresponding partial differential equation. Let $f$ be a
  nonnegative, nonnegative definite function. Let $h_0^{\dom}(t)$ be a
  locally integrable function defined on $\R_+ \coloneqq [0, \infty)$. Define formally the following
  functions:
  \begin{subequations}\label{E:h&K}
  \begin{align}\label{E:K}
    K_\lambda^{\dom}(t) & \coloneqq \sum_{n=0}^{\infty} \lambda^{2n} h_n^{\dom}(t) \shortintertext{where}
    h_n^{\dom}(t)       & \coloneqq \left(k^{\dom} * h_{n-1}^{\dom}\right)(t) \quad \text{for $n\ge 1$ and}, \label{E:hn}   \\
    k^{\dom}(t)         & \coloneqq \sup_{x,x'\in \dom}\iint_{\dom^2} G(t, x, y) G(t, x', y') f(y-y')\, \ud y\, \ud y'. \label{E:kO}
  \end{align}
  \end{subequations}
\end{definition}
These functions depend on the fundamental solution
$G$. When it is clear from the context, the superscript $\dom$ will be omitted.

In the above, ``$*$" is the standard convolution in the time variable:
  \begin{align*}
    h * k (t) = \int_0^t h(t-s) k(s) \ud s.
  \end{align*}
  For $n \ge 1$, we will also denote by $k^{* n}$ the $n$-th convolution power of $k$, i.e., $k^{* 1} = k$ and
\begin{equation*}
 k^{* (n+1)}(t) = \int_0^t k^{* n}(t-s) k(s) \ud s.
\end{equation*}

\begin{remark}
  In~\cite{chen.kim:19:nonlinear} and~\cite{chen.huang:19:comparison}, the
  kernel function $k(t)$ is defined as
  \begin{align*}
    k(t) = \int_{\R^d} f(z)G(t,z) \ud z,
  \end{align*}
  where $G(t,x)$ is the heat kernel on $\R^d$. This is consistent
  to~\eqref{E:kO} (up to a factor of $2$):
  \begin{align*}
    k^{\R^d}(t) & = \sup_{x,x'\in\R^d}\iint_{\R^{2d}} G(t,x-y)G(t,x'-y') f(y-y') \ud y \ud y' \\
                & = \sup_{x,x'\in\R^d}\iint_{\R^{2d}} G(t,x-y)G(t,x'-y+z) f(z) \ud z \ud y \\
                & = \sup_{x,x'\in\R^d}\int_{\R^{d}} G(2t,x-x'-z) f(z) \ud z
                  = \int_{\R^{d}} G(2t,z) f(z) \ud z.
  \end{align*}
\end{remark}

The following two lemmas provide estimates for $k^\dom(t)$ in the case of the
Dirichlet and Neumann heat kernel, respectively. In particular, for the case of
$\mathscr{L} = -\Delta$, our lower and upper bounds generalize (3.3) and (3.5)
of~\cite{nualart:18:moment} from $\dom$ being an open ball to more general domains.

\begin{lemma}\label{Lem:iint-d}
  If $\dom$ is a bounded Lipschitz domain, then we have the
  following integral estimates for the Dirichlet heat kernel:
  \begin{enumerate}
    \item[(i)] There exists a positive finite constant $C$ such that for all
      $t>0$, for all $x, x' \in \dom$,
      \begin{equation}\label{E:iint-d-ub}
        \iint_{\dom \times \dom} G_D(t, x, y) G_D(t, x', y') f(y-y')\, \ud y\, \ud y'
        %k(t)
        \le C e^{-2\mu_1 t} (1 \wedge t)^{-\beta/2}.
      \end{equation}
    \item[(ii)] There exists $0 < \eps_0 < 1$ such that for any $0 < \eps \le
      \eps_0$, there exists a positive finite constant $C_\eps$ such that for
      all $t > 0$, for all $x, x' \in \dom_\eps$ with $|x-x'| \le \sqrt t$,
      \begin{equation}\label{E:iint-d-lb}
        \iint_{\dom \times \dom} G_D(t, x, y) G_D(t, x', y') f(y-y')\, \ud y\, \ud y'
        \ge C_\eps e^{-2\mu_1 t} (1\wedge t)^{-\beta/2}.
      \end{equation}
  \end{enumerate}
\end{lemma}
\begin{proof}
  (i) Let $I$ be the left-hand side of~\eqref{E:iint-d-ub}, By~\eqref{E:f-bd}
  and the upper bound in~\eqref{E:G-d},
  \begin{align*}
  I & \le C e^{-2\mu_1 t}\bigg(\iint_{\dom^2} |y-y'|^{-\beta}\, \ud y\,\ud y'\, \textbf{1}_{\{t \ge 1\}} \\
    & \quad + \iint_{\dom^2} \frac{1}{t^{d/2}}\, e^{-c_1\frac{|x-y|^2}{t}}
      \frac{1}{t^{d/2}}\, e^{-c_1\frac{|x'-y'|^2}{t}} |y-y'|^{-\beta} \,\ud y\,\ud y' \, \textbf{1}_{\{t < 1\}}
    \bigg).
  \end{align*}
  The first integral is a finite constant since $\beta < d$. For the second
  integral, we can enlarge the domain of integration to $\R^d \times \R^d$ and
  use the Plancherel theorem to get the upper bound
  \begin{align*}
    C \int_{\R^d} e^{-i(x-x') \cdot \xi}\,e^{-\frac{2}{c_1}t|\xi|^2} |\xi|^{\beta-d} \ud\xi
    \le C \int_{\R^d} e^{-\frac{2}{c_1}t|\xi|^2} |\xi|^{\beta-d} \ud\xi
     = C' t^{-\beta/2}.
  \end{align*}
  The last equality can be obtained by scaling. This proves the upper
  bound~\eqref{E:iint-d-ub}.

  (ii) We now turn to the proof of the lower bound~\eqref{E:iint-d-lb}.
  Let $t > 0$ and $x, x' \in \dom_\eps$ be such that $|x-x'| \le \sqrt{t}$.
  By~\eqref{E:f-bd}, we have
  \begin{align}\label{I_eps-lb}
  \begin{split}
    I \ge C e^{-2\mu_1 t}
    & \bigg(\iint_{\dom_\eps\times \dom_\eps} G_D(t, x, y) G_D(t, x', y') |y-y'|^{-\beta}\, \ud y\, \ud y' \, {\bf 1}_{\{t \ge 1\}} \\
    & + \iint_{\dom\times \dom} G_D(t, x, y) G_D(t, x', y') |y-y'|^{-\beta}\, \ud y\, \ud y' \,{\bf 1}_{\{t < 1\}} \bigg).
  \end{split}
  \end{align}
  We estimate the two integrals separately. Since $\dom$ is bounded, $\sup_{y,
  y' \in \dom} |y - y'| \le M < \infty$. By the lower bounds in~\eqref{E:G-d}
  and~\eqref{E:Phi-bd}, the first integral in~\eqref{I_eps-lb} is bounded below by
  \begin{align*}
    C_2^2 (1 \wedge (c_0^{-1} \eps^{a_1}))^4 e^{-2c_2M^2} M^{-\beta}\big[\op{Vol}(\dom_\eps)\big]^2
    \textbf{1}_{\{t \ge 1\}}
    &= C_\eps {\textbf{1}}_{\{t \ge 1\}}.
  \end{align*}
  By the lower bound in~\eqref{E:G-d}, the second integral
  in~\eqref{I_eps-lb} is bounded below by
  \begin{align}\label{I_eps-lb2}
  \begin{split}
  &\iint_{V(x) \times V(x')}
  C_2^2 (1 \wedge \Phi_1(x)) (1 \wedge \Phi_1(y)) (1 \wedge \Phi_1(x')) (1 \wedge \Phi_1(y'))\\
  &\quad \times \frac{1}{t^{d}} e^{-2c_2}
  \textbf{1}_{\{|x-y| \le \sqrt t,\, |x'-y'| \le \sqrt{t}\}}|y-y'|^{-\beta} \,\ud y\,\ud y'\, \textbf{1}_{\{t < 1\}}.
  \end{split}
  \end{align}
  Since $|x-x'| \le \sqrt{t}$, we have $|y-y'| \le 3\sqrt{t}$ on the set
  $\{|x-y| \le \sqrt t, |x'-y'| \le \sqrt{t}\}$. By~\eqref{E:Phi-bd} and
  Lemma~\ref{Lem:vol2}, for $0 < \eps \le \eps_0$,~\eqref{I_eps-lb2} is bounded
  below by
  \begin{align*}
    & C_2^2 (1 \wedge (c_0^{-1}\eps^{a_1}))^2 (1 \wedge (c_0^{-1}c_\eps^{a_1}))^2 t^{-d} e^{-2c_2} (3\sqrt{t})^{-\beta} \\
    & \quad \times \op{Vol}(V(x) \cap B(x, \sqrt{t})) \times \op{Vol}(V(x') \cap B(x', \sqrt{t}))\, \textbf{1}_{\{t < 1\}} \\
    & \ge C_\eps t^{-\beta/2}\,{\textbf{1}}_{\{t < 1\}}.
  \end{align*}
  With this, we complete the proof of Lemma~\ref{Lem:iint-d}.
\end{proof}

\begin{lemma}\label{Lem:iint-n}
  If $\dom$ is a bounded Lipschitz domain, then we have the
  following integral estimates for the Neumann heat kernel:
  \begin{enumerate}
    \item[(i)] There exists a positive finite constant $C_1$ such that for all
      $t>0$, for all $x, x' \in \dom$,
    \begin{equation}\label{iint-n-ub}
      \iint_{\dom^2} G_N(t, x, y) G_N(t, x', y') f(y-y')\, \ud y\, \ud y'
      \le C_1 (1 \wedge t)^{-\beta/2}.
    \end{equation}
    \item[(ii)] If~\eqref{E:G-n-lower} holds, then there exists a positive
      finite constant $C_2$ such that for all $t > 0$, for all $x, x' \in \dom$
      with $|x-x'| \le \sqrt t$,
      \begin{equation}\label{iint-n-lb}
        \iint_{\dom^2} G_N(t, x, y) G_N(t, x', y') f(y-y')\, \ud y\, \ud y'
        \ge C_2 (1\wedge t)^{-\beta/2}.
      \end{equation}
  \end{enumerate}
\end{lemma}
The proof of Lemma~\ref{Lem:iint-n} follows the same strategy as that of
lemma~\ref{Lem:iint-d}. We will leave it to the interested readers.

\bigskip

Lemmas~\ref{Lem:iint-d} and~\ref{Lem:iint-n} suggest the study of the following
functions.
\begin{definition}\label{D:wht}
  Let $\widehat{K}_\lambda(t)$ and $\widehat{h}_n(t)$ be defined as~\eqref{E:K}
  and~\eqref{E:hn}, respectively, except that $h^{\dom}_0(t)$ and $k^{\dom}(t)$
  in~\eqref{E:kO} be replaced, respectively, by
  \begin{align}\label{E:wht-k}
    \widehat{h}_0(t) \equiv 1 \quad \text{and} \quad
    \widehat{k}(t) = (1 \wedge t)^{-\rho} \quad
    \text{with $\rho\in(0,1)$,}
  \end{align}
  namely,
  \begin{align*}
   \widehat{K}_\lambda(t) \coloneqq \sum_{n = 0}^{\infty} \lambda^{2n} \widehat{h}_n(t)
   \quad \text{and} \quad
   \widehat{h}_n(t) \coloneqq \left(\widehat{k}^{* n} * 1\right)(t) \quad \text{for $n \ge 1$}.
  \end{align*}
\end{definition}

Before proceeding to the next lemma, we recall some useful formulas and
inequalities:
\begin{itemize}
  \item From the Beta integral (see, e.g.,~\cite{olver.lozier.ea:10:nist}), we
    have that for all $t > 0$, $n \ge 2$, and $r_0, r_1, \dots, r_n > -1$,
  \begin{equation}\label{E:beta-id}
  \begin{gathered}
     \int_0^t \ud s_1 \int_0^{s_1} \ud s_2 \cdots\!\!\int_0^{s_{n-1}} \ud s_n
      (t-s_1)^{r_0} (s_1-s_2)^{r_1}
      \dots (s_{n-1} - s_n)^{r_{n-1}} s_n^{r_n}\\
     = \frac{\prod_{i = 0}^n\Gamma(1+r_i)}{\Gamma(n+ \sum_{i = 0}^n r_i + 1)} t^{n + \sum_{i = 0}^n r_i}.
  \end{gathered}
  \end{equation}
  \item For any $a > 0$, there exist positive finite constants $C$ and $c$
    depending on $a$ such that
    \begin{equation}\label{Gamma-bd}
      c^n (n!)^a \le \Gamma(an+1) \le C^n (n!)^a,
      \qquad \text{for all integers $n \ge 0$.}
    \end{equation}
  This inequality can be easily verified using Stirling's formula; see (68)
  of~\cite{balan.conus:16:intermittency}.
  \item The following result is proved in~\cite[Lemma
  A.1]{balan.conus:16:intermittency} and~\cite[Lemma
  5.2]{balan.jolis.ea:17:intermittency}: For any $a > 0$, there exist positive
  finite constants $C_1, C_2, C_3, C_4$ depending on $a$ such that for all $x >
  0$,
  \begin{equation}\label{E:exp-bd}
  C_1 \exp(C_2 x^{1/a}) \le \sum_{n = 0}^\infty \frac{x^n}{(n!)^a}
  \le C_3 \exp(C_4 x^{1/a}).
  \end{equation}
\end{itemize}

\begin{lemma}\label{Lem:I}
  Let $\lambda > 0$. Let $\widehat{K}_\lambda(t)$ and $\widehat{h}_n(t)$ be
  defined as in Definition~\ref{D:wht}. Then:
  \begin{enumerate}
    \item[(i)] The function $t \mapsto \widehat{h}_n(t)$ is nondecreasing for
      each $n\ge 0$.
    \item[(ii)] There exist positive finite constants $C_1$ and $C_2$ depending only
      on $\rho$ such that for all $t > 0$, for all integers $n \ge 1$,
      \begin{equation}\label{E:wht-hn}
         ((1-\rho)/2)^n t^n \sum_{k = 0}^n \frac{C_1^k t^{-k\rho}}{(n-k)! (k!)^{1-\rho}}
        \le \widehat{h}_n(t)
        \le t^n \sum_{k = 0}^n \frac{C_2^k t^{-k\rho}}{(n-k)! (k!)^{1-\rho}}.
      \end{equation}
    \item[(iii)] There exist positive finite constants $C_3, \dots, C_{8}$
      depending only on $\rho$ such that for all $t>0$,
      \begin{equation}\label{E:sum-I-bd}
        C_3 \exp\left(t\left(C_4 \lambda^2 + C_5 \lambda^{\frac{2}{1-\rho}}\right)\right)
        \le \widehat{K}_\lambda(t)
        \le C_6 \exp\left(t\left(C_7 \lambda^2 + C_{8} \lambda^{\frac{2}{1-\rho}}\right)\right).
      \end{equation}
    \item[(iv)] For any $p\ge 1$ and $t>0$, it holds that $\sum_{n = 0}^{\infty} \left[
      \lambda^{2n} \widehat{h}_n(t) \right]^{1/p}<\infty$.
  \end{enumerate}
\end{lemma}
\begin{proof}
  (i) Obviously, $\widehat{h}_0$ is nondecreasing. Suppose $\widehat{h}_n$ is
  nondecreasing. Then
  \begin{align*}
    \widehat{h}_{n+1}(t) = \int_0^t (1\wedge (t-s))^{-\rho}\, \widehat{h}_n(s) \ud s
                         = \int_0^t (1\wedge s)^{-\rho}\, \widehat{h}_n(t-s) \ud s,
  \end{align*}
  which is also nondecreasing (see also the proof of Lemma 2.6 of~\cite{chen.kim:19:nonlinear}).

  (ii) By expanding $\widehat{h}_n(t)$ recursively, we see that
  \begin{equation*}
    \widehat{h}_n (t) = \int_0^t \ud s_1 \int_0^{s_1} \ud s_2 \cdots \int_0^{s_{n-1}} \ud s_n
    \prod_{i = 1}^n (1\wedge (s_{i-1} - s_{i}))^{-\rho},
  \end{equation*}
  where $s_0 = t$.
  Note that
  \begin{align}\label{E:Crho}
   \frac{1}{2}(1+x^{-\rho}) \le (1\wedge x)^{-\rho} \le 1 + x^{-\rho}
   \quad \text{for all $x \ge 0$}.
  \end{align}
  So, in order to prove~\eqref{E:wht-hn}, it suffices to prove that
  \begin{equation}\label{E:wht-hn*}
   (1-\rho)^n t^n \sum_{k = 0}^n \frac{C_1^k t^{-k\rho}}{(n-k)! (k!)^{1-\rho}}
   \le h^*_n(t)
   \le t^n \sum_{k = 0}^n \frac{C_2^k t^{-k\rho}}{(n-k)! (k!)^{1-\rho}},
  \end{equation}
  where
  \begin{equation*}
    h^*_n (t) = \int_0^t \ud s_1 \int_0^{s_1} \ud s_2 \cdots \int_0^{s_{n-1}} \ud s_n
    \prod_{i = 1}^n \left(1+ (s_{i-1} - s_{i})^{-\rho}\right).
  \end{equation*}
  We observe that
  \begin{align*}
    \prod_{i = 1}^n (1+ (s_{i-1} - s_i)^{-\rho})
    = \sum_{k = 0}^n\, \sum_{1 \le i_1 < i_2 < \dots < i_k \le n}\,
    \prod_{j = 1}^k (s_{i_j-1} - s_{i_j})^{-\rho},
  \end{align*}
  where we have used the convention that when $k = 0$ the summation and product
  inside gives one. Then, by~\eqref{E:beta-id}, we have
  \begin{align}\label{I_n^*}
    \begin{split}
      h_n^* (t)
      & = \sum_{k = 0}^n\, \sum_{1 \le i_1 < i_2 < \dots < i_k \le n}
          \int_0^t \ud s_1 \int_0^{s_1} \ud s_2 \dots \int_0^{s_{n-1}} \ud s_n \,
          \prod_{j = 1}^k \left(s_{i_j-1} - s_{i_j}\right)^{-\rho}\\
      & = \sum_{k = 0}^n \binom{n}{k}
          \frac{\Gamma(1-\rho)^k}{\Gamma(n-k\rho+1)}t^{n-k\rho}.
    \end{split}
  \end{align}
  By the recursion identity of the Gamma function $\Gamma(z+1) = z\Gamma(z)$, we
  have
  \begin{align*}
    \Gamma(n-k\rho+1)
    = \Gamma\big(k(1-\rho)+1\big) \prod_{j = 1}^{n-k} \big(k(1-\rho) + j\big).
  \end{align*}
  Because $0 < \rho < 1$, for all $0 \le j \le n$, it holds that $(1-\rho)(k+j)
  < k(1-\rho)+j < k+j$ and hence,
  \begin{align*}
    \left(1-\rho\right)^{n-k} \frac{n!}{k!}\: \Gamma\left(k(1-\rho)+1\right)
    \le \Gamma\left(n-k\rho+1\right) \le
    \frac{n!}{k!} \: \Gamma\left(k(1-\rho)+1\right).
  \end{align*}
  Also, by~\eqref{Gamma-bd}, $c^k (k!)^{1-\rho} \le \Gamma(k(1-\rho) + 1)
  \le C^k (k!)^{1-\rho}$. Hence, it follows that
  \begin{equation}\label{Gamma-bd2}
    \left(1-\rho\right)^{n-k} c^{k} \frac{n!}{(k!)^\rho}\le \Gamma(n-k\rho+1)
    \le C^{k} \frac{n!}{(k!)^\rho}.
  \end{equation}
  Putting this back into~\eqref{I_n^*}, we get~\eqref{E:wht-hn*} with
  $C_1 = C^{-1}\Gamma(2-\rho)$ and
  $C_2 = c^{-1}\left(1-\rho\right)^{-1}\Gamma(1-\rho)$. Hence,~\eqref{E:wht-hn}
  follows.

  As for part (iii), using~\eqref{E:wht-hn}, interchanging the order of
  summation and applying~\eqref{E:exp-bd} yield~\eqref{E:sum-I-bd}. Finally, for
  part (iv), after an application of the sub-additivity of the function
  $x\mapsto x^{1/p}$ to the far right-hand side of~\eqref{E:wht-hn}, one can
  carry out the same arguments as the proof of the upper bound
  of~\eqref{E:sum-I-bd} to show that the series in question converges. This
  completes the proof of Lemma~\ref{Lem:I}.
\end{proof}

The above lemma plays the same role as, e.g., Lemma A.2
of~\cite{chen.hu.ea:21:regularity} where the kernel function $k(t)$ takes the
form of $\left(t-s\right)^{-\rho}$. In that case, the computations can be made
explicit by using the Mittag-Leffler function. Indeed, Lemma~\ref{Lem:I} can be
rephrased as the Gronwall-type lemma below.

\begin{lemma}[Gronwall-type Lemma]\label{L:Gronwall}
  Let $\lambda \in \R_+$, $\rho\in (0,1)$, $\widehat{k}(t) = (1\wedge
  t)^{-\rho}$ and $b: \R_+ \to \R_+$ be a nonnegative function. Suppose that $H:
  \R_+ \to \R_+$ is a locally integrable nonnegative function such that for all
  $t \ge 0$,
  \begin{align}\label{E:Gron-Equ}
    H(t) & \le b(t) + \lambda^2 \int_0^t \widehat{k}(t-s) H(s) \ud s. \shortintertext{Then}
    H(t) & \le b(t) + \sum_{n = 1}^\infty \lambda^{2n} \left(\widehat{k}^{*n} * b\right)(t). \label{E:Gron-Res}
  \end{align}
\end{lemma}

\begin{remark}
  The following variation will not be used in this paper but it is worth noting
  that if $b(t) \ge 0$ and ``$\le$'' in~\eqref{E:Gron-Equ} is replaced by
  ``$\ge$'' (resp.\ ``$ = $''), then we will obtain the
  conclusion~\eqref{E:Gron-Res} with ``$\le$'' replaced by ``$\ge$'' (resp.\ ``$
  = $''). This can be shown by the same proof below.
\end{remark}

\begin{proof}[Proof of Lemma~\ref{L:Gronwall}]
Using~\eqref{E:Gron-Equ} inductively, we get that for any $t \ge 0$, for any $N \ge 1$,
\begin{equation*}
   H(t) \le b(t) + \sum_{n = 1}^{N-1} \lambda^{2n} \left({\widehat{k}}^{* n} * b\right)(t) + \lambda^{2N} \left({\widehat{k}}^{* N} * H\right) (t).
\end{equation*}
Similarly to the proof of Lemma~\ref{Lem:I}(ii), we can use~\eqref{E:beta-id} to deduce that
\begin{align*}
\widehat{k}^{* N}(t) \le \sum_{k = 0}^N \binom{N}{k} \frac{\Gamma(1-\rho)^k}{\Gamma(N-k\rho)} t^{N-1-k\rho}.
\end{align*}
Then, by~\eqref{Gamma-bd2} and the bound $\binom{N}{k} \le 2^N$, we have
\begin{align*}
\left(\widehat{k}^{* N} * H\right)(t)
\le \int_0^t H(t-s) \sum_{k = 0}^N \frac{\widetilde{C}^N}{(N!)^{1-\rho}} s^{N-1-k\rho} \ud s.
\end{align*}
Since $0 < \rho < 1$, if $N$ is large enough, then $N-1-k\rho \ge N - 1 - N\rho = N(1-\rho) - 1 > 0$ for all $k \le N$, and hence
\begin{align*}
\lambda^{2N}\left(\widehat{k}^{* N} * H\right)(t)
\le \frac{(\lambda^2 \widetilde{C})^N}{(N!)^{1-\rho}} N \big(t^{N-1} + t^{N(1-\rho)-1}\big) \int_0^t H(t-s) \ud s.
\end{align*}
Since $H$ is locally integrable, $\lambda^{2N}(\widehat{k}^{* N} * H)(t) \to 0$
as $N \to \infty$ and~\eqref{E:Gron-Res} follows.
\end{proof}

\subsection{Proof of part (i) of Theorem~\ref{T:ExUnque}}

Now we are ready to prove Theorem~\ref{T:ExUnque} for the case where the initial
data is bounded.
\begin{proof}[Proof of Theorem~\ref{T:ExUnque} (the bounded initial data case)]
  Here we assume that the initial condition $\nu$ is absolutely continuous with
  respect to the Lebesgue measure with a bounded density $g$, namely, $\nu(\ud
  x) = g(x)\ud x$ and $g\in L^\infty(\dom)$. The proof follows a standard Picard
  iteration scheme. Let $u_0(t,x) = J(t,x)$ (see~\eqref{E:J0}) and for $n \ge
  1$,
  \[
    u_n(t, x) = J(t,x) + \lambda \int_0^t \int_\dom G(t-s, x, y) \sigma(s, y, u_{n-1}(s, y)) W(\ud s, \ud y).
  \]
  By Burkholder's inequality and Minkowski's inequality, for all $p \ge 2$, $n
  \ge 2$,
  \begin{align*}
    \sup_{0<s\le t}\Norm{u_n(s, x) - u_{n-1}(s, x)}_p^2
    &\le C_p \Lip_\sigma^2 \lambda^2 \int_0^t \ud s \iint_{\dom^2} G(t-s, x, y)f(y-y')G(t-s, x, y') \\
    & \hspace{50pt} \times \sup_{z \in \dom}\Norm{u_{n-1}(s, z) - u_{n-2}(s, z)}_p^2\,\ud y\,\ud y',
  \end{align*}
  where $C_p$ is a constant depending only on $p$. By Lemma~\ref{Lem:iint-d}
  (for the Dirichlet case) or Lemma~\ref{Lem:iint-n} (for the Neumann case),
  \begin{align*}
    \sup_{x\in \dom}\iint_{\dom^2} G(t-s, x, y) f(y-y') G(t-s, x, y')\, \ud y\, \ud y'
        \le C e^{-2\mu (t-s)} (1\wedge (t-s))^{-\beta/2},
  \end{align*}
  where $\mu = \mu_1$ in the Dirichlet case and $\mu = 0$ in the Neumann case.
  Let
  \[
    H_n(t) \coloneqq e^{2\mu t}\sup_{\substack{(s,x)\in (0,t]\times \dom}} \Norm{u_n(s, x) - u_{n-1}(s, x)}_p^2.
  \]
  It follows that for all $n\ge 2$,
  \begin{align}\label{E:Hnt}
    H_n(t) \le a \int_0^t H_{n-1}(s) (1\wedge (t-s))^{-\beta/2} \ud s,
  \end{align}
  where $a = C_p C \Lip_\sigma^2 \lambda^2$, and for $n = 1$,
  \begin{align*}
    H_1(t) \le a \Norm{g}_{L^\infty(\dom)}^2
    \int_0^t e^{2\mu s} (1\wedge (t-s))^{-\beta/2}\ud s.
  \end{align*}
  For any $t\in [0,T]$ with $T$ fixed, using the notation in
  Definition~\ref{D:wht} with $\rho = \beta/2$, the above bound can be written
  as $H_1(t) \le C_T a (\widehat{k} * \widehat{h}_0)(t)$ for some constant
  $C_T$, and we have $H_n(t) \le C_T a^n \widehat{h}_n(t)$ for all $n \ge 0$.
  Hence, we can apply Lemma~\ref{Lem:I}(iv) with $p = 2$ to see that $\sum_{n =
  1}^\infty H_n^{1/2}(t) < \infty$. This implies that $u_n(t, x)$ converges to
  some $u(t, x)$ in $L^p (\Omega)$ which satisfies~\eqref{E:mild-sol}
  and~\eqref{E:sol-Lp}.

  For the uniqueness, suppose that $u$ and $\widetilde{u}$ are two mild solutions
  satisfying~\eqref{E:sol-Lp} with $p = 2$. By Burkholder's inequality with $p = 2$
  and similar calculations to those above, we get that
  \begin{align*}
     H(t) \le C \lambda^2 L_\sigma^2 \int_0^t H(s) (1 \wedge (t-s))^{-\beta/2} \ud s,
     \quad \text{for all $t > 0$,}
  \end{align*}
  where $H(t) \coloneqq \sup_{(s,x)\in (0,t]\times \dom}\Norm{u(s, x) - \tilde
  u(s, x)}_2^2$. The condition~\eqref{E:sol-Lp} implies that $H$ is locally
  bounded. Then we can apply the Gronwall-type Lemma~\ref{L:Gronwall} with
  $b(t)\equiv 0$ to see that $H(t)\equiv 0$, i.e., $u(t, x) = \widetilde{u}(t, x)$
  a.s. This completes the proof of Theorem~\ref{T:ExUnque} in the case where the
  initial measure has a bounded density.
\end{proof}

\section{The \texorpdfstring{$p$}{pth}-th moment bounds and rough initial data}\label{S:Key}

Our next task is to establish more bounds for the convolution-type integrals of
the heat kernels. Specifically, we seek optimal upper and lower bounds for the
integral:
\begin{align*}
  \iint_{\dom \times \dom} G(t-s, x, z) G(t-s, x', z') f(z-z') G(s, z, y) G(s, z', y') \,\ud z\, \ud z'.
\end{align*}
The following identity plays a key role in the estimation of the above integral:
for all $0 < s < t$ and $v, w \in \R^d$,
\begin{equation}\label{E:exp-id}
    \exp\left({-C\frac{|v|^2}{t-s}}\right)\exp\left({-C\frac{|v-w|^2}{s}}\right)
  = \exp\left(-C\frac{|w|^2}{t}\right)\exp\left(-C\frac{|v-\frac{t-s}{t}w|^2}{(t-s)s/t}\right).
\end{equation}
This identity can be verified by direct calculations;
see~\cite[p.657]{chen.kim:19:nonlinear}. It can also be interpreted as an
expression for the density of the Brownian bridge in terms of a conditional
density; see~\cite[Chapter 6]{candil:22:localization}.

\begin{lemma}\label{Lem:iintGGf-d}% {{{
  If $\dom$ is a bounded Lipschitz domain, then we have the
  following integral estimates:
  \begin{enumerate}
    \item[(i)] There exists a finite constant $C$ such that for all $0 < s < t$,
      for all $x, x', y, y' \in \dom$,
      \begin{equation}\label{E:iintGGf-d-ub}
      \begin{split}
        \iint_{\dom \times \dom}
        & \frac{e^{-2\mu_1 (t-s)}}{1\wedge (t-s)^d}e^{-c_1\frac{|x-z|^2+|x'-z'|^2}{t-s}}
         \frac{e^{-2\mu_1 s}}{1\wedge s^d}e^{-c_1\frac{|z-y|^2+|z'-y'|^2}{s}}
        f(z-z') \, \ud z\, \ud z'\\
        &\le C \frac{e^{-2\mu_1 t}}{1\wedge t^d} e^{-c_1 \frac{|x-y|^2 + |x'-y'|^2}{t}}
        \left(1 \wedge \frac{(t-s)s}{t}\right)^{-\beta/2}.
      \end{split}
      \end{equation}
    \item[(ii)] If $\dom$ is convex, then there exists $0 < \eps_0 < 1$ such
      that for all $0 < \eps \le \eps_0$, there exists a positive finite
      constant $C_\eps$ such that for all $0 < s < t$, for all $x, x', y, y' \in
      \dom_\eps$ with $|x-x'| \le \sqrt{(t-s)s/t}$ and $|y-y'| \le
      \sqrt{(t-s)s/t}$,
      \begin{equation}\label{E:iintGGf-d-lb}
      \begin{split}
        \iint_{\dom \times \dom}
        &\widetilde{G}_D(t-s, x, x', z, z')
        \widetilde{G}_D(s, z, z', y, y')
        f(z-z') \, \ud z\, \ud z'\\
        &\ge C_\eps \frac{e^{-2\mu_1 t}}{1\wedge t^d} e^{-c_2 \frac{|x-y|^2 + |x'-y'|^2}{t}}
        \left(1 \wedge \frac{(t-s)s}{t}\right)^{-\beta/2}.
      \end{split}
      \end{equation}
  \end{enumerate}
\end{lemma}

\begin{proof}
  (i) Denote the left-hand side of~\eqref{E:iintGGf-d-ub} by $I$.
  By~\eqref{E:f-bd},
  \begin{align*}
    I & \le C_f C_1^4\: e^{-2\mu_1 t} \frac{1}{(1\wedge (t-s)^d)(1\wedge s^d)} \iint_{\dom^2}
    e^{-c_1\left(\frac{|x-z|^2}{t-s} + \frac{|x'-z'|^2}{t-s} + \frac{|z-y|^2}{s} +
    \frac{|z'-y'|^2}{s}\right)} |z-z'|^{-\beta} \,\ud z\,\ud z'\\
    & = C_f C_1^4\: e^{-2\mu_1 t} e^{-c_1\left(\frac{|y-x|^2}{t}+\frac{|y'-x'|^2}{t}\right)}\\
    & \quad \times \frac{1}{(1\wedge (t-s)^d)(1\wedge s^d)}
    \iint_{\dom^2} e^{-c_1 \left(\frac{|(z-x) - \frac{t-s}{t}(y-x)|^2}{(t-s)s/t}
    + \frac{|(z'-x') - \frac{t-s}{t}(y'-x')|^2}{(t-s)s/t}\right)}|z-z'|^{-\beta} \,\ud z\,\ud z'.
  \end{align*}
  For the last equality, we have applied the identity~\eqref{E:exp-id} with $v =
  z-x$ and $w = y-x$, and also with $v' = z'-x'$ and $w' = y'-x'$. Set $\tau =
  (t-s)s/t$. We claim that there exists a finite constant $C$ such that for all
  $a, a' \in \R^d$, for all $\tau > 0$,
  \begin{equation}\label{iint-ub}
    I'\coloneqq \iint_{\dom^2} e^{-c_1 \left(\frac{|z-a|^2}{\tau}
    + \frac{|z'-a'|^2}{\tau}\right)}|z-z'|^{-\beta} \,\ud z\,\ud z'
    \le C(1\wedge \tau)^{d-\beta/2}.
  \end{equation}
  Indeed, if $\tau \ge 1$, then~\eqref{iint-ub} holds since $\beta < d$. If $\tau
  < 1$, then by the Plancherel theorem,
  \begin{align*}
  I' & \le \iint_{\R^d \times \R^d} e^{-c_1 \left(\frac{|z-a|^2}{\tau} + \frac{|z'-a'|^2}{\tau}\right)}|z-z'|^{-\beta} \,\ud z\,\ud z' \\
     & = C \tau^d \int_{\R^d} e^{-i(a-a')\cdot \xi} e^{-\frac{2}{c_1}\tau|\xi|^2}|\xi|^{\beta-d} \ud\xi                            \\
     & \le C \tau^d \int_{\R^d} e^{-\frac{2}{c_1}\tau|\xi|^2}|\xi|^{\beta-d} \ud\xi
       = C' \tau^{d-\beta/2},
  \end{align*}
  where the last equality can be obtained by a scaling argument. This
  verifies~\eqref{iint-ub} and hence the following:
  \begin{gather*}
    I \le Ce^{-2\mu_1 t} e^{-c_1\left(\frac{|y-x|^2}{t}+\frac{|y'-x'|^2}{t}\right)} h(s, t) \left(1\wedge \frac{(t-s)s}{t}\right)^{-\beta/2}, \shortintertext{where}
    h(s, t) \coloneqq \frac{1}{(1\wedge (t-s)^d)(1\wedge s^d)} \left(1 \wedge \frac{(t-s)^d s^d}{t^d}\right).
  \end{gather*}
  It can be verified by straightforward calculations that if $t-s < 1$ and $s <
  1$, then $h(s, t) \le t^{-d}$; otherwise, $h(s, t) \le 1$. Therefore, we obtain
  the upper bound~\eqref{E:iintGGf-d-ub}.

  (ii) To prove the lower bound~\eqref{E:iint-d-lb}, let $0 < s < t$ and $x, x',
  y, y' \in \dom_\eps$ be such that $|x-x'| \le \sqrt{(t-s)s/t}$ and $|y-y'| \le
  \sqrt{(t-s)s/t}$. Denote the left-hand side of~\eqref{E:iint-d-lb} by $\tilde
  I$. By~\eqref{E:f-bd},~\eqref{E:G-d},~\eqref{E:Phi-bd} and~\eqref{E:exp-id},
  we have that
  \begin{align*}
    \widetilde{I}
    & \ge        C_\eps \frac{e^{-2\mu_1 t}}{(1\wedge (t-s)^d)(1\wedge s^d)} e^{-c_2\left(\frac{|y-x|^2}{t} + \frac{|y'-x'|^2}{t}\right)} \iint_{\dom \times \dom} \ud z \ud z'                                        \\
    & \quad \times  (1\wedge \Phi_1(z))^2 (1\wedge \Phi_1(z'))^2 e^{-c_2 \left(\frac{|(z-x) - \frac{t-s}{t}(y-x)|^2}{1 \wedge \tau} + \frac{|(z'-x') - \frac{t-s}{t}(y'-x')|^2}{1 \wedge \tau}\right)} |z-z'|^{-\beta} \\
    & \eqqcolon A, \qquad \text{with $\tau = (t-s)s/t$}.
  \end{align*}

  Now, we consider the following two cases: $\tau \ge 1$ and $\tau <1$. In the
  rest of the proof, the constant $C_\eps$ above will be used to denote a
  generic constant that depends on $\eps$, whose value may change at each
  appearance.

  \medskip\noindent\textit{Case 1.~} Suppose $\tau\ge 1$. Observe that $(1\wedge
  (t-s)^d)(1\wedge s^d) \le 1\wedge t^d$ and $(1\wedge \tau)^{-\beta/2} = 1$.
  Also, by $\sup_{y, y' \in D}|y-y'| \le M < \infty$ and~\eqref{E:Phi-bd}, we
  see that
  \begin{align*}
    A &\ge C_\eps \frac{e^{-2\mu_1 t}}{1\wedge t^d}
      e^{-c_2\left(\frac{|y-x|^2}{t}+\frac{|y'-x'|^2}{t}\right)}
      \iint_{\dom_\eps \times \dom_\eps} (1 \wedge (c_0^{-1}\eps^{a_1}))^4
      e^{-2c_2 M^2} M^{-\beta}\, \ud z\, \ud z'\\
    & = C_\eps \frac{e^{-2\mu_1 t}}{1\wedge t^d}
      e^{-c_2\left(\frac{|y-x|^2}{t}+\frac{|y'-x'|^2}{t}\right)} (1\wedge \tau)^{-\beta/2},
  \end{align*}
  which proves~\eqref{E:iintGGf-d-lb}.

  \medskip\noindent{\textit{Case 2.~}} Suppose $\tau < 1$. In this case, $1
  \wedge \tau = \tau$. Notice that
  \begin{align*}
    (z-x) - \frac{t-s}{t}(y-x) = z - a, \quad \text{where $a\coloneqq \frac st x + \frac{t-s}{t} y$},
  \end{align*}
  and $a'$ is defined similarly.
  The convexity assumption on $\dom$ ensures that both $a$ and $a'$ are in $\dom_\eps$.
  By Lemma~\ref{Lem:vol2}, for $0 < \eps \le \eps_0$,
  \begin{align*}
    A & \ge C_\eps \frac{e^{-2\mu_1 t}}{(1\wedge (t-s)^d)(1\wedge s^d)} e^{-c_2\left(\frac{|y-x|^2}{t}+\frac{|y'-x'|^2}{t}\right)} \\
      & \quad \times \iint_{E(a) \times E(a')} e^{-2c_2} \textbf{1}_{\{|z-a| \le \sqrt{\tau}, \, |z'-a'| \le \sqrt{\tau}\}} |z-z'|^{-\beta} \, \ud z\, \ud z'.
  \end{align*}
  Since $|x-x'| \le \sqrt \tau$ and $|y-y'| \le \sqrt \tau$, we have $|a-a'|\le
  \sqrt \tau$. Hence, $|z-z'| \le 3\sqrt{\tau}$ on the set $\{|z-a| \le
  \sqrt{\tau}, |z'-a'| \le \sqrt{\tau}\}$. Also, by~\eqref{vol-V(x)}, it follows
  that
  \begin{align*}
    A \ge C_\eps \frac{e^{-2\mu_1 t}}{(1\wedge (t-s)^d)(1\wedge s^d)}
        e^{-c_2\left(\frac{|y-x|^2}{t}+\frac{|y'-x'|^2}{t}\right)} \tau^{d-\beta/2}.
  \end{align*}
  To finish the proof, it remains to show that the following bound holds for some
  constant $c > 0$:
  \begin{equation}\label{min-lb}
    \frac{\tau^d}{(1\wedge (t-s)^d)(1\wedge s^d)} \ge \frac{c}{1 \wedge t^d}.
  \end{equation}
  Indeed, for $t \le 2$, we can use the bound
  \begin{align*}
        \frac{\tau^d}{(t-s)^d s^d}
    =   \frac{1}{t^d}
    =   \frac{1}{2 \wedge t^d}
    \ge \frac{1}{2(1\wedge t^d)}.
  \end{align*}
  For $t > 2$, consider the following two cases. If $s < t/2$, then $t-s > t/2 >
  1$ and hence
  \begin{align*}
    \frac{\tau^d}{(1\wedge (t-s)^d)(1\wedge s^d)}
    = \frac{\tau^d}{s^d}
    = \frac{(t-s)^d}{t^d} \ge \frac{1}{2^d}
    = \frac{1}{2^{d}(1 \wedge t^d)}.
  \end{align*}
  If $s \ge t/2$, which is $> 1$, then
  \begin{align*}
    \frac{\tau^d}{(1\wedge (t-s)^d)(1\wedge s^d)}
    = \frac{\tau^d}{(t-s)^d}
    = \frac{s^d}{t^d} \ge \frac{1}{2^d}
    = \frac{1}{2^{d}(1 \wedge t^d)}.
  \end{align*}
  This proves~\eqref{min-lb} and completes the proof of the lower
  bound~\eqref{E:iintGGf-d-lb}.
\end{proof}

\begin{lemma}\label{Lem:iintGGf-n}
  If $\dom$ is a bounded Lipschitz domain, then we have the
  following integral estimates:
  \begin{enumerate}
  \item[(i)] There exists a positive finite constant $C$ such that for all $0 <
    s < t$, for all $x, x', y, y' \in \dom$,
    \begin{equation}\label{E:iintGGf-n-ub}
    \begin{split}
      \iint_{\dom \times \dom}
      &\widetilde{G}_N(t-s, x, x', z, z')
       \widetilde{G}_N(s, z, z',y,y')
      f(z-z') \, \ud z\, \ud z'\\
      &\le \frac{C}{1 \wedge t^d}e^{-c_3 \frac{|x-y|^2 + |x'-y'|^2}{t}}
      \left(1\wedge \frac{(t-s)s}{t}\right)^{-\beta/2}.
    \end{split}
    \end{equation}
  \item[(ii)] If $\dom$ is convex and \eqref{E:G-n-lower} holds, then there
    exists a positive finite constant $C$ such that for all $0 < s < t$, for all
    $x, x', y, y' \in \dom$ with $|x-x'|\le \sqrt{(t-s)s/t}$ and $|y-y'| \le
    \sqrt{(t-s)s/t}$,
    \begin{equation}\label{E:iintGGf-n-lb}
    \begin{split}
      \iint_{\dom \times \dom}
      &\widetilde{G}_N(t-s, x, x', z, z')
       \widetilde{G}_N(s, z, z',y,y')
        f(z-z') \, \ud z\, \ud z'\\
      &\ge \frac{C}{1\wedge t^d} e^{-c_4 \frac{|x-y|^2 + |x'-y'|^2}{t}}
      \left(1\wedge \frac{(t-s)s}{t}\right)^{-\beta/2}.
    \end{split}
    \end{equation}
  \end{enumerate}
\end{lemma}
The proof of Lemma~\ref{Lem:iintGGf-n} is similar to that of
Lemma~\ref{Lem:iintGGf-d}, which will be left to interested readers.

\begin{lemma}\label{Lem:tildeh}
  Let $0 < \rho < 1$. Define $\widetilde{h}_0(t) \equiv 1$ and for $n \ge 1$
  \begin{equation}\label{E:wtdh}
    \widetilde{h}_n(t)\coloneqq
    \int_0^t \ud s_1 \int_0^{s_1} \ud s_2 \dots \int_0^{s_{n-1}} \ud s_n
    \prod_{j = 1}^n\left(1 \wedge \frac{(s_{j-1}-s_j)s_j}{s_{j-1}}\right)^{-\rho},
  \end{equation}
  where we use the convention that $s_0 = t$. Then, for all $t > 0$, for all
  integers $n \ge 0$,
  \begin{equation}\label{E:tildeh-bd}
        2^{-n} \widehat{h} _n(t)
    \le \widetilde{h}_n(t)
    \le \left(2^{1+\rho}\right)^n \widehat{h}_n(t),
  \end{equation}
  where $\widehat{h}_n(t)$ is as defined in Lemma~\ref{Lem:I}. Moreover, for any
  $\lambda > 0$, there exists positive finite constants $ C_3,\dots, C_8$ such
  that
  \begin{equation}\label{E:tildeK-bd}
        C_3 \exp\left(t\left(C_4 \lambda^2 + C_5 \lambda^{\frac{2}{1-\rho}}\right)\right)
    \le \sum_{n = 0}^{\infty} \lambda^{2n} \widetilde{h}_n(t)
    \le C_6 \exp\left(t\left(C_7 \lambda^2 + C_{8} \lambda^{\frac{2}{1-\rho}}\right)\right).
  \end{equation}
\end{lemma}
\begin{proof}
  The estimate~\eqref{E:tildeh-bd} can be deduced using Lemma~4.17
  of~\cite{candil:22:localization}. For completeness, we provide a direct proof.
  We prove~\eqref{E:tildeh-bd} by induction. Clearly, it holds for $n = 0$.
  Suppose it holds for some $n \ge 0$. Since $(t-s)s/t \ge s/2$ for $s \in [0,
  t/2]$ and $(t-s)s/t \ge (t-s)/2$ for $s \in [t/2, t]$, we have
  \begin{align*}
      \widetilde{h}_{n+1}(t)
    & = \int_0^t \left(1 \wedge \frac{(t-s)s}{t}\right)^{-\rho} \widetilde{h}_n(s)\, \ud s \\
    & \le 2^\rho (2^{1+\rho})^n \int_0^{t/2} (1\wedge s)^{-\rho} \widehat{h}_n(s)\, \ud s
      + 2^\rho (2^{1+\rho})^n \int_{t/2}^t (1 \wedge (t-s))^{-\rho} \widehat{h}_n(s)\, \ud s\\
    & = 2^\rho \int_{t/2}^t (1\wedge (t-s))^{-\rho} \widehat{h}_n(t-s)\, \ud s
      + 2^\rho \int_{t/2}^t (1 \wedge (t-s))^{-\rho} \widehat{h}_n(s)\, \ud s.
  \end{align*}
  Since $t-s \le s$ for $s \in [t/2, t]$, and $\widehat{h}_n(\cdot)$ is
  nondecreasing by Lemma~\ref{Lem:I}, we see that
  \begin{align*}
      \widetilde{h}_{n+1}(t)
    & \le 2^\rho (2^{1+\rho})^n \int_{t/2}^t (1\wedge (t-s))^{-\rho} \widehat{h}_n(s)\, \ud s
      + 2^\rho (2^{1+\rho})^n \int_{t/2}^t (1 \wedge (t-s))^{-\rho} \widehat{h}_n(s)\, \ud s\\
    & \le 2^{1+\rho} (2^{1+\rho})^n \int_0^t (1\wedge (t-s))^{-\rho} \widehat{h}_n(s) \, \ud s\\
    & \le (2^{1+\rho})^{n+1} \widehat{h}_{n+1}(t),
  \end{align*}
  where we have applied~\eqref{E:Crho} in the last step. This proves the upper
  bound of~\eqref{E:tildeh-bd}. On the other hand, for the lower bound, it holds
  clearly for $n = 0$. For general $n\ge 1$, by the bound $(t-s)s/t \le t-s$ for
  all $s \in [0, t]$, the induction hypothesis and~\eqref{E:Crho}, we have
  \begin{align*}
    \widetilde{h}_{n+1}(t)
    & \ge 2^{-n}\int_0^t (1 \wedge (t-s))^{-\rho} \widehat{h}_n(s) \, \ud s\\
    & \ge 2^{-n-1}\int_0^t \left(1 + (t-s)^{-\rho}\right) \widehat{h}_n(s) \, \ud s\\
    & =   2^{-(n+1)}\widehat{h}_{n+1}(t).
  \end{align*}
  Finally,~\eqref{E:tildeK-bd} follows clearly from~\eqref{E:tildeh-bd}
  and~\eqref{E:sum-I-bd}. This proves Lemma~\ref{Lem:tildeh}.
\end{proof}

Recall the following form of Burkholder's inequality~\cite[Theorem
B.1]{khoshnevisan:14:analysis}: For any $p \in [2, \infty)$ and any continuous
$L^2$-martingale $\{M_t, t \ge 0 \}$,
\begin{equation}\label{BDG}
  \E(|M_t|^p) \le (4p)^{p/2} \E(\langle M \rangle_t^{p/2}),
\end{equation}
where $\langle M \rangle_t$ denotes the quadratic variation of $M_t$. \medskip

\subsection{Proof of part (ii) of Theorem~\ref{T:ExUnque}}

Now we are ready to prove the second part of Theorem~\ref{T:ExUnque}.

\begin{proof}[Proof of Theorem~\ref{T:ExUnque} (the rough initial data case)]
  We first assume the Dirichlet boundary condition. Let $u_0(t, x) = J(t, x)$
  and
  \begin{align*}
    u_n(t, x) = J(t, x) + \lambda \int_0^t \int_{\dom} G(t-s, x, y) \sigma(s, y, u_{n-1}(s, y)) W(\ud s, \ud y) \quad \text{for $n \ge 1$}.
  \end{align*}
  Let $a = 8p \lambda^2 L_\sigma^2 C$, where $C$ is the constant
  in~\eqref{E:iintGGf-d-ub}. We claim that, for all $n \ge 0$,
  \begin{align}\label{E:u-Lp}
    \|u_n(t, x)\|_p \le \sqrt 2 e^{-\mu_1 t}J_{c_1}(t, x) \left(\sum_{i = 0}^n a^i\, \widetilde{h}_i(t)\right)^{1/2} \text{for all} t>0, x \in \dom,
  \end{align}
  where $\widetilde{h}_n(t)$ is defined in~\eqref{E:wtdh}. By
  Proposition~\ref{P:G},~\eqref{E:u-Lp} holds for $n = 0$. Suppose
  that~\eqref{E:u-Lp} is true for some $n\ge 0$. By Burkholder's
  inequality~\eqref{BDG} and Minkowski's inequality, we get that
  \begin{align*}
    \|u_{n+1}(t, x)\|_p^2 \le 2 J^2 (t, x) + 8p \lambda^2 L_\sigma^2 I_n(t, x),
  \end{align*}
  where
  \begin{align*}
    I_n(t, x) \coloneqq \int_0^t \ud s \iint_{\dom^2} \ud y\, \ud y' G_D(t-s, x, y) G_D(t-s, x, y')
    f(y-y') \|u_{n}(s, y)\|_p \|u_{n}(s, y')\|_p.
  \end{align*}
  By the induction hypothesis,
  \begin{align*}
    I_n(t, x) & \le 2\int_0^t \ud s\iint_{\dom^2} \ud y\, \ud y'f(y-y')G_D(t-s,x,y)G_D(t-s,x,y') \\
              & \quad \times |J(s,y)|\, |J(s,y')| \left(\sum_{i = 0}^{n} a^i \,\widetilde{h}_i(s)\right) \\
              & \le 2\sum_{i = 0}^n a^i \int_0^t \ud s\, \widetilde{h}_i(s)\iint_{\dom^2}\ud y\,\ud y'f(y-y')G_D(t-s,x,y)G_D(t-s,x,y') \\
              & \quad\times \iint_{\dom^2}|\nu|(\ud z)|\nu|(\ud z')G_D(s,y,z)G_D(s,y,z').
  \end{align*}
  Now interchange the order of the two double integrals and apply
  Lemma~\ref{Lem:iintGGf-d}(i) to get that
  \begin{align*}
    I_n(t, x)
    & \le 2\sum_{i = 0}^n a^i \int_0^t \ud s\, \widetilde h_i(s)
    \iint_{\dom^2}|\nu|(\ud z)|\nu|(\ud z')\\
    & \quad \times
    C\frac{e^{-2\mu_1 t}}{1\wedge t^d} e^{-c_1 \frac{|x-z|^2 +|x-z'|^2}{t}}
    \left(1\wedge \frac{(t-s)s}{t}\right)^{-\beta/2}\\
    & = 2C e^{-2\mu_1 t}J_{c_1}^2 (t,x) \sum_{i = 0}^n a^i
    \int_0^t \ud s \left(1\wedge \frac{(t-s)s}{t}\right)^{-\beta/2} \widetilde{h}_i(s)\\
    & = 2C e^{-2\mu_1 t}J_{c_1}^2 (t,x) \sum_{i = 0}^n a^i\, \widetilde h_{i+1}(t).
  \end{align*}
  Recall that $a = 8p \lambda^2 L_\sigma^2 C$. Hence,
  \begin{align*}
    \|u_{n+1}(t, x)\|_p^2
    & \le 2e^{-2\mu_1 t}J_{c_1}^2 (t, x) + 16p\lambda^2 L_\sigma^2 C e^{-2\mu_1 t}J_{c_1}^2 (t,x) \sum_{i = 0}^n a^i\, \widetilde h_{i+1}(t) \\
    & = 2e^{-2\mu_1 t}J_{c_1}^2 (t, x) \sum_{i = 0}^{n+1}a^i\, \widetilde{h}_i(t).
  \end{align*}
  This proves~\eqref{E:u-Lp}. \medskip

  Next, we prove that $u_n(t, x)$ is a Cauchy sequence in $L^p (\Omega)$. Let
  $u_{-1}(t, x) = 0$. By Burkholder's inequality~\eqref{BDG} and Minkowski's
  inequality,
  \begin{align*}
        \Norm{u_m(t, x)-u_n(t, x)}_p^2
  \le & \: 4p \lambda^2 L_\sigma^2 \int_0^t \ud s \iint_{\dom^2} \ud y\, \ud y' G_D(t-s, x, y) G_D(t-s, x, y')f(y-y') \\
      & \times \|u_{m-1}(s, y) - u_{n-1}(s, y)\|_p\|u_{m-1}(s, y') - u_{n-1}(s, y')\|_p.
  \end{align*}
  Then, similarly to the above, we can show by induction that, for all $0 \le n
  < m$,
  \begin{align*}
    \|u_m(t, x)-u_n(t, x)\|_p \le \sqrt 2 e^{-\mu_1 t} J_{c_1}(t, x) \left(\sum_{i = n+1}^m a^i \, \widetilde h_i(t)\right)^{1/2}
    \text{for all $t > 0$ and $x \in \dom$.}
  \end{align*}
  By Lemma~\ref{Lem:tildeh},
  \begin{align*}
    \sum_{i = 0}^\infty a^i\, \widetilde h_i(t) \le C_6 e^{t\left(C_7 a + C_8 a^{\frac{2}{2-\beta}}\right)} < \infty.
  \end{align*}
  This implies that $u_n(t, x)$ is a Cauchy sequence in $L^p (\Omega)$, and
  hence converges in $L^p (\Omega)$ to some $u(t, x)$ which
  satisfies~\eqref{E:mild-sol}. In particular, by~\eqref{E:u-Lp} and
  Lemma~\ref{Lem:tildeh}, we have
  \begin{align*}
    \|u(t, x)\|_p
    & \le \sqrt 2 e^{-\mu_1 t} J_{c_1}(t, x) \left(\sum_{i = 0}^\infty (8p\lambda^2 L_\sigma^2 C)^i \, \widetilde{h}_i(t)\right)^{1/2} \\
    & \le \sqrt 2 C_6 J_{c_1}(t, x) e^{\frac 12 t\left(8C_7 C p \lambda^2 L_\sigma^2 +C_8 (8C)^{\frac{2}{2-\beta}}p^{\frac{2}{2-\beta}} \lambda^{\frac{4}{2-\beta}} L_\sigma^{\frac{4}{2-\beta}} -\mu_1\right)}.
  \end{align*}

  For uniqueness, suppose that $u(t, x)$ and $\widetilde{u}(t, x)$ are mild solutions
  of~\eqref{E:she-d} satisfying~\eqref{E:sol-unique}. Let $T > 0$. Then, for all
  $t \in (0, T]$, for all $x \in \dom$,
  \begin{equation}\label{Ineq:unique}
    \begin{split}
        \|u(t, x) - \widetilde{u}(t, x)\|_2^2
      & \le \lambda^2 L_\sigma^2 \int_0^t \ud s \iint_{\dom^2} \ud y\, \ud y' G_D(t-s, x, x, y_1, y_1') f(y-y') \\
      & \quad \times \|u(s, y) - \widetilde{u}(s, y)\|_2 \|u(s, y') - \widetilde{u}(s, y')\|_2.
    \end{split}
  \end{equation}
  We claim that there exists $C_T < \infty$ such that for all $n \ge 1$, for all
  $t \in (0, T]$ and all $x \in \dom$,
  \begin{equation}\label{claim:unique}
    \|u(t, x) - \widetilde{u}(t, x)\|_2^2
    \le 4C_T^2 J_c^2 (t, x) (\lambda L_\sigma)^{2n} \widetilde h_n(t),
  \end{equation}
  where $\widetilde h_n$ is defined in Lemma~\ref{Lem:tildeh}. We prove this
  claim by induction. First, consider $n = 1$. Using the condition that $u$ and
  $\widetilde{u}$ both satisfy~\eqref{E:sol-unique}, we get that
  \begin{align*}
    \MoveEqLeft \|u(t, x) - \widetilde{u}(t, x)\|_2^2\\
    & \le 4C_T^2 \lambda^2 L_\sigma^2 \int_0^t \ud s \iint_{\dom^2} \ud y\, \ud y' \widetilde{G}(t-s_1, x, x, y_1, y_1') f(y-y') J_{c}(s, y) J_{c}(s, y')\\
    & = 4C_T^2 \lambda^2 L_\sigma^2 \iint_{\dom^2} |\nu|(\ud z) |\nu|(\ud z') \\
    & \quad \times \int_0^t \ud s \iint_{\dom^2} \ud y\, \ud y' \widetilde{G}(t-s_1, x, x, y_1, y_1') f(y-y')
      \frac{1}{1 \wedge s^d} e^{-c \frac{|y-z|^2 + |y'-z'|^2}{s}}.
  \end{align*}
  Then, by Lemma~\ref{Lem:iintGGf-d}(i),
  \begin{align*}
      \|u(t, x) - \widetilde{u}(t, x)\|_2^2
    & \le 4C_T^2 \lambda^2 L_\sigma^2 J_{c}^2 (t, x) \int_0^t e^{-2\mu_1 (t-s)} \left(1 \wedge \frac{(t-s)s}{t}\right)^{-\beta/2} \ud s \\
    & \le 4C_T^2 J_{c}^2 (t, x) (\lambda L_\sigma)^2 \int_0^t \left(1 \wedge \frac{(t-s)s}{t}\right)^{-\beta/2} \ud s.
  \end{align*}
  This proves~\eqref{claim:unique} for $n = 1$. Assume that~\eqref{claim:unique}
  holds for some $n \ge 1$. We apply the induction hyposthesis to $\|u(s, y) -
  \widetilde{u}(s, y)\|_2$ and $\|u(s, y') - \widetilde{u}(s, y')\|_2$ on the right-hand
  side of~\eqref{Ineq:unique} to get
  \begin{align*}
        \|u(t, x) - \widetilde{u}(t, x)\|_2^2
    & \le 4C_T^2 (\lambda L_\sigma)^{2(n+1)} \int_0^t \ud s \, \widetilde{h}_n(s) \\
    & \quad \times \iint_{\dom^2} \ud y\, \ud y' \widetilde{G}(t-s_1, x, x, y_1, y_1') f(y-y') J_{c}(s, y) J_{c}(s, y').
  \end{align*}
  Then, by similar calculations to those in the $n = 1$ case, we obtain
  \begin{align*}
      \Norm{u(t, x) - \widetilde{u}(t, x)}_2^2
    & \le 4C_T^2 J_{c}^2 (t, x) (\lambda L_\sigma)^{2(n+1)} \int_0^t \widetilde{h}_n(s) \left(1 \wedge \frac{(t-s)s}{t}\right)^{-\beta/2} \ud s \\
    & = 4C_T^2 J_{c}^2 (t, x) (\lambda L_\sigma)^{2(n+1)} \widetilde{h}_{n+1}(t).
  \end{align*}
  This proves the claim~\eqref{claim:unique}. Finally, by
  Lemma~\ref{Lem:tildeh}, we have $\sum_{n = 0}^\infty (\lambda L_\sigma)^{2n}
  \widetilde h_n(t) < \infty,$ which implies that $(\lambda L_\sigma)^{2n}
  \widetilde h_n(t) \to 0$ as $n \to \infty$. Hence, by~\eqref{claim:unique},
  $u(t, x) = \widetilde{u}(t, x)$ a.s.

  The case of the Neumann boundary condition can be proved in the same way
  except that Lemma~\ref{Lem:iintGGf-n} (i) is applied in place of
  Lemma~\ref{Lem:iintGGf-d} (i). This completes the proof of part (ii) of
  Theorem~\ref{T:ExUnque}.
\end{proof}

\section{The two-point correlation function}\label{S:Two-point}

\subsection{A general formula for the two-point correlation function}

The authors in~\cite{chen.kim:19:nonlinear} have defined and studied the
space-time convolution-type operator ``$\triangleright$'' for the heat kernel in
the whole space $\R^d$, where the heat kernel $G(t,x,y)$ can be written as
$G(t,x-y)$. However, when the domain is not the whole space $\R^d$, for example
when it is a bounded domain, this translation invariant property is no longer
true and one has to keep the fundamental solution in the form of three
parameters. This natural generalization of the operator ``$\triangleright$'' for
fundamental solutions from those with two parameters to those with three has
been carried out by the first author's thesis; see~\cite[Chapter
5]{candil:22:localization}. Here, let us first briefly recall this
generalization.

For two measurable functions $k_1, k_2: \R_+ \times \dom^4 \to \R$, we define
\begin{align}\label{E:tri}
  (k_1 \triangleright k_2)(t, x, x', y, y')
  = \int_0^t \ud s \iint_{\dom^2} \ud z\, \ud z' \,
  k_1(t - s, x, x', z, z') k_2(s, z, z', y, y') f(z - z')
\end{align}
whenever the multiple integral is well-defined. For $h: \R_+ \times \dom^2 \to
\R$, we define $k_1 \triangleright h = k_1 \triangleright \overline{h}$, where
$\overline{h}(t, x, x', y, y') \coloneqq h(t, x, x')$.

\begin{remark}
  In~\cite{chen.kim:19:nonlinear}, where $U = \R^d$, the operator ``$\triangleright$'' is defined by
  \begin{align*}
    (k_1 \triangleright k_2)(t, x, x'; y)
    = \int_0^t \ud s \iint_{\R^d \times \R^d} \ud z\, \ud z' \,
           & k_1\big(t - s, x-z, x'- z'; y - (z-z')\big) \\
    \times & k_2(s, z, z'; y) f\big(y-(z - z')\big)
  \end{align*}
  for measurable functions $k_i:\R_+\times\R^3\to \R$, $i = 1,2$. The prototype
  of the function $k$ is $k(t,x,x';y) = G(t,x-y)G(t,x'-y)$, which corresponds to
  (and actually equals to) $J(t,x)J(t,x')$ when the initial condition is $\nu =
  \delta_y$. But when there is a lack of space invariance, one has to write this
  function $k$ as $k(t,x,x',y,y') = G(t,x,y)G(t,x',y')$. Indeed, introducing one
  more parameter in the definition of the convolution operator
  ``$\triangleright$" makes the associative property much more straightforward:
  Provided that one can apply Fubini's theorem to interchange of the order of integration, one gets that for any
  measurable functions $k_1, k_2, k_3: \R_+ \times \dom^4 \to \R$,
  \begin{align}\label{E:associative}
      (k_1 \triangleright (k_2 \triangleright k_3))(t, x, x', y, y')
    = ((k_1 \triangleright k_2) \triangleright k_3)(t, x, x', y, y').
  \end{align}
  Therefore, it makes sense to write $k_1 \triangleright k_2 \triangleright k_3$
  without parentheses to specify the order. Moreover, for any $n \ge 1$, we
  define $k^{\triangleright n}$ by $k^{\triangleright n} \coloneqq \underbrace{k
  \triangleright \dots \triangleright k}_{\text{$n$ times}}$.
\end{remark}

Now we apply this operator to the function $\widetilde{G}(t, x, x', y, y') =
G(t, x, y)G(t, x', y')$. For any $\lambda > 0$ and $x,x',y,y'\in \dom$, define
formally the function $\mathcal{K}^\lambda$ by
\begin{equation}\label{Def:K}
  \mathcal{K}^\lambda(t, x, x', y, y')
  \coloneqq \sum_{n = 1}^\infty \lambda^{2n} \widetilde{G}^{\triangleright n}(t, x, x', y, y').
\end{equation}
Specifically, we will write $\mathcal{K}^\lambda_D$ and $\widetilde{G}_D$ when
$G = G_D$ is the Dirichlet heat kernel, and write $\mathcal{K}^\lambda_N$ and
$\widetilde{G}_N$ when $G = G_N$ is the Neumann heat kernel. Upper and lower
bounds for $\mathcal{K}^\lambda_D$ and $\mathcal{K}^\lambda_N$ will be proved
later in Propositions~\ref{P:K-d} and~\ref{Lem:K-n}. In particular, the upper
bounds there imply the convergence of the series in~\eqref{Def:K}.

The connection between the two-point correlation function and the function
$\mathcal{K}^\lambda$ is given by the next proposition.

\begin{proposition}\label{P:2p-corr}
  Let $\dom$ be a bounded Lipschitz domain and $u(t, x)$ be the solution
  of~\eqref{E:she-d} or~\eqref{E:she-n}. Recall that $J(t,x)$ is the solution to
  the homogeneous equation; see~\eqref{E:J0}. Let $\widetilde{J}(t, x, x') =
  J(t, x) J(t, x')$.
  \begin{enumerate}
    \item[(i)] If $\sigma(t, x, u) = u$ for all $t > 0$, $x \in \dom$ and $u \in \R$, then
      \begin{equation}\label{E:2p-corr}
        \E(u(t, x) u(t, x')) = \widetilde{J}(t, x, x') +
        \sum_{n = 1}^\infty \lambda^{2n} (\widetilde{G}^{\triangleright n} \triangleright \widetilde{J})(t, x, x', 0, 0).
      \end{equation}
    \item[(ii)] If there exists a constant $\lip_\sigma > 0$ such that
      $\sigma(t, x, u) \ge \lip_\sigma |u|$ for all $t > 0$, $x \in \dom$, $u
      \in \R$, then
      \begin{equation}\label{E:2p-corr-LB}
        \E(u(t, x) u(t, x')) \ge \widetilde{J}(t, x, x') +
        \sum_{n = 1}^\infty (\lambda \lip_\sigma)^{2n} (\widetilde{G}^{\triangleright n} \triangleright \widetilde{J})(t, x, x', 0, 0).
      \end{equation}
    \item[(iii)] If $\sigma(t, x, u) \le \Lip_\sigma u$ for all $t > 0$, $x \in
      \dom$, $u \in [0, \infty)$ and $u(t,x)\ge 0$ a.s. for all $t> 0$ and $x\in
      \dom$, then
      \begin{equation}\label{E:2p-corr-UB}
        \E(u(t, x)u(t, x')) \le \widetilde{J}(t, x, x') +
        \sum_{n = 1}^\infty (\lambda \Lip_\sigma)^{2n} (\widetilde{G}^{\triangleright
        n} \triangleright \widetilde{J})(t, x, x', 0, 0).
      \end{equation}
  \end{enumerate}
  Moreover, for any $a > 0$, we have
  \begin{equation}\label{E:iintK}
     \widetilde{J}(t, x, x') + \sum_{n = 1}^\infty a^{2n} (\widetilde{G}^{\triangleright n} \triangleright \widetilde{J})(t, x, x', 0, 0)
     = a^{-2} \iint_{\dom^2} \mathcal{K}^a (t, x, x', y, y') \nu(\ud y) \nu(\ud y').
  \end{equation}
\end{proposition}

\begin{remark}
  Note that case (i) in the above proposition covers the important special
  case---the Anderson model. In this case, the two-point correlation function
  enjoys an explicit formula, having an equality in~\eqref{E:2p-corr}. For the
  nonlinear case, one needs to either introduce a cone condition, $\sigma(t, x,
  u) \ge l_\sigma |u|$, for the lower bounds as in case (ii) or assume the
  nonnegativity of solution for the upper bounds as in case (iii).
\end{remark}

\begin{proof}[Proof of Proposition~\ref{P:2p-corr}]
  Let $u_0(t, x) = J(t, x)$. For $n \ge 1$, let
  \begin{align*}
    u_n(t, x) =
    J(t, x) + \lambda \int_0^t \int_\dom G(t-s, x, y) \sigma(s, y, u_{n-1}(s, y)) W(\ud s, \ud y).
  \end{align*}
  Let $\rho_n(t, x, x') \coloneqq \E(u_n(t, x) u_n(t, x'))$. From the proof of
  Theorem~\ref{T:ExUnque}, we have $u_n(t, x) \to u(t, x)$ in $L^2 (\Omega)$ as
  $n \to \infty$, and $u(t, x)$ satisfies~\eqref{E:mild-sol}. It follows that
  $\rho_n(t, x, x') \to \rho(t, x, x')$ as $n \to \infty$, where $\rho(t, x, x')
  = \E(u(t, x) u(t, x'))$. \medskip

  (i) Suppose that $\sigma(t, x, u) = u$ for all $u \in \R$. Then, we have
  \begin{align*}
      & \quad \E(u_n(t, x) u_n(t, x'))                                                                                                                           \\
    & = \widetilde{J}(t, x, x') + \lambda^2 \int_0^t \ud s \iint_{\dom^2} \ud y\, \ud y'\, G(t-s, x, y) f(y-y') G(t-s, x', y') \E[u_{n-1}(s, y) u_{n-1}(s, y')] \\
    & = \widetilde{J}(t, x, x') + \lambda^2 (\widetilde{G} \triangleright \rho_{n-1})(t, x, x', 0, 0).
  \end{align*}
  Iterating this, we get
  \begin{align*}
    \rho_n(t, x, x')
    & = \widetilde{J}(t, x, x') + \lambda^2 (\widetilde{G} \triangleright \rho_{n-1})(t, x, x', 0, 0)\\
    & = \widetilde{J}(t, x, x') + \sum_{m = 1}^n \lambda^{2m} (\widetilde{G}^{\triangleright m} \triangleright \widetilde{J})(t, x, x', 0, 0).
  \end{align*}
  Then, we let $n \to \infty$ to get~\eqref{E:2p-corr}.
  Note that this also implies that the series
  \[
    \sum_{n = 1}^\infty \lambda^{2n} \left(\widetilde{G}^{\triangleright n} \triangleright \widetilde{J}\right) (t, x, x', 0, 0)
  \]
  is convergent for any $\lambda > 0$, $t > 0$ and $x, x' \in \dom$. \medskip

  (ii) Suppose that $\sigma(t, x, u) \ge \lip_\sigma |u|$ for all $t > 0$, $x \in \dom$ and $u \in \R$. We have
  \begin{align*}
    \rho_n(t, x, x')
    = \widetilde{J}(t, x, x') + & \lambda^2 \int_0^t \ud s \iint_{\dom^2} \ud y\, \ud y'\, G(t-s, x, y) f(y-y') G(t-s, x', y') \\
                           & \times \E[\sigma(s, y, u_{n-1}(s, y)) \sigma(s, y', u_{n-1}(s, y'))].
  \end{align*}
  For any $s \ge 0$ and $y, y' \in \dom$,
  \begin{align*}
    \E[\sigma(s, y, u_{n-1}(s, y))\sigma(s, y', u_{n-1}(s, y'))]
    &\ge \lip_\sigma^2\, \E(|u_{n-1}(s, y) u_{n-1}(s, y')|)\\
    &\ge \lip_\sigma^2\, \E(u_{n-1}(s, y) u_{n-1}(s, y')).
  \end{align*}
  It follows that
  \begin{align*}
    \MoveEqLeft \rho_n(t, x, x')\\
    & \ge \widetilde{J}(t, x, x') + (\lambda\, \lip_\sigma)^2
    \int_0^t \ud s \iint_{\dom^2} \ud y\, \ud y'\, G(t-s, x, y) f(y-y') G(t-s, x', y')
    \rho_{n-1}(s, y, y')\\
    & = \widetilde{J}(t, x, x') + (\lambda\, \lip_\sigma)^2 \left(\widetilde{G}^{\triangleright n} \triangleright \rho_{n-1}\right)(t, x, x', 0, 0).
  \end{align*}
  By induction,
  \begin{align*}
    \rho_n(t, x, x') \ge \widetilde{J}(t, x, x') + \sum_{m = 1}^n (\lambda\, \lip_\sigma)^{2m} \left(\widetilde{G}^{\triangleright m} \triangleright \widetilde{J}\right)(t, x, x', 0, 0).
  \end{align*}
  Then, we let $n \to \infty$ to get~\eqref{E:2p-corr-LB}. \medskip

  (iii) Suppose $\sigma(t, x, u) \le \Lip_\sigma u$ for all $t > 0$, $x\in \dom$ and $u \in [0, \infty)$. By the
  nonnegativity assumption of the solution, we see that
  \begin{align*}
    \MoveEqLeft \rho(t, x, x') = \E(u(t, x)u(t, x'))\\
    &= \widetilde{J}(t, x, x') +
    \lambda^2 \int_0^t \ud s \iint_{\dom^2} \ud y\, \ud y'\, G(t-s, x, y) f(y-y')G(t-s, x', y')\\
    & \hspace{170pt} \times \E[\sigma(s, y, u(s, y)) \sigma(s, y', u(s, y'))]\\
    & \le \widetilde{J}(t, x, x') +
    (\lambda \Lip_\sigma)^2 \int_0^t \ud s \iint_{\dom^2} \ud y\, \ud y'\, G(t-s, x, y) f(y-y')G(t-s, x', y') \E[u(s, y) u(s, y')]\\
    & = \widetilde{J}(t, x, x') + (\lambda \Lip_\sigma)^2 \left(\widetilde{G} \triangleright \rho\right)(t, x, x', 0, 0).
  \end{align*}
  Iterating this, we get,
  \begin{align*}
    \rho(t, x, x')
     \le \widetilde{J}(t, x, x') + \sum_{m = 1}^{n-1} (\lambda \Lip_\sigma)^{2m} \left(\widetilde{G}^{\triangleright m} \triangleright \widetilde{J}\right)(t, x, x', 0, 0)
    + (\lambda \Lip_\sigma)^{2n} \left(\widetilde{G}^{\triangleright n} \triangleright \rho\right)(t, x, x', 0, 0).
  \end{align*}
  Finally, we let $n\to \infty$ to complete the proof of part (iii). \medskip

  It remains to prove~\eqref{E:iintK}. Recall that $|\nu|(\dom) < \infty$. By
  Proposition~\ref{P:K-d}(i) or~\ref{Lem:K-n}(i), for any $t > 0$ and $x, x' \in
  \dom$,
  \begin{align*}
  \iint_{\dom^2} \mathcal{K}^a(t, x, x', y, y') |\nu|(\ud y) |\nu|(\ud y') < \infty.
  \end{align*}
  Then, by the definition~\eqref{Def:K} of $\mathcal{K}^a$ and Fubini's theorem,
  \begin{align*}
  a^{-2} \iint_{\dom^2} \mathcal{K}^a (t, x, x', y, y')\nu(\ud y)\nu(\ud y')
  = \sum_{n = 1}^\infty a^{2n-2} \iint_{\dom^2} \widetilde{G}^{\triangleright n}(t, x, x', y, y') \nu(\ud y) \nu(\ud y').
  \end{align*}
  We compare the right-hand side above with~\eqref{E:iintK}. For $n = 1$,
  \[
  \iint_{\dom^2} \widetilde{G}(t, x, x', y, y') \nu(\ud y) \nu(\ud y') = \widetilde{J}(t, x, x').
  \]
  For $n \ge 2$, by expressing $\widetilde{G}^{\triangleright n} =
  \widetilde{G}^{\triangleright (n-1)}\triangleright \widetilde{G}$ and
  interchanging the order of integration, we have
  \begin{align*}
    \MoveEqLeft \iint_{\dom^2} \widetilde{G}^{\triangleright n}(t, x, x', y, y') \nu(\ud y) \nu(\ud y')\\
    & = \iint_{\dom^2} \nu(\ud y) \nu(\ud y') \int_0^t \ud s \iint_{\dom^2} \ud z\, \ud z'\, \widetilde{G}^{\triangleright (n-1)}(t-s, x, x', z, z') f(z-z') \widetilde{G}(s, z, z', y, y') \\
    & = \int_0^t \ud s \iint_{\dom^2} \ud z\, \ud z'\, \widetilde{G}^{\triangleright (n-1)}(t-s, x, x', z, z') f(z-z') \widetilde{J}(s, z, z')                                              \\
    & = \left(\widetilde{G}^{\triangleright (n-1)} \triangleright \widetilde{J}\right)(t, x, x', 0, 0).
  \end{align*}
  This proves~\eqref{E:iintK} and completes the proof of
  Proposition~\ref{P:2p-corr}.
\end{proof}

\subsection{Estimation of the resolvent kernel functions \texorpdfstring{$\mathcal{K}$}{K}}

From Proposition~\ref{P:2p-corr}, we see that it is possible to estimate the
two-point correlation function once we have some useful bounds for
$\mathcal{K}^\lambda$. Therefore, our goal is to establish explicit upper and
lower bounds for $\mathcal{K}^\lambda$. To this end, we first prove a lemma
which strengthens Lemma~\ref{Lem:vol2}.

\begin{lemma}\label{Lem:volA}
  Let $\dom \subset \R^d$ be a bounded Lipschitz domain with the $\delta$-cone
  property. Let $0 < \eps_0 < 1$ be defined
  by~\eqref{Def:eps_0}, and recall that, for each $x \in \dom$, $V(x)$ is the
subset of $\dom$ defined in~\eqref{E:V(x)}. Then, there is a positive constant
$C$ such that the following property holds: For any $\eps \in (0, \eps_0]$, for
any $x \in \dom_\eps$, for any $r, s$ such that $0 < r \le s \le \eps_0/2$, for
any $z \in V(x) \cap B(x, s)$, we have
  \begin{align}\label{volA-lb}
  \op{Vol}(V(x) \cap B(x, s) \cap B(z, r)) \ge C(1\wedge r)^d.
  \end{align}
\end{lemma}

\begin{proof}
  First, suppose that case (2) in~\eqref{E:V(x)} holds. Then for all $z \in B(x,
  s)$ and $0 < r \le s \le \eps_0/2$, we have
  \begin{align*}
  V(x) \cap B(x, s) \cap B(z, r) = B(x, s) \cap B(z, r).
  \end{align*}
  This intersection clearly contains a ball of radius $r/3$, and hence has
  volume which is bounded from below by $C (r/3)^d$.

  Suppose that case (1) in~\eqref{E:V(x)} holds. Let
  \begin{align*}
     \mathscr{C}_1 = \{y \in \R^d: 0 < |y| < 1~\text{and}~y \cdot \xi > |y|\cos(\eps_0/2) \}.
  \end{align*}
  Note that $\mathscr{C}_1$ is a bounded convex domain, so it satisfies the
  $\bar \delta$-cone property for some $\bar \delta > 0$ by Proposition 2.4.4
  of~\cite{henrot.pierre:05:variation}. By scaling and translation, we see that
  \begin{align*}
    \mathrm{Vol}(V(x) \cap B(x, s) \cap B(z, r))
    = s^d \: \mathrm{Vol}\left(\overline{\mathscr{C}_1}\cap B((z-x)/s, r/s)\right)
  \end{align*}
  with $(z-x)/s \in \overline{\mathscr{C}_1}$ and $0 < r/s \le 1$. So, it suffices to prove
  the existence of a constant $C > 0$ such that
  \begin{align}\label{C1-lb}
    \mathrm{Vol}(\overline{\mathscr{C}_1} \cap B(w, r)) \ge C r^d, \quad
    \text{for all $w \in \overline{\mathscr{C}_1}$ and $r \in(0,1]$.}
  \end{align}

  To prove~\eqref{C1-lb}, we first consider the case that $0 < r \le
  \bar\delta$. If $\mathrm{dist}(w, \partial \mathscr{C}_1) \ge \bar\delta$,
  then $B(w, r) \subset \overline{\mathscr{C}_1}$ and hence
  \begin{align*}
    \mathrm{Vol}(\overline{\mathscr{C}_1} \cap B(w, r)) = C_1 r^d.
  \end{align*}
  If $\mathrm{dist}(w, \partial \mathscr{C}_1) < \bar\delta$, then $w \in
  \overline{\mathscr{C}_1} \cap B(v, \bar\delta)$ for some $v \in \partial \mathscr{C}_1$.
  By the $\bar\delta$-cone property of $\mathscr{C}_1$, we can find a unit
  vector $\eta = \eta_v \in \R^d$ such that $\mathscr{C}(w, \eta, \bar\delta)
  \subset \mathscr{C}_1$. Then,
  \begin{align*}
    \mathrm{Vol}(\overline{\mathscr{C}_1} \cap B(w, r))
    & \ge \mathrm{Vol}(\mathscr{C}(w, \eta, \bar\delta) \cap B(w, r))                                       \\
    & = \mathrm{Vol}\{y \in \R^d: 0 < |y-w| < r \text{ and } (y-w) \cdot \eta \ge |y-w| \cos(\eps_0/2) \} \\
    & \ge C_2 r^d.
  \end{align*}
  Finally, consider the case that $\bar\delta < r \le 1$. We have $\overline{\mathscr{C}_1}
  \cap B(w, r) \supset \overline{\mathscr{C}_1} \cap B(w, \bar\delta)$. Then, by the first
  case that we just proved,
  \begin{align*}
    \mathrm{Vol}(\overline{\mathscr{C}_1} \cap B(w, r)) \ge (C_1 \wedge C_2)\bar\delta^d \ge (C_1 \wedge C_2) \bar \delta^d r^d.
  \end{align*}
  This completes the proof of Lemma~\ref{Lem:volA}.
\end{proof}

Now, we are ready to establish upper and lower bounds for
$\mathcal{K}^\lambda_D$ and $\mathcal{K}^\lambda_N$.

\begin{proposition}\label{P:K-d}
  Let $\dom$ be a bounded Lipschitz domain. Then:
  \begin{enumerate}
  \item[(i)] There exist positive finite constants $C$, $c$, $c'$ such that for all
    $t > 0$, for all $x, x', y, y' \in \dom$,
    \begin{equation}\label{K-d-ub}
      \mathcal{K}_D^\lambda(t, x, x', y, y')
      \le \frac{C \lambda^2}{1\wedge t^d} e^{-c_1\left(\frac{|x-y|^2}{t} + \frac{|x'-y'|^2}{t}\right)}
      e^{2t\left(c \lambda^2 + c' \lambda^{\frac{4}{2-\beta}}-\mu_1\right)}.
    \end{equation}
  \item[(ii)] For all $\eps > 0$ small, there exist positive finite constants
    $\overline{C}$, $\overline{c}$, $\widetilde{c}$ depending on $\eps$ such that
    for all $t > 0$, for all $x, x', y, y' \in \dom_\eps$,
    \begin{equation}\label{K-d-lb}
      \qquad \mathcal{K}_D^\lambda(t, x, x', y, y')
      \ge \frac{\overline{C} \lambda^2}{1 \wedge t^d} e^{-16c_2\frac{|x-x'|^2}{t}}e^{-12c_2\left(\frac{|x-y|^2}{t} + \frac{|x'-y'|^2}{t}\right)}
      e^{2t\left(\overline{c} \lambda^2 + \widetilde{c} \lambda^{\frac{4}{2-\beta}}-\mu_1\right)}
    \end{equation}
    and $\overline{C}\to 0$, $\bar{c} \to 0$ and $\widetilde{c}\to 0$ as $\eps \to 0$.
  \end{enumerate}
\end{proposition}

We first make a few remarks:

\begin{remark}
  (1) Under the conditions of Proposition~\ref{P:2p-corr}, for the delta initial
  condition $\nu = \delta_y$, we have
  \begin{align*}
    (\lambda l_\sigma)^{-2}\,\mathcal{K}^\lambda_D(t, x, x', y, y)
    \le \mathbb{E}(u(t, x) u(t, x'))
    \le (\lambda L_\sigma)^{-2}\,\mathcal{K}^\lambda_D(t, x, x', y, y).
  \end{align*}
  (2) Proposition~\ref{P:K-d} applies to all bounded Lipschitz domains. But in
  case of $C^{1, \alpha}$-domains, we will improve the bounds~\eqref{K-d-ub} and \eqref{K-d-lb} in Proposition~\ref{P:K-C1alpha} so that the moment
  estimates will be consistent with Dirichlet condition at the boundary,
  namely,
  \begin{align*}
    \lim_{\text{$x$ or $x'$}\to \partial \dom}\mathbb{E}(u(t, x) u(t, x')) = 0.
  \end{align*}
  (3) In case of $\dom = \R^d$, the factor $\exp(-c|x-x'|^2 /t)$
  in~\eqref{K-d-lb} also appears in the lower bound in Lemma 2.7
  of~\cite{chen.kim:19:nonlinear}. We think that a sharp upper bound for
  $\mathcal{K}_D^\lambda$ would have this extra exponential factor as well.
\end{remark}

\begin{proof}[Proof of Proposition~\ref{P:K-d}]
  (i). We claim that, for all $n \ge 1$, $t > 0$ and $x, x', y, y' \in \dom$,
  \begin{align}\label{claim:K-d-ub}
    \widetilde{G}_D^{\triangleright n}(t, x, x', y, y')
    \le C^n \frac{e^{-2\mu_1 t}}{1\wedge t^d}
    e^{-c_1\Big(\frac{|x-y|^2}{t} + \frac{|x'-y'|^2}{t}\Big)}\widetilde{h}_{n-1}(t),
  \end{align}
  where $\widetilde{h}_n(t)$ is the iterated integral defined in
  Lemma~\ref{Lem:tildeh}. We prove this by induction. For $n = 1$, this follows
  from the upper bound in~\eqref{E:G-d}. Assume that~\eqref{claim:K-d-ub} holds
  for some $n \ge 1$. Then, by the induction hypothesis, the upper bound
  in~\eqref{E:G-d} and Lemma~\ref{Lem:iintGGf-d}(i),
  \begin{align*}
    \MoveEqLeft \widetilde{G}_D^{\triangleright (n+1)}(t, x, x', y, y')\\
    &= \int_0^t \ud s \iint_{\dom^2} \widetilde{G}_D(t-s, x, x', z, z') \widetilde{G}_D^{\triangleright n}(s, z, z', y, y') f(z-z')\, \ud z\,\ud z' \\
    &\le C^n \int_0^t \ud s\, \widetilde{h}_{n-1}(s) \iint_{\dom^2} \ud z\, \ud z'\, \widetilde{G}_D(t-s, x, x', z, z')
    \frac{e^{-2\mu_1 s}}{1\wedge s^d}e^{-c_1\frac{|z-y|^2+|z'-y'|^2}{s}}
    f(z-z') \, \ud z\, \ud z'\\
    & \le C^{n+1} \frac{e^{-2\mu_1 t}}{1\wedge t^d}
    e^{-c_1\Big(\frac{|x-y|^2}{t} + \frac{|x'-y'|^2}{t}\Big)}\int_0^t \widetilde{h}_{n-1}(s) \left(1\wedge \frac{(t-s)s}{t}\right)^{-\beta/2} \ud s\\
    & = C^{n+1} \frac{e^{-2\mu_1 t}}{1\wedge t^d}
    e^{-c_1\Big(\frac{|x-y|^2}{t} + \frac{|x'-y'|^2}{t}\Big)} \widetilde{h}_n(t).
  \end{align*}
  This proves the claim~\eqref{claim:K-d-ub}. Then, by Lemmas~\ref{Lem:tildeh}
  and~\ref{Lem:I},
  \begin{align*}
    \mathcal{K}^\lambda_D(t, x, x', y, y')
    & = \sum_{n = 1}^\infty \lambda^{2n} \widetilde{G}_D^{\triangleright n}(t, x, x', y, y')\\
    & \le C\lambda^2 \frac{e^{-2\mu_1 t}}{1\wedge t^d}
      e^{-c_1\Big(\frac{|x-y|^2}{t} + \frac{|x'-y'|^2}{t}\Big)}
      \sum_{n = 1}^\infty (C \lambda^2)^{n-1} \widetilde{h}_{n-1}(t)\\
    & \le C' \lambda^2 \frac{e^{-2\mu_1 t}}{1\wedge t^d}
      e^{-c_1\Big(\frac{|x-y|^2}{t} + \frac{|x'-y'|^2}{t}\Big)}
      e^{t \Big(c \lambda^2 + c' \lambda^{\frac{4}{2-\beta}}\Big)}.
  \end{align*}
  This proves the upper bound~\eqref{K-d-ub}. \medskip

  (ii). To prove the lower bound~\eqref{K-d-lb}, we need to derive lower bounds
  for $\widetilde G_D^{\triangleright n}(t, x, x', y, y')$ for each $n \ge 1$.
  Let $t > 0$ and $x, x', y, y' \in \dom_\eps$. First, by the lower bounds
  in~\eqref{E:G-d} and~\eqref{E:Phi-bd},
  \begin{align*}
    \widetilde{G}_D(t, x, x', y, y') \ge C_\eps^2 \frac{e^{-2\mu_1 t}}{1 \wedge t^d} e^{-c_2 \Big(\frac{|x-y|^2}{t} + \frac{|x'-y'|^2}{t} \Big)}.
  \end{align*}
  For $n = 2$, we have
  \begin{align*}
    \widetilde{G}_D^{\triangleright 2}(t, x, x', y, y')
    & \ge \int_{t/4}^{t/2} \ud s \iint_{[V(x) \cap B(x, \sqrt{t})]^2} \ud z \, \ud z'\, \\
    & \quad \times \widetilde{G}_D(t-s, x, x', z, z') f(z-z')\widetilde{G}_D(s, z, z', y, y'),
  \end{align*}
  where $V(x) \subset \dom$ is defined in~\eqref{E:V(x)}. By~\eqref{E:G-d}, for
  $t/4 < s < t/2$, we have
  \begin{align*}
        \widetilde{G}_D(t-s, x, x', z, z')
    & \ge C_2^2\: \left(1 \wedge \Phi_1(x)\right)\left(1 \wedge \Phi_1(x')\right)\left(1 \wedge \Phi_1(z)\right)\left(1 \wedge \Phi_1(z')\right) \\
    & \quad \times \frac{e^{-2\mu_1(t-s)}}{1 \wedge (t-s)^d} e^{-c_2\frac{|x-z|^2+|x'-z'|^2}{t/2}} \shortintertext{and}
        \widetilde{G}_D(s, z, z', y, y')
    & \ge C_2^2\: \left(1 \wedge \Phi_1(z)\right)\left(1 \wedge \Phi_1(z')\right)\left(1 \wedge \Phi_1(y)\right)\left(1 \wedge \Phi_1(y')\right) \\
    & \quad \times \frac{e^{-2\mu_1s}}{1 \wedge s^d} e^{-c_2\frac{|z-y|^2+|z'-y'|^2}{t/4}}.
  \end{align*}
  For $z, z' \in V(x) \cap B(x, \sqrt{t})$, by the triangle inequality and
  Cauchy--Schwarz inequality, we have
  \begin{align*}
    |x-z|^2 + |x'-z'|^2 & \le |x-z|^2 + 2(|x'-x|^2 + |x-z'|^2) \\
                        & \le 2|x-x'|^2 + 3t
  \end{align*}
  and
  \begin{align*}
    |z-y|^2 + |z'-y'|^2 & \le 2(|z-x|^2 + |x-y|^2) + 3(|z'-x|^2 + |x-x'|^2 + |x'-y'|^2) \\
                        & \le 5t + 2|x-y|^2 + 3|x'-y'|^2 + 3|x-x'|^2.
  \end{align*}
  By applying the bounds above,~\eqref{E:Phi-bd} and Lemma~\ref{Lem:vol2}, we
  get that
  \begin{align*}
    \widetilde{G}_D^{\triangleright 2}(t, x, x', y, y')
    \ge C_\eps^4 \frac{e^{-2\mu_1 t}}{1 \wedge t^d} e^{-16c_2\frac{|x-x'|^2}{t}}e^{-12c_2\frac{|x-y|^2 + |x'-y'|^2}{t}} t^{1-\beta/2}.
  \end{align*}
  For $n \ge 2$, expanding $\widetilde{G}_D^{\triangleright (n+1)}$ in its full
  integral form, we have
  \begin{align*}
    \MoveEqLeft \widetilde{G}_D^{\triangleright (n+1)}(t, x, x', y, y')\\
    \ge & \int_0^t \ud s_1 \int_0^{s_1} \ud s_2 \cdots \int_0^{s_{n-2}} \ud s_{n-1}\int_0^{s_{n-1}} \ud s_{n}                                                          \\
        & \times \iint_{\dom^2} \ud z_1\, \ud z_1'\, \widetilde{G}_D(t-s_1, x, x', z_1, z_1') f(z_1 - z_1')                                                            \\
        & \times \iint_{\dom^2} \ud z_2\, \ud z_2'\, \widetilde{G}_D(s_1-s_2, z_1, z_1', z_2, z_2') f(z_2 - z_2')                                                      \\
        & \times \cdots \times \iint_{\dom^2} \ud z_{n-1} \, \ud z_{n-1}'\, \widetilde{G}_D(s_{n-2}-s_{n-1}, z_{n-2}, z_{n-2}', z_{n-1}, z_{n-1}') f(z_{n-1}-z_{n-1}') \\
        & \times \iint_{\dom^2} \ud z_{n} \, \ud z_{n}'\,
          \widetilde{G}_D(s_{n-1}-s_{n}, z_{n-1}, z_{n-1}', z_{n}, z_{n}') f(z_{n}-z_{n}')
          \widetilde{G}_D(s_{n}, z_{n}, z_{n}', y, y').
  \end{align*}
  Then, we derive a lower bound by integrating on the smaller intervals
  \begin{align}\label{interval-si}
    \begin{split}
      s_1 & \in \left[ \left(1-\frac{1}{4n}\right)\frac t 2, \frac{t}{2}\right],     \\
      s_2 & \in \left[\left(1-\frac{1}{2n}\right)\frac t 2, s_1-\frac{t}{8n}\right], \\
      s_3 & \in \left[\left(1-\frac{2}{2n}\right)\frac t 2, s_2 - (s_1-s_2)\right],  \\
          & \vdotswithin{\in}                                                        \\
      s_n & \in \left[\left(1-\frac{n-1}{2n}\right)\frac t 2, s_{n-1}-(s_{n-2}-s_{n-1})\right].
    \end{split}
  \end{align}
  Let
  $0 < \eps_0 < 1$ be given by~\eqref{Def:eps_0}. For each $\eps \in (0,
  \eps_0)$ and $x \in \dom_\eps$, recall the subset $V(x) \subset \dom$ defined
  by~\eqref{E:V(x)}. Consider
  \begin{alignat*}{2}
    z_1, z_1' & \in A_1  &  & \coloneqq V(x) \cap B(x, \sqrt{(t-s_1)/(5n)} \wedge (\eps_0/2)),                                                       \\
    z_2       & \in A_2  &  & \coloneqq V(x) \cap B(x, \sqrt{s_1-s_2} \wedge (\eps_0/2)) \cap B(z_1 , \sqrt{s_1-s_2} \wedge (\eps_0/2)),             \\
    z_2'      & \in A_2' &  & \coloneqq V(x) \cap B(x, \sqrt{s_1-s_2} \wedge (\eps_0/2)) \cap B(z_1', \sqrt{s_1-s_2} \wedge (\eps_0/2)),             \\
              &          &  & \vdotswithin{\coloneqq}                                                                                                \\
    z_n       & \in A_n  &  & \coloneqq V(x) \cap B(x, \sqrt{s_{n-1}-s_n} \wedge (\eps_0/2)) \cap B(z_{n-1} , \sqrt{s_{n-1}-s_n} \wedge (\eps_0/2)), \\
    z_n'      & \in A_n' &  & \coloneqq V(x) \cap B(x, \sqrt{s_{n-1}-s_n} \wedge (\eps_0/2)) \cap B(z_{n-1}', \sqrt{s_{n-1}-s_n} \wedge (\eps_0/2)).
  \end{alignat*}
  Note that by \eqref{interval-si}, we have $(t-s_1)/(5n) \le t/(8n) \le s_1-s_2
  \le s_2-s_3 \le \dots \le s_{n-1}-s_n$, which ensures that for $2 \le i \le
  n$, both $z_{i-1}$ and $z_{i-1}'$ lie in $V(x) \cap B(x,
  \sqrt{s_{i-1}-s_i}\wedge (\eps_0/2))$. Then, by Lemma~\ref{Lem:volA}, for $2
  \le i \le n$,
  \begin{align*}
    \mathrm{Vol}(A_i) \wedge \mathrm{Vol}(A_i') \ge C (\sqrt{s_{i-1}-s_i} \wedge (\eps_0/2))^d
    \ge C (\eps_0/2)^d ((s_{i-1}-s_i) \wedge 1)^{d/2}.
  \end{align*}
  By Lemma~\ref{Lem:vol2} and \eqref{E:Phi-bd}, we have
  \begin{gather*}
    \mathrm{Vol}(A_1) \ge C \left(\sqrt{\frac{t-s_1}{5n}}\wedge (\eps_0/2)\right)^d \ge C (\eps_0/2)^d (5n)^{-d/2} ((t-s_1) \wedge 1)^{d/2}, \\
    1\wedge \Phi_1(z) \ge 1 \wedge (c_0^{-1} \eps^{a_1}) \quad \text{for all $z \in V(x)$}.
  \end{gather*}
  Also, on $A_1 \times A_1$, we have $f(z_1 - z_1') \ge C(((t-s_1)/n) \wedge 1)^{-\beta/2}$,
  and on $A_i \times A_i'$, where  $2 \le i \le n$, we have $f(z_i - z_i') \ge C((s_{i-1}-s_i) \wedge 1)^{-\beta/2}$.
  Then, by the lower bound in~\eqref{E:G-d} and the inequalities
  \begin{align*}
    | z_n-z_{n-1} | ^2 + | z_n' - z_{n-1}' |^2 & \le 2(s_{n-1}-s_n), \\
    |z_n - y|^2 + |z_n' - y'|^2                & \le 5(s_{n-1}-s_n) + 2|x-y|^2 + 3|x'-y'|^2 + 3|x-x'|^2,
  \end{align*}
  we have
  \begin{gather*}
    \iint_{A_n\times A_n'} \ud z_{n} \, \ud z_{n}'\,
    \widetilde{G}_D(s_{n-1}-s_{n}, z_{n-1}, z_{n-1}', z_{n}, z_{n}') f(z_{n}-z_{n}')
    \widetilde{G}_D(s_{n}, z_{n}, z_{n}', y, y') \\
    \ge C_\eps^4 e^{-2\mu_1 s_{n-1}}(1\wedge (s_{n-1}-s_n))^{-\frac \beta 2}
    \frac{1}{1 \wedge t^d}
    e^{-c_2\frac{3|x-x'|^2 + 3|x-y|^2 + 3|x'-y'|^2 + 5s_{n-1}}{s_n}}.
  \end{gather*}
  For $i = 2, \dots, n-1$, by $|z_i - z_{i-1}|^2 + |z_i' - z_{i-1}'|^2 \le
  2(s_{i-1}-s_i)$, we have
  \begin{align*}
    \iint_{A_i\times A_i'} \ud z_{i} \, \ud z_{i}'\, \widetilde{G}_D(s_{i-1}-s_{i}, z_{i-1}, z_{i-1}', z_{i}, z_{i}') f(z_{i}-z_{i}')
    \ge C_\eps^2 e^{-2\mu_1 (s_{i-1}-s_i)} (1\wedge (s_{i-1}-s_i))^{-\frac \beta 2}.
  \end{align*}
  For $i = 1$, by the inequality
  \begin{align*}
    | z_1 - x | ^2 + | z_1' - x' | ^2 &\le 3(t-s_1) + 2|x-x'|^2,
  \end{align*}
  we have
  \begin{align*}
    \iint_{A_1\times A_1} \ud z_{1} \, \ud z_{1}'\, \widetilde{G}_D(t-s_{1}, x, x', z_1, z_1') f(z_1-z_1')
    \ge \frac{C_\eps^2}{n^{d/2}} e^{-2\mu_1 (t-s_1)} e^{-2c_2\frac{|x-x'|^2}{t-s_1}} \left(1\wedge \frac{t-s_1}{n} \right)^{-\frac \beta 2}.
  \end{align*}
  Combining these estimates and using~\eqref{E:Crho}, we get that
  \begin{align*}
    \MoveEqLeft \widetilde{G}_D^{\triangleright (n+1)}(t, x, x', y, y')                                                                               \\
   & \ge C_\eps^{2(n+1)} \frac{e^{-2\mu_1 t}}{1 \wedge t^d} \int_{(1-\frac{1}{4n})\frac{t}{2}}^{\frac{t}{2}} \ud s_1 \, \left(1 + \left(\frac{t-s_1}{n}\right)^{-\frac \beta 2}\right) e^{-2c_2 \frac{|x-x'|^2}{t-s_1}}       \\
   & \quad \times \int_{(1-\frac{1}{2n})\frac{t}{2}}^{s_1-\frac{t}{8n}} \ud s_2\, \left(1+ (s_1-s_2)^{-\frac \beta 2}\right)\times
   \int_{(1-\frac{2}{2n})\frac{t}{2}}^{s_2-(s_1-s_2)} \ud s_3\, \left(1+ (s_2-s_3)^{-\frac \beta 2}\right)\times \cdots \\
   & \quad \cdots \times \int_{(1-\frac{n-2}{2n})\frac{t}{2}}^{s_{n-2}-(s_{n-3}-s_{n-2})} \ud s_{n-1} \, \left(1 + (s_{n-2}-s_{n-1})^{-\frac \beta 2}\right) \\
   & \quad \times \int_{(1-\frac{n-1}{2n})\frac{t}{2}}^{s_{n-1}-(s_{n-2}-s_{n-1})} \ud s_n\, \left(1 + (s_{n-1}-s_n)^{-\frac\beta 2}\right) e^{-c_2\frac{3|x-x'|^2 + 3|x-y|^2 + 3|x'-y'|^2 + 5t}{t/4}}.
  \end{align*}
  Let $I$ denote the above multiple integral for $s_1, \dots, s_n$. By
  \eqref{interval-si}, we have $s_{i-1}-s_i \ge \frac{t}{8n}$ and
  $s_{i-1}-(1-\frac{i}{2n})\frac{t}{2} \ge (1-\frac{i-1}{2n})\frac{t}{2} -
  (1-\frac{i}{2n})\frac{t}{2} = \frac{t}{4n}$ for $2 \le i \le n$. Fixing $s_1
  \in \left[\left(1-\frac{1}{4n} \right)\frac{t}{2}, \frac{t}{2}\right]$, by the
  change of variables $s_n \mapsto s_{n-1}-s_n$, $s_{n-1} \mapsto
  s_{n-2}-s_{n-1}$, $\dots$, $s_2 \mapsto s_1-s_2$, we have
  \begin{align*}
    s_n & \in \left[s_{n-2}-s_{n-1}, s_{n-1}-\left(1-\frac{n-1}{2n}\right)\frac t 2\right] \supset \left[\frac{t}{8n}, \frac{t}{4n}\right], \\
        & \vdotswithin{\in}                                                                                                                 \\
    s_3 & \in \left[s_1-s_2, s_2 - \left(1-\frac{2}{2n}\right)\frac t 2\right] \supset \left[\frac{t}{8n}, \frac{t}{4n}\right],             \\
    s_2 & \in \left[\frac{t}{8n}, s_1 - \left(1-\frac{1}{2n}\right)\frac t 2\right] \supset \left[\frac{t}{8n}, \frac{t}{4n}\right].
  \end{align*}
  It follows that
  \begin{align*}
    I \ge \:
    & e^{-10c_2} e^{-16c_2\frac{|x-x'|^2}{t}} e^{-12c_2 \frac{|x-y|^2+|x'-y'|^2}{t}} \\
    & \times \int_{(1- \frac{1}{4n})\frac t 2}^{\frac{t}{2}} \ud s_1 \, \left(1 + \left(\frac{t-s_1}{n}\right)^{-\frac \beta 2}\right)
      \times \int_{\frac{t}{8n}}^{\frac{t}{4n}} \ud s_2 \, \left(1+s_2^{-\frac \beta 2}\right)\times \cdots \\
    & \times \int_{\frac{t}{8n}}^{\frac{t}{4n}} \ud s_{n-1} \, \left(1+s_{n-1}^{-\frac \beta 2}\right)
      \times \int_{\frac{t}{8n}}^{\frac{t}{4n}} \ud s_n \, \left(1+s_n^{-\frac \beta 2}\right)\\
    & \eqqcolon e^{-10c_2} e^{-16c_2\frac{|x-x'|^2}{t}} e^{-12c_2 \frac{|x-y|^2+|x'-y'|^2}{t}} \prod_{i=1}^{n} I_i.
  \end{align*}
  The above $\ud s_2\cdots\ud s_n$ integrals can be evaluated explicitly, which
  is equal to
  \begin{align*}
    \prod_{i=2}^{n} I_i
     = \left(\frac{t}{8n} + \frac{\left(\frac{t}{4n}\right)^{1-\frac \beta 2}-\left(\frac{t}{8n}\right)^{1-\frac \beta 2}}{1-\frac \beta 2}\right)^{n-1}
     \ge c^{n-1} \left(\frac{t}{n}\right)^{n-1} \left(1+\left(\frac{t}{n}\right)^{-\frac \beta 2}\right)^{n-1}. 
  \end{align*}
  As for the $\ud s_1$ integral, we have that
  \begin{align*}
    I_1 \ge \int_{(1- \frac{1}{4n})\frac t 2}^{\frac{t}{2}} \ud s_1 \, \left(1 + \left(\frac{t- t/2}{n}\right)^{-\frac \beta 2}\right)  
        \ge c \: \frac{t}{n} \left(1+\left(\frac{t}{n}\right)^{-\frac \beta 2}\right).
  \end{align*}
  Therefore, we see that
  \begin{align*}
    I & \ge c^n n^{-\beta/2} e^{-10c_2} e^{-16c_2\frac{|x-x'|^2}{t}} e^{-12c_2 \frac{|x-y|^2+|x'-y'|^2}{t}} \left(\frac tn\right)^n \left(1 + \left(\frac{t}{n}\right)^{-\beta/2} \right)^n \\
      & = c^n n^{-\beta/2} e^{-10c_2} e^{-16c_2\frac{|x-x'|^2}{t}} e^{-12c_2 \frac{|x-y|^2+|x'-y'|^2}{t}} \sum_{k = 0}^n \frac{n!}{k!(n-k)!} \left(\frac tn\right)^{n-k\beta/2}.
  \end{align*}
  By Stirling's formula, $\frac{n^{k\beta/2}}{n^n} \ge C^n
  \frac{(k!)^{\beta/2}}{n!}$ for all $k = 0, 1, \dots, n$ and $n \ge 2$. It
  follows that
  \begin{align*}
    \widetilde{G}_D^{\triangleright (n+1)}(t, x, x', y, y')
    \ge C_\eps^{2(n+1)} c^n \frac{e^{-2\mu_1 t}}{1 \wedge t^d} e^{-16c_2\frac{|x-x'|^2}{t}}e^{-12c_2 \frac{|x-y|^2 + |x'-y'|^2}{t}} \sum_{k = 0}^n \frac{t^{n-k\beta/2}}{(n-k)! (k!)^{1-\beta/2}}.
  \end{align*}
  Finally, recalling the definition of $\mathcal{K}^\lambda_D$, interchanging
  the order of summation, and using~\eqref{E:exp-bd}, we get that
  \begin{align*}
    \MoveEqLeft \mathcal{K}^\lambda_D(t, x, x', y, y')
       = \lambda^2 \sum_{n = 0}^\infty \lambda^{2n} \widetilde{G}_D^{\triangleright (n+1)}(t, x, x', y, y')\\
    & \ge C_\eps^2 \lambda^2 \frac{e^{-2\mu_1 t}}{1 \wedge t^d}e^{-16c_2\frac{|x-x'|^2}{t}} e^{-12c_2\frac{|x-y|^2 + |x'-y'|^2}{t}}
      \sum_{n = 0}^\infty {(C_\eps^2 c\lambda^2 t)}^n
      \sum_{k = 0}^n \frac{t^{-k\beta/2}}{(n-k)! (k!)^{1-\beta/2}}\\
    & \ge C_\eps^2 C \lambda^2 \frac{e^{-2\mu_1 t}}{1 \wedge t^d} e^{-16c_2\frac{|x-x'|^2}{t}}e^{-12c_2\frac{|x-y|^2 + |x'-y'|^2}{t}}
      e^{t\Big(K \lambda^2 + K' \lambda^{\frac{4}{2-\beta}} \Big)}.
  \end{align*}
  The proof of Proposition~\ref{P:K-d} is complete.
\end{proof}

\begin{proposition}\label{Lem:K-n}
  Let $\dom$ be a bounded Lipschitz domain. Then:
  \begin{enumerate}
    \item[(i)] There exist positive finite constants $C$, $c$, $c'$ such that
      for all $t > 0$, for all $x, x', y, y' \in \dom$,
      \begin{equation}\label{K-n-ub}
        \mathcal{K}_N^\lambda(t, x, x', y, y')
        \le \frac{C\lambda^2}{1\wedge t^d} e^{-c_3\left(\frac{|x-y|^2}{t} + \frac{|x'-y'|^2}{t}\right)}
        e^{t\Big(c\lambda^2 + c' \lambda^{\frac{4}{2-\beta}}\Big)}.
      \end{equation}
    \item[(ii)] If $\dom$ is convex (or if~\eqref{E:G-n-lower} holds), then there
      exist positive finite constants $\overline{C}$, $\overline{c}$,
      $\widetilde{c}$ such that for all $t > 0$, for all $x, x', y, y' \in
      \dom$,
      \begin{equation}\label{K-n-lb}
        \mathcal{K}^\lambda_N(t, x, x', y, y')
        \ge \frac{\overline{C} \lambda^2}{1\wedge t^d} e^{-16c_4 \frac{|x-x'|^2}{t}} e^{-12c_4\left(\frac{|x-y|^2}{t} + \frac{|x'-y'|^2}{t}\right)}
        e^{t\Big(\overline{c}\lambda^2 + \widetilde{c} \lambda^{\frac{4}{2-\beta}}\Big)}.
      \end{equation}
  \end{enumerate}
\end{proposition}

\begin{proof}
The proof is similar to that of Lemma~\ref{P:K-d}.
\end{proof}

\subsection{Proof of Theorem~\ref{T:corr-bd}}

\begin{proof}[Proof of Theorem~\ref{T:corr-bd}]
  The correlation bounds~\eqref{E:corr-ub} and~\eqref{E:corr-lb} under Dirichlet
  condition (or Neumann condition, respectively) follows immediately from
  Propositions~\ref{P:2p-corr} and~\ref{P:K-d} (or~\ref{P:2p-corr}
  and~\ref{Lem:K-n}, respectively).
\end{proof}

\section{The case of bounded \texorpdfstring{$C^{1, \alpha}$}{C one alpha}-domains with Dirichlet condition and some
variations}\label{S:Holder-Domain}

In the case of $C^{1, \alpha}$-domains, one can expect better
  moment estimates under Dirichlet boundary condition because the heat kernel
  estimates in~\eqref{E:G-d} hold with $a_1 = a_2 = 1$. This implies that the
  Dirichlet heat kernel estimates in~\eqref{E:G-d} are sharp, yielding matching
upper and lower bounds with the same factor $\left(1\wedge \frac{\Phi_1(x)}{1
\wedge t^{1/2}}\right)\left(1 \wedge \frac{\Phi_1(y)}{1 \wedge t^{1/2}}\right)$.

\subsection{Better estimates for the resolvent kernel \texorpdfstring{$\mathcal{K}$}{K}}

In this part, we study the case of bounded $C^{1, \alpha}$-domains with
Dirichlet boundary condition. For the points $x, x', y, y'$ that are close to
$\partial \dom$, the lemma below provides more precise estimate than what one
gets from the upper bound in Lemma~\ref{Lem:iintGGf-d}.

\begin{lemma}\label{Lem:G^2}
  If $\dom$ is a bounded $C^{1, \alpha}$-domain for some $\alpha > 0$, then
  there exists a positive finite constant $C$ such that for all $t > 0$, for all
  $x, x', y, y' \in \dom$,
  \begin{align}\label{E:G^2}
    &\int_0^t \ud s \iint_{\dom^2} \ud z \, \ud z' \, \Psi(t-s, x) \Psi(t-s, x') \Psi(t-s, z) \Psi(t-s, z')
    \frac{e^{-2\mu_1 (t-s)}}{1 \wedge (t-s)^d} e^{-c_1 \frac{|x-z|^2 + |x'-z'|^2}{t-s}}\notag\\
    &\qquad\times f(z-z') \Psi(s, z) \Psi(s, z') \Psi(s, y) \Psi(s, y') \frac{e^{-2\mu_1 s}}{1 \wedge s^d}
    e^{-\frac{2c_1}{3} \frac{|z-y|^2 + |z'-y'|^2}{s}}\notag\\
    &\le C \Psi(t, x) \Psi(t, x') \Psi(t, y) \Psi(t, y') \frac{e^{-2\mu_1 t}}{1 \wedge t^d}
    e^{-\frac{2c_1}{3}\frac{|x-y|^2 + |x'-y'|^2}{t}} \int_0^t \left(1 \wedge \frac{(t-s)s}{t}\right)^{-\beta/2} \ud s,
  \end{align}
  where $\Psi(t, x)$ is as defined in~\eqref{E:Psi}.
\end{lemma}

\begin{proof}
  We first derive the following bound: for all $s, r > 0$ and $v, w \in \dom$,
  \begin{align}\label{Psi-bd}
    \Psi(s, v) \Psi(r, w) e^{-\frac{c_1}{6} \frac{|v-w|^2}{r}} \le C_0 \Psi(s, w).
  \end{align}
  Since $\dom$ is a $C^{1, \alpha}$-domain for some $\alpha > 0$,~\eqref{E:Phi-bd}
  holds with $a_1 = a_2 = 1$. By~\eqref{E:Phi-bd} and the triangle inequality,
  \begin{align*}
    \Phi_1(v) \le c_0(\op{dist}(w, \partial \dom) + |v-w|) \le c_0^2\Phi_1(w) + c_0|v-w|.
  \end{align*}
  It follows that
  \begin{align*}
      \Psi(s, v) = 1 \wedge \frac{\Phi_1(v)}{1\wedge s^{1/2}}
    & \le 1 \wedge \left(\frac{c_0^2\Phi_1(w)+c_0|v-w|}{\Phi_1(w)}\cdot \frac{\Phi_1(w)}{1 \wedge s^{1/2}}\right) \\
    & \le \left(c_0^2 + \frac{c_0|v-w|}{\Phi_1(w)}\right) \Psi(s, w).
  \end{align*}
  Thus,
  \begin{align*}
      \Psi(s, v) \Psi(r, w)
    & \le \left(c_0^2 + \frac{c_0|v-w|}{\Phi_1(w)}\right) \left(1\wedge \frac{\Phi_1(w)}{1\wedge r^{1/2}}\right) \Psi(s, w) \\
    & \le \left(c_0^2 + \frac{c_0|v-w|}{1 \wedge r^{1/2}}\right) \Psi(s, w).
  \end{align*}
  Moreover, since the function $xe^{-x^2}$ is bounded, we have
  \begin{align*}
      \Psi(s, v) \Psi(r, w) e^{-\frac{c_1}{6} \frac{|v-w|^2}{r}}
    & \le \left(c_0^2 + c_0 |v-w| + \frac{c_0|v-w|}{r^{1/2}}e^{-\frac{c_1}{6} \frac{|v-w|^2}{r}}\right) \Psi(s, w) \\
    & \le C_0 \Psi(s, w),
  \end{align*}
  where $C_0$ is a finite constant depending on $c_0$, $c_1$ and the diameter of
  $\dom$. This proves~\eqref{Psi-bd}.

  In order to prove the lemma, we split the integral over $[0, t]$ into the sum
  of the integral over $[t/2, t]$ and the integral over $[0, t/2]$. Denote these
  two integrals by $I_1$ and $I_2$, respectively. By~\eqref{Psi-bd}, we have the
  following two bounds
  \begin{align*}
    \Psi(t-s, x)\Psi(s, z) e^{-\frac{c_1}{3}\frac{|x-z|^2}{t-s}}     & \le C_0 \Psi(s, x), \\
    \Psi(t-s, x')\Psi(s, z') e^{-\frac{c_1}{3}\frac{|x'-z'|^2}{t-s}} & \le C_0 \Psi(s, x').
  \end{align*}
  These bounds and $\Psi(t-s, z) \Psi(t-s, z') \le 1$ imply that
  \begin{align*}
    I_1 &\le C_0^2 \int_{t/2}^t \ud s \, \Psi(s, x) \Psi(s, x') \Psi(s, y) \Psi(s, y') \\
        & \quad \times \iint_{\dom^2} \ud z \, \ud z' \, \frac{e^{-2\mu_1(t-s)}}{1 \wedge (t-s)^d} e^{-\frac{2c_1}{3} \frac{|x-z|^2 + |x'-z'|^2}{t-s}} f(z-z') \frac{e^{-2\mu_1 s}}{1 \wedge s^d} e^{-\frac{2c_1}{3}\frac{|z-y|^2 + |z'-y'|^2}{s}}.
  \end{align*}
  Then, using the bound $\Psi(s, x) \le \Psi(t/2, x) \le \sqrt{2} \Psi(t, x)$
  for $t/2 \le s \le t$ and Lemma~\ref{Lem:iintGGf-d}, we obtain
  \begin{align*}
    I_1 \le C \Psi(t, x) \Psi(t, x') \Psi(t, y) \Psi(t, y) \frac{e^{-2\mu_1 t}}{1 \wedge t^d}
    e^{-\frac{2c_1}{3} \frac{|x-y|^2 + |x'-y'|^2}{t}} \int_{t/2}^t \left(1 \wedge \frac{(t-s)s}{t}\right)^{-\beta/2} \ud s.
  \end{align*}
  For $I_2$, we use the bounds
  \begin{align*}
    \Psi(t-s, z)\Psi(s, y) e^{-\frac{c_1}{6}\frac{|z-y|^2}{s}}     & \le C_0 \Psi(t-s, y), \\
    \Psi(t-s, z')\Psi(s, y') e^{-\frac{c_1}{6}\frac{|z'-y'|^2}{s}} & \le C_0 \Psi(t-s, y'),
  \end{align*}
  and $\Psi(s, z) \Psi(s, z') \le 1$ to get that
  \begin{align*}
    I_2 &\le C_0^2 \int_0^{t/2} \ud s\, \Psi(t-s, x)\Psi(t-s, x') \Psi(t-s, y) \Psi(t-s, y')\\
        & \quad \times \iint_{\dom^2} \ud z \, \ud z' \, \frac{e^{-2\mu_1(t-s)}}{1 \wedge (t-s)^d} e^{-c_1 \frac{|x-z|^2+|x'-z'|^2}{t-s}} f(z-z') \frac{e^{-2\mu_1 s}}{1 \wedge s^d} e^{-c_1 \frac{|z-y|^2+|z'-y'|^2}{2s}}.
  \end{align*}
  For $0 \le s \le t/2$, we apply the identity~\eqref{E:exp-id} with $t$ and $s$
  replaced by $t' = t+s$ and $s' = 2s$ respectively, and with $v = z-x$ and $w =
  y-x$. This gives
  \begin{align*}
    e^{-c_1 \frac{|x-z|^2+|x'-z'|^2}{t-s}}e^{-c_1 \frac{|z-y|^2+|z'-y'|^2}{2s}}
  & = e^{-c_1 \frac{|x-y|^2 + |x'-y'|^2}{t+s}} e^{-c_1 \frac{|(z-x) - \frac{t'-s'}{t'}(y-x)|^2 + |(z'-x') - \frac{t'-s'}{t'}(y'-x')|^2}{\tau'}} \\
  & \le e^{-\frac{2c_1}{3} \frac{|x-y|^2 + |x'-y'|^2}{t}} e^{-c_1 \frac{|(z-x) - \frac{t'-s'}{t'}(y-x)|^2 + |(z'-x') - \frac{t'-s'}{t'}(y'-x')|^2}{\tau'}},
  \end{align*}
  where $\tau' = (t'-s')s'/t'$.
  It is easy to see that for $0 \le s \le t/2$,
  \begin{align*}
    (1 \wedge \tau')^{d-\beta/2}\le 2^d \left(1\wedge \frac{(t-s)s}{t}\right)^{d-\beta/2}.
  \end{align*}
  Hence, we can follow the proof of Lemma~\ref{Lem:iintGGf-d}(i) and use the
  bound $\Psi(t-s, x) \le \sqrt{2} \Psi(t, x)$ for $0 \le s \le t/2$ to deduce
  that
  \begin{align*}
    I_2 \le C' \Psi(t, x) \Psi(t, x') \Psi(t, y) \Psi(t, y) \frac{e^{-2\mu_1 t}}{1 \wedge t^d}
    e^{-\frac{2c_1}{3} \frac{|x-y|^2 + |x'-y'|^2}{t}} \int_0^{t/2} \left(1 \wedge \frac{(t-s)s}{t}\right)^{-\beta/2} \ud s.
  \end{align*}
  The proof of Lemma~\ref{Lem:G^2} is complete.
\end{proof}

We can now strengthen the bounds~\eqref{K-d-ub} and~\eqref{K-d-lb} for
$\mathcal{K}^\lambda_D(t, x, x', y, y')$ in the case of bounded $C^{1,
\alpha}$-domains. Note that $\Psi(t, x) = 0$ for all $x \in \partial \dom$,
hence~\eqref{K-d-ub2} and~\eqref{K-d-lb2} below provide more precise estimates
than~\eqref{K-d-ub} and~\eqref{K-d-lb} for $x, x', y, y'$ that are close to
$\partial \dom$.

\begin{proposition}\label{P:K-C1alpha}
  If $\dom$ is a bounded $C^{1, \alpha}$-domain for some $\alpha > 0$. Then:
  \begin{enumerate}
  \item[(i)] There exist positive finite constants $C$, $c$, $c'$ such that for
    all $t > 0$ and $x, x', y, y' \in \dom$,
    \begin{align}\label{K-d-ub2}
      \begin{split}
        \mathcal{K}^\lambda_D(t, x, x', y, y')  \le
        & \frac{C\lambda^2}{1 \wedge t^d} \Psi(t, x) \Psi(t, x') \Psi(t, y) \Psi(t, y') \\
        & \times e^{-\frac{2c_1}{3}\left(\frac{|x-y|^2}{t} + \frac{|x'-y'|^2}{t}\right)}
                 e^{2t\left(c \lambda^2 + c' \lambda^{\frac{4}{2-\beta}}-\mu_1\right)}.
      \end{split}
    \end{align}
  \item[(ii)] There exist positive finite constants $\bar C$, $\bar c$, $\tilde
    c$ such that for all $t > 0$, for all $x, x', y, y' \in \dom$,
  \begin{align}\label{K-d-lb2}
    \begin{split}
      \mathcal{K}^\lambda_D(t, x, x', y, y') \ge
      & \frac{\bar C \lambda^2}{1 \wedge t^d} \Psi(t, x) \Psi(t, x') \Psi(t, y) \Psi(t, y') \\
      & \times e^{-16c_2 \frac{|x-x'|^2}{t}} e^{-12c_2 \left(\frac{|x-y|^2}{t} + \frac{|x'-y'|^2}{t} \right)} e^{2t\left(\bar c\lambda^2 + \tilde c \lambda^{\frac{4}{2-\beta} - \mu_1}\right)}.
    \end{split}
  \end{align}
  \end{enumerate}
\end{proposition}

\begin{proof}
  (i). Similarly to the proof of Proposition~\ref{P:K-d}(i), the upper
  bound~\eqref{K-d-ub2} can be proved by applying Lemmas~\ref{Lem:G^2},
  \ref{Lem:tildeh} and~\ref{Lem:I}.

  (ii). To prove the lower bound~\eqref{K-d-lb2}, we claim that for each $n \ge
  0$,
  \begin{align}\label{E:Gn-d-lb}
    \begin{split}
        \widetilde{G}^{\triangleright (n+1)}_D(t, x, x', y, y') & \ge
      C^{n+1} \Psi(t, x) \Psi(t, x') \Psi(t, y) \Psi(t, y') \\
      & \quad \times \frac{e^{-2\mu_1 t}}{1 \wedge t^d} e^{-16c_2\frac{|x-x'|^2}{t}} e^{-12c_2 \frac{|x-y|^2+|x'-y'|^2}{t}}
        \sum_{k=0}^n \frac{t^{n-k\beta/2}}{(n-k)!(k!)^{1-\beta/2}}.
    \end{split}
  \end{align}
  Indeed, for $n = 0$, \eqref{E:Gn-d-lb} follows from Proposition~\ref{P:G}(i).
  For $n \ge 1$, we can prove~\eqref{E:Gn-d-lb} by modifying the proof in
  Proposition~\ref{P:K-d}(ii) and we outline the major changes as follows.
  Instead of $V(x)$ defined in~\eqref{E:V(x)}, we now consider
  \begin{align*}
    \widetilde{V}(x) =
    \begin{cases}
      \overline{\mathscr{C}(x+(\eps_0/4)\xi_{y_i}, \xi_{y_i}, \eps_0/4)} & \text{in case (1) of~\eqref{E:V(x)}}, \\
      \overline{B(x, \eps_0/4)}                                          & \text{in case (2) of~\eqref{E:V(x)}},
    \end{cases}
  \end{align*}
  so that $\widetilde{V}(x) \subset V(x) \subset \dom$ and there exists
  $\delta_0 > 0$ such that for all $x \in \dom$, for all $z \in
  \widetilde{V}(x)$, $\operatorname{dist}(z, \partial \dom) \ge \delta_0$.
  Hence, \eqref{E:Phi-bd} implies that
  \begin{align}\label{E:inf-Phi}
    \inf_{x \in \dom} \inf_{z \in \widetilde{V}(x)}\Phi_1(z) \ge c_0^{-1}\delta_0 > 0.
  \end{align}
  For any $x \in \dom$, define $\tilde{x}$ by
  \begin{align*}
    \tilde{x} =
    \begin{cases}
      x+(\eps_0/4)\xi_{y_i} & \text{in case (1) of }\eqref{E:V(x)}, \\
      x & \text{in case (2) of }\eqref{E:V(x)}.
    \end{cases}
  \end{align*}
  Then, as in Lemma \ref{Lem:volA}, we can show that there exists $C_0 > 0$ such
  that for all $t > 0$, for all $x \in \dom$, for all $0 \le r \le s \le
  \eps_0/4$, for all $z \in \widetilde{V}(x) \cap B(\tilde{x}, s)$,
  \begin{align}\label{E:vol-Vtilde}
    \operatorname{Vol}(\widetilde{V}(x) \cap B(\tilde{x}, s) \cap B(z, r)) \ge C_0 (1 \wedge r)^d.
  \end{align}
  For $n = 1$, similarly to the proof in Proposition~\ref{P:K-d}(ii), we have
  \begin{align*}
    \MoveEqLeft \widetilde{G}_D^{\triangleright 2}(t, x, x', y, y')\\
    & \ge \int_{t/4}^{t/2} \ud s \iint_{[\widetilde{V}(x) \cap B(\tilde{x}, \sqrt{t})]^2} \ud z\, \ud z'\,
      \widetilde{G}_D(t-s, x, x', z, z')f(z-z') \widetilde{G}_D(s, z, z', y, y')\\
    & \ge C \Psi(t, x) \Psi(t, x') \Psi(t, y) \Psi(t, y') \frac{e^{-2\mu_1 t}}{(1\wedge t^d)^2}
      e^{-16 c_2 \frac{|x-x'|^2}{t}} e^{-12 c_2 \frac{|x-y|^2 + |x'-y'|^2}{t}}\\
    & \quad \times \int_{t/4}^{t/2} \ud s \iint_{[\widetilde{V}(x) \cap B(\tilde{x}, \sqrt{t})]^2} \ud z\, \ud z' \, (1 \wedge \Phi_1(z))^2 (1 \wedge \Phi_1(z'))^2 |z-z'|^{-\beta}.
  \end{align*}
  Then, by using~\eqref{E:inf-Phi} and~\eqref{E:vol-Vtilde}, we
  obtain~\eqref{E:Gn-d-lb} for $n=1$.

  For $n \ge 2$, we first restrict the integrals of $s_1, s_2, s_3, \dots, s_n$
  to the smaller intervals   $\left[(1-\frac{1}{4n})\frac t 2, \frac t
  2\right]$, $\left[(1-\frac{1}{2n})\frac t 2, s_1-\frac{t}{8n}\right]$,
  $\left[(1-\frac{2}{2n})\frac{t}{2}, s_2 - (s_1-s_2) \right]$, $\dots$,
  $\left[(1-\frac{n-1}{2n})\frac t 2, s_{n-1}-(s_{n-1}-s_{n-1})\right]$,
  respectively, so that $\frac{t-s_1}{5n} \le \frac{t}{8n} \le s_1-s_2 \le
  s_2-s_3 \le \dots \le s_{n-1}-s_n$. We can then modify the proof of
  Proposition~\ref{P:K-d}(ii) by considering
  \begin{alignat*}{2}
    z_1, z_1' & \in A_1  &  & \coloneqq \widetilde{V}(x) \cap B(\tilde{x}, \sqrt{(t-s_1)/(5n)} \wedge (\eps_0/4)),                                                       \\
    z_2       & \in A_2  &  & \coloneqq \widetilde{V}(x) \cap B(\tilde{x}, \sqrt{s_1-s_2} \wedge (\eps_0/4)) \cap B(z_1 , \sqrt{s_1-s_2} \wedge (\eps_0/4)),             \\
    z_2'      & \in A_2' &  & \coloneqq \widetilde{V}(x) \cap B(\tilde{x}, \sqrt{s_1-s_2} \wedge (\eps_0/4)) \cap B(z_1', \sqrt{s_1-s_2} \wedge (\eps_0/4)),             \\
              &          &  & \vdotswithin{\coloneqq}                                                                                                                    \\
    z_n       & \in A_n  &  & \coloneqq \widetilde{V}(x) \cap B(\tilde{x}, \sqrt{s_{n-1}-s_n} \wedge (\eps_0/4)) \cap B(z_{n-1} , \sqrt{s_{n-1}-s_n} \wedge (\eps_0/4)), \\
    z_n'      & \in A_n' &  & \coloneqq \widetilde{V}(x) \cap B(\tilde{x}, \sqrt{s_{n-1}-s_n} \wedge (\eps_0/4)) \cap B(z_{n-1}', \sqrt{s_{n-1}-s_n} \wedge (\eps_0/4)).
  \end{alignat*}
  Then, along the lines of the proof of Proposition~\ref{P:K-d}(ii), we can
  deduce~\eqref{E:Gn-d-lb} using~\eqref{E:inf-Phi} and~\eqref{E:vol-Vtilde}
  above. 
  
  Finally, recall that
  \begin{align*}
    \mathcal{K}^\lambda_D(t, x, x', y, y')
    = \lambda^2 \sum_{n=0}^\infty \lambda^{2n} \widetilde{G}_D^{\triangleright (n+1)}(t, x, x', y, y').
  \end{align*}
  By using~\eqref{E:Gn-d-lb}, interchanging the order of summation, and
  using~\eqref{E:exp-bd}, we can obtain the lower bound~\eqref{K-d-lb2} as in
  the proof of Proposition~\ref{P:K-d}(ii). This completes the proof of
  Proposition~\ref{P:K-C1alpha}.
 \end{proof}

\subsection{Proof of Theorem~\ref{T:C1alpha}}

\begin{proof}[Proof of Theorem~\ref{T:C1alpha}]
  In Remark~\ref{R:JcFinite}, we have seen that $\nu$ satisfies
  condition~\eqref{E:C1alpha-nu} if and only if $J_{2c_1/3}^* (t, x) < \infty$
  for all $t > 0$ and $x \in \dom$. The proof of the existence, $p$-th moment
  bounds and uniqueness of the solution is similar to the proof of
  Theorem~\ref{T:ExUnque} in Section~\ref{S:Key}, with the use of
  Lemma~\ref{Lem:G^2} instead of Lemma~\ref{Lem:iintGGf-d}~(i). The correlation
  bounds follow immediately from Propositions~\ref{P:2p-corr}
  and~\ref{P:K-C1alpha}.
\end{proof}

\subsection{Proof of Corollary~\ref{C:prod-dom}}

\begin{proof}[Proof of Corollary~\ref{C:prod-dom}]
  Note that the Dirichlet kernel kernel for $\mathscr{L}$ on $\dom = \prod_{i =
  1}^m \dom_i$ is given by
  \begin{align*}
    G_D^\dom(t, x, y) = \prod_{i = 1}^m G_D^{\dom_i}(t, x_i, y_i)
  \end{align*}
  for $t > 0$ and $x = (x_1, \dots, x_m)$, $y = (y_1, \dots, y_m) \in \dom$,
  where $G_D^{\dom_i}(t, x_i, y_i)$ is the Dirichlet heat kernel for
  $\mathscr{L}_i$ on the $C^{1, \alpha_i}$-domain $\dom_i$. This and
  Proposition~\ref{P:G} applied to each $G_D^{\dom_i}(t, x_i, y_i)$ imply the
  following heat kernel estimate
  \begin{align*}
    G_D^\dom(t, x, y) \le C_1 \Psi^* (t, x)\Psi^* (t, y) \frac{e^{-\mu_1 t}}{1 \wedge t^{d/2}}
    e^{-c_1 \frac{|x-y|^2}{t}}
  \end{align*}
  with suitable constants $C_1$ and $c_1$, where $\Psi^*$ is as defined
  in~\eqref{E:Psi-prod} and $\mu_1$ is the sum of the leading eigenvalues of the
  Dirichlet operators $\mathscr{L}_i$. By the proof of~\eqref{Psi-bd} in
  Lemma~\ref{Lem:G^2}, for each $i$, we have
  \begin{align*}
    \left(1 \wedge \frac{\Phi_1^{\dom_i}(v_i)}{1 \wedge s^{1/2}}\right)
    \left(1 \wedge \frac{\Phi_1^{\dom_i}(w_i)}{1 \wedge r^{1/2}}\right)
    e^{-\frac{c_1}{6} \frac{|v_i-w_i|^2}{r}}
    \le C_{0, i} \left(1 \wedge \frac{\Phi_1^{\dom_i}(w_i)}{1 \wedge s^{1/2}}\right)
  \end{align*}
  for all $s, r > 0$ and $v_i, w_i \in \dom_i$, where $C_{0, i}$ is a constant.
  Taking products over $i \in \{1, \dots, m\}$ gives
  \begin{align*}
    \Psi^* (s, v) \Psi^* (r, w) e^{-\frac{c_1}{6} \frac{|v-w|^2}{r}} \le C_0 \Psi^* (s, w)
  \end{align*}
  for all $s, r > 0$ and $v, w \in \dom$, where $C_0 = \prod_{i = 1}^m C_{0,i}$.
  Then, the same proof of Lemma~\ref{Lem:G^2} shows that the
  estimate~\eqref{E:G^2} holds with every $\Psi$ replaced by $\Psi^*$. This
  implies that the estimate~\eqref{K-d-ub2} also holds with every $\Psi$
  replaced by $\Psi^*$. Using these estimates, the statements (i) and (ii) in
  Theorem~\ref{T:C1alpha} with $\Psi$ replaced by $\Psi^*$ can be proved for the
  product domain the same way they are proved in Theorem~\ref{T:C1alpha}.
  Condition~\eqref{E:int-prod-nu} ensures that $J^*_c(t, x) < \infty$ for all $t
  > 0$ and $x \in \dom$.
\end{proof}

\subsection{Some auxiliary results related to examples in Section~\ref{SS:Examples}}

\begin{lemma}\label{L:BesselJMax}
   For $\nu > -1$, let $J_\nu(\cdot)$ be the Bessel function of first kind and
   of order $\nu$ and let $z_0$ be any positive zero of $J_\nu(\cdot)$. Then,
   \begin{align*}
     \sup_{r\in (0,1)} \frac{r^{-\nu} J_\nu(r z_0)}{1-r}<\infty.
   \end{align*}
\end{lemma}
\begin{proof}
   Set $f(r) = \frac{r^{-\nu} J_\nu(r z_0)}{1-r}$, which is a continuous
   function on $(0,1)$. By Eq.~10.7.3 of~\cite{olver.lozier.ea:10:nist}, we see
   that $\lim_{r\to 0} f(r) = \Gamma(\nu)^{-1}
   \left(\frac{z_0}{2}\right)^\nu<\infty$. Because all zeros of $J_\nu(\cdot)$
   are simple (see Section 10.21 \textit{ibid}), we see that $\lim_{r \to 1}f(r)
   = -z_0 J_\nu'(z_0) <\infty$. Therefore, $f(r)$ is a continuous function on
   $[0,1]$, from which the desired result follows.
\end{proof}

\begin{remark}\label{R:EgienBall}
  In the setting of Example~\ref{Eg:ball}, the leading eigenfunction is given by
  \begin{align*}
    \Phi_1(x) = \frac{1}{C_d} |x|^{(2-d)/2} J_{(d-2)/2}\left(z_0|x|\right),
  \end{align*}
  where $z_0$ is the first positive zero of the Bessel function
  $J_{(d-2)/2}(x)$, and $C_d$ is some normalization constant. One may refer to
  \S 34.2 in Chapter III of~\cite{treves:75:basic} for the case $d = 2$ and
  Section H in Chapter 2 of~\cite{folland:95:introduction} for the general case
  $d \ge 3$. In particular, the leading eigenfunction $\Phi_1(x)$ corresponds to
  $F_k^{lm}(x)$ in Theorem 2.66 (\textit{ibid.}) with $k = 0$, $l = 1$ and $m = 1$.
  Here are a few comments:
  \begin{enumerate}
    \item The multiplicity for the leading eigenvalue $z_0$ is one since $d_k =
      1$; see Corollary 2.55 (\textit{ibid.}).
    \item $Y_0^1(x)$ in Theorem 2.66 (\textit{ibid.}) is a constant function.
    \item In Figure~\ref{SF:Psi_4d}, we have chosen the following normalization
      constant:
      \begin{align}\label{E:Cd}
        C_d =
        \lim_{r\to 0} r^{(2-d) /2} J_{(n-2) /2}(z_0 r) = 2^{-d/2} d\,
        \Gamma\left(1+d/2\right)^{-1} z_0^{(d-2) /2},
      \end{align}
      where the limit is due to~\cite[Eq. 10.7.3 on p.
      223]{olver.lozier.ea:10:nist}, so that $\max_{|x|\le 1}\Psi(1,x) = 1$.
  \end{enumerate}
\end{remark}

The following Lemma is used in Example~\ref{Eg:Ann}.

\begin{lemma}\label{L:simple0}
  $R_1$ and $R_2$ are two simple zeros for $Z(r)$ defined in Example~\ref{Eg:Ann}.
\end{lemma}
\begin{proof}
  It is straightforward to check that $R_1$ and $R_2$ are two zeros of $Z(r)$.
  To show that they are simple, one needs to prove that $Z'(r)\ne 0$ for $r =
  R_1$ and $R_2$. By Eq.~10.6.3 of~\cite{olver.lozier.ea:10:nist}, we see that
  \begin{align*}
    z_0^{-1} Z'(r)
    & =   J_0(R_1 z_0) Y_0'(r z_0)
        - J_0'(r z_0)  Y_0(R_1 z_0)  \\
    & = - J_0(R_1 z_0) Y_1(r z_0)
        + J_1(r z_0)  Y_0(R_1 z_0).
  \end{align*}
  By setting $r = R_1$ and applying Eq.~10.5.2 (\textit{ibid.}), we have that
  \begin{align*}
    z_0^{-1} Z'(R_1)
    & = - J_0(R_1 z_0) Y_1(R_1 z_0)
        + J_1(R_1 z_0) Y_0(R_1 z_0)
      = \frac{2}{\pi R_1 z_0} \ne 0,
  \end{align*}
  which proves that $R_1$ is a simple zero of $Z(r)$.

  Let $F(z)$ denote the function in~\eqref{E:ZeroJY}. By setting $r = R_2$, we
  have
  \begin{align}\label{E:Z'R2}
    z_0^{-1} Z'(R_2)
    & = - J_0(R_1 z_0) Y_1(R_2 z_0)
        + J_1(R_2 z_0) Y_0(R_1 z_0).
  \end{align}
  Because $J_0^2 (z) + Y_0^2 (z) > 0$ for all $z>0$ (see, e.g., Eq.~10.9.30
  \textit{ibid}), we see that if $Y_0(R_1z_0) = 0$, then $J_0(R_1z_0) \ne 0$.
  Because all zeros of $Y_\nu(\cdot)$ are simple, we see that $Y_1(R_1z_0) =
  -Y_0'(R_1z_0) \ne 0$. But since $z_0$ is a zero of $F(z)$, i.e.,
  \begin{align*}
    0 & = J_0(R_1z_0) Y_0(R_2z_0) -
          J_0(R_2z_0) Y_0(R_1z_0)
        = J_0(R_1z_0) Y_0(R_2z_0).
  \end{align*}
  Hence, $Y_0(R_2z_0) = 0$, which further implies that both $J_0(R_2z_0)\ne 0$
  and $Y_1(R_2z_0) = - Y_0'(R_2z_0) \ne 0$. This proves Case 1 in
  Table~\ref{ST:Case1}. Applying the same arguments with $Y_0(R_1z_0) = 0$
  replaced by $Y_0(R_2z_0) = 0$, $J_0(R_1z_0) = 0$, and $J_0(R_2z_0) = 0$, we
  see that only three cases can happen, which are listed in the following
  Table~\ref{T:3Cases}:

  \begin{table}[htpb]
    \renewcommand{\thesubtable}{\thetable.\arabic{subtable}}
    \makeatletter
      \renewcommand{\p@subtable}{}
    \makeatother
    \captionsetup[subtable]{style = default, margin = 0pt, parskip = 0pt, hangindent = 0pt, labelformat = simple}

    \centering
    \caption{Three cases in the proof of Lemma~\ref{L:simple0}}\label{T:3Cases}
    \renewcommand{\arraystretch}{1.2}

    \subfloat[][Case 1.]{\label{ST:Case1}
    \centering
    \begin{tabular}{|c|c|c|c|c|}
      \hline \rowcolor{lightgray}
                & $J_0$   & $J_1$ & $Y_0$ & $Y_1$   \\ \hline
      $R_1 z_0$ & $\ne 0$ & ---   & $= 0$ & $\ne 0$ \\
      $R_2 z_0$ & $\ne 0$ & ---   & $= 0$ & $\ne 0$ \\ \hline
    \end{tabular}}
    \quad
    \subfloat[][Case 2.]{\label{ST:Case2}
    \centering
    \begin{tabular}{|c|c|c|c|c|}
      \hline \rowcolor{lightgray}
                & $J_0$ & $J_1$   & $Y_0$   & $Y_1$ \\ \hline
      $R_1 z_0$ & $= 0$ & $\ne 0$ & $\ne 0$ & ---   \\
      $R_2 z_0$ & $= 0$ & $\ne 0$ & $\ne 0$ & ---   \\ \hline
    \end{tabular}}
    \quad
    \subfloat[][Case 3.]{\label{ST:Case3}
    \centering
    \begin{tabular}{|c|c|c|c|c|}
      \hline \rowcolor{lightgray}
                & $J_0$   & $J_1$ & $Y_0$   & $Y_1$ \\ \hline
      $R_1 z_0$ & $\ne 0$ & ---   & $\ne 0$ & ---   \\
      $R_2 z_0$ & $\ne 0$ & ---   & $\ne 0$ & ---   \\ \hline
    \end{tabular}}
  \end{table}

  \noindent\textbf{Cases 1 \& 2}: From the expression of $Z'(R_2)$
  in~\eqref{E:Z'R2}, we see that in these two cases, $Z'(R_2) \ne 0$.

  \noindent\textbf{Case 3}: Since both $Y_0(R_1z_0)$ and $Y_0(R_2z_0)$ are
  nonzero, and since $F(z_0) = 0$, we see that
  \begin{align*}
    J_0(R_1 z_0) = \frac{Y_0(R_1 z_0) J_0(R_2 z_0)}{Y_0(R_2 z_0)}.
  \end{align*}
  This and Eq.~10.5.2 (\textit{ibid.}) imply that
  \begin{align*}
    z_0^{-1} Z'(R_2)
    & = \frac{Y_0(R_1z_0)}{Y_0(R_2z_0)}
    \left(-J_0(R_2z_0)Y_1(R_2z_0) + J_1(R_2z_0)Y_0(R_2z_0)\right)\\
    & = \frac{Y_0(R_1z_0)}{Y_0(R_2z_0)} \frac{2}{\pi R_2z_0} \ne 0.
  \end{align*}
  Therefore, combining the above three cases, we have proved that $R_2$ is a
  simple zero of $Z(r)$.
\end{proof}

\bigskip
\noindent
\textbf{Acknowledgments} C.~Y.~Lee is partially supported by Taiwan's National
Science and Technology Council grant NSTC111-2115-M-007-015-MY2.
L.~Chen is partially supported by NSF DMS-2246850 and a
  collaboration grant from Simons Foundation (959981). The authors would also
like to thank the anonymous referee for the careful reading and helpful
suggestions, which have greatly improved the quality of the paper.

% -------------------------------------------------
% References go here.
% -------------------------------------------------

% \global\let\savedifeof = \ifeof
% \newcommand\ifeof#1{\global\let\ifeof = \savedifeof\iftrue}%

\addcontentsline{toc}{section}{References}
\printbibliography[title = {References}]

\end{document}